\documentclass{article}
\usepackage[english]{babel}
\usepackage[T1]{fontenc}
\usepackage[utf8]{inputenc}
\RequirePackage{amsthm,amsmath,amsfonts,amssymb}
\RequirePackage[numbers]{natbib}

\newcommand{\ac}[1]{\textcolor{red}{add citation}}

\usepackage{enumerate}
\usepackage{color}
\usepackage{multirow}\usepackage{graphicx}
\usepackage{amsfonts}
\usepackage{amsmath}
\usepackage{amssymb}

\usepackage{color}
\usepackage{dsfont}

\usepackage{pgfplots}
\usepackage{subfigure}
\usetikzlibrary{patterns}
\usepackage{amssymb}
\usepackage{graphicx}
\usepackage{amsmath}
\usepackage{bm}
\usepackage{ltablex}
\usepackage{algorithm} 
\usepackage{algorithmic}  
\usepackage[algo2e]{algorithm2e} 
\usepackage{caption}

\usetikzlibrary{decorations.pathreplacing}

\usetikzlibrary{calc,patterns}

\usepackage{url}
\usepackage{fancyhdr}
\usepackage{indentfirst}

\usepackage{enumerate}
\usepackage{dsfont}
\usepackage{tabulary}
\usepackage{booktabs}
\usepackage[colorlinks=true,citecolor=blue]{hyperref}
\usepackage{amsthm}
\usepackage{color}

\usepackage{comment}
\usepackage{tikz-cd}
\usepackage{graphicx}
\usepackage{amsfonts}
\usepackage{longtable}
\usepackage{amsmath}
\usepackage{amssymb}
\usepackage{url}
\usepackage{bbm}
\usepackage{fancyhdr}
\usepackage{indentfirst}
\usepackage{enumerate}
\usepackage{dsfont}
\pgfplotsset{compat=1.7}
\usepackage[colorlinks=true,citecolor=blue]{hyperref}
\usepackage{cleveref}
\usepackage{amsthm}
\usepackage{color}
\renewcommand{\cite}{\citet}
\usepackage{comment}
\usepackage[margin=0.8in]{geometry}

\usepackage{algorithm}

\renewcommand{\d}{\,\mathrm{d}}
\newcommand{\dd}{\overset{\mathrm{law}}{=}}
\newcommand{\p}{\mathbb{P}}
\newcommand{\q}{\mathbb{Q}}

\usepackage{xcolor}

\newcommand{\E}{\mathbb{E}}    
\newcommand{\R}{\mathbb{R}}    
\newcommand{\N}{\mathbb{N}}    

\theoremstyle{plain}
\newtheorem{theorem}{Theorem}
\newtheorem{corollary}[theorem]{Corollary}
\newtheorem{lemma}[theorem]{Lemma}
\newtheorem{proposition}[theorem]{Proposition}

\theoremstyle{definition}

\newtheorem{conjecture}{Conjecture}

\theoremstyle{remark}
\newtheorem{remark}{Remark}

\usepackage{tikz}

\usetikzlibrary{arrows.meta,positioning}

\newcommand{\ee}{\varepsilon}

\newcommand{\n}[1]{\left\lVert#1\right\rVert}

\newcommand{\bz}{{\mathbf{z}}}
\newcommand{\bxi}{{\boldsymbol{\xi}}}

\newcommand{\bzeta}{{\boldsymbol{\zeta}}}

\newcommand{\wx}{\widetilde{x}}
\newcommand{\wtx}{\widetilde{t}_x}

\newcommand{\bv}{\mathbf{v}}

\newcommand{\bu}{\mathbf{u}}

\newcommand{\bS}{\mathbb{S}}

\newcommand{\bx}{\mathbf{x}}
\newcommand{\by}{\mathbf{y}}
\newcommand{\bX}{\mathbf{X}}

\newcommand{\bl}{c_2}
\newcommand{\bc}{\mathbf{c}}
\newcommand{\bla}{\boldsymbol\lambda}

\newcommand{\bet}{\boldsymbol{\eta}}

\def\d{\mathrm{d}}

\newcommand{\bone}{ {\mathbbm{1}} }
\renewcommand{\S}{\mathbb{S}}
\renewcommand{\bS}{\mathbf S}

\newcommand{\sA}{\mathscr {A}}
\newcommand{\sB}{\mathscr {B}}
\newcommand{\sC}{\mathscr {C}}
\newcommand{\sD}{\mathscr {D}}
\newcommand{\sE}{\mathscr {E}}
\newcommand{\sH}{\mathscr {H}}
\newcommand{\sI}{\mathscr {I}}
\newcommand{\sJ}{\mathscr {J}}

\newcommand{\sK}{\mathscr {K}}

\newcommand{\sG}{\mathscr {G}}
\newcommand{\sF}{\mathscr {F}}
\newcommand{\cF}{\mathcal {F}}

\renewcommand{\H}{\mathbb{H}}

\usepackage{mathrsfs}

\newcommand{\be}{\mathbf{e}}
\newcommand{\z}{\mathbf{0}}

\newcommand{\lst}{\preceq_{\mathrm{st}}}
\newcommand{\gst}{\succeq_{\mathrm{st}}}

\renewcommand{\geq}{\geqslant}
\renewcommand{\leq}{\leqslant}
\renewcommand{\epsilon}{\varepsilon}

\newcommand{\bZ}{{\mathbb{Z}}}

\title{Tightness Analysis of First Passage Times of $d$-Dimensional Branching Random Walk}

\author{Jose Blanchet\thanks{Department of Management Science and Engineering, Stanford University. Email: jose.blanchet@stanford.edu} \and Zhenyuan Zhang\thanks{Department of Mathematics, Stanford University. Email: zzy@stanford.edu}}
\begin{document}
\maketitle

\begin{abstract}
Given a discrete-time non-lattice supercritical branching random walk in $\mathbb{R}^d$, we investigate its first passage time to a shifted unit ball of a distance $x$ from the origin, conditioned upon survival. We provide precise asymptotics up to $O(1)$ (tightness) for the first passage time as a function of $x$ as $x\to\infty$, thus resolving a conjecture in Blanchet--Cai--Mohanty--Zhang (2024). Our proof builds on the previous analysis of Blanchet--Cai--Mohanty--Zhang (2024) and employs a careful multi-scale analysis on the genealogy of particles within a distance of $\asymp \log x$  near extrema of a one-dimensional branching random walk, where the cluster structure plays a crucial role.
\end{abstract}


\tableofcontents

\section{Introduction and main contribution}
We examine the first passage times of discrete-time non-lattice branching random walks in $\R^d$. In our setting, a \textit{branching random walk} (BRW) is initiated by a single particle located at the origin $\z\in\R^d$ at time $n=0$. At each time $n+1$, each particle at time $n$ dies and independently reproduces its descendants according to some probability law on the non-negative integers $\N_0$ with a mean greater than one. Each descendant then independently performs a random walk step, collectively forming the set of particles at time $n+1$. In particular, the underlying genealogy of the particles resembles a supercritical Galton--Watson process.
The  \textit{first passage time} (FPT) is by definition the first time some particle is present in a prescribed subset of $\R^d$. We refer readers to a more formal definition of BRW in Section 2.2 of \citep{zeitouni2016branching}. In this paper, we focus on the first passage times to $B_x$, the unit ball centered at $(x,0,\dots,0)\in\R^d$. 

In recent years, numerous studies have focused on BRW in dimension one. Let $M_n$ denote the maximum position of the particles in generation $n$ (or equivalently, at time $n$). The precise asymptotics of $M_n$ and the limit behavior near the frontier have been well-studied by \citep{addario2009minima,aidekon2013convergence,bramson2016convergence,bramson2009tightness,hu2016big,hu2009minimal,madaule2017convergence}, along with the references therein. In particular, under mild conditions, it is shown that there exist constants $c_1,c_2>0$ such that $M_n=c_1n-\frac{3}{2c_2}\log n+O(1)$, where the $O(1)$ term converges in law to a randomly-shifted Gumbel distribution. We refer to \citep{arguin2016extrema,shi2015branching,zeitouni2016branching} for notes on BRW and related topics, and Section \ref{sec:one-dim results} for some selected results useful for our purpose.
In terms of FPT, the only works we are aware of are \citep{buraczewski2019large,jelenkovic2015maximums}, which characterized the law of large numbers and large deviation behavior for the FPT of a one-dimensional BRW with a negative drift.

On the other hand, multi-dimensional BRW have been less studied. However, some of this body of work is reported in \citep{bezborodov2023maximal,blanchet2024first,uchiyama1982spatial,zhang2024large}. In particular, \citep{bezborodov2023maximal} investigated the asymptotic behavior of the maximum distance from the origin for a spherically-symmetric BRW in $\R^d$ and established a precise asymptotic (up to an $O(1)$ factor). In this direction, we also mention recent but earlier studies on the maximum norm of branching Brownian motion (BBM) in $\R^d$ by \citep{berestycki2024extremal,kim2023maximum,kim2024shape,mallein2015maximal}.

Spatial branching processes, including both BRW and BBM, have a wide spectrum of applications ranging from ecology to modeling epidemics (\citep{fisher1937wave,kolmogorov1937etude,konig2020branching,kot2004stochasticity}, among many others).
More recently, motivated by problems from polymer physics, the work \citep{zhang2024modeling} initiated the study of the first passage times of spatial branching processes. In particular, it was noted therein that the precise asymptotic for the FPT of BBM follows from established results on multi-dimensional Fisher--KPP equations with boundary value conditions (\citep{ducrot2015large,gartner1982location,roquejoffre2019sharp}; see also Remark \ref{rem:gartner} below). The follow-up work \citep{blanchet2024first} further investigated the FPT of BRW and obtained partial results. For the spherically-symmetric case, they provide asymptotics up to $O(\log\log x)$ using a particle genealogy approach. For the general case, they achieve asymptotics up to $O(\log x)$ by analyzing random walks in cones. It was conjectured therein that the asymptotic should be precise up to a tight $O(1)$ factor, based on numerics and the following two pieces of theoretical evidence. First, the FPT is exponentially concentrated around its median (Theorem 2 of \citep{blanchet2024first}; see also Lemma \ref{lemma:concentration} below). Second, the analogous asymptotic holds for the BBM (Theorem 1 of \citep{zhang2024modeling}; see also Remark \ref{rem:gartner} below). 

In this paper, we fully resolve the aforementioned conjecture on the FPT of BRW in $\R^d$ by obtaining the precise asymptotic. Our most general result (Theorem \ref{thm:main2} below) does not require radial symmetry of the process --- while the main body of this work is devoted to the radially-symmetric case, we identify the necessary changes for the radially-asymmetric case in Section \ref{sec:non-sphere}. 
Our strategy builds in part on the particle genealogy approach in \citep{blanchet2024first} but requires new probabilistic ideas and a finer multi-scale investigation of the genealogy of the one-dimensional BRW. 
In the subsections below, we state precisely our main result and explain the main ideas underlying our proof. 

\subsection{Statement of the main results}
Consider a discrete-time BRW model with offspring distribution $\{p_i\}_{i\geq 0}$, whose mean is denoted by $\rho=\sum_i i\,p_i$ and we assume that $\rho>1$ (the supercritical case).
Recall that for $x\in\R$, $B_x$ denotes the ball of radius one centered at $(x,0,\dots,0)$ in $\R^d$. We let $V_n$ denote the collection (i.e.~set) of particles at time step $n$, and  $\{\bet_{v,n}(k)\}_{0\leq k\leq n}$ denote the $d$-dimensional random walk that leads to $v\in V_n$.
Denote by $\eta_{v,n}(k)\in\R$ (resp.~$\widehat{\bet}_{v,n}(k)\in\R^{d-1}$) the first coordinate (resp.~last $d-1$ coordinates) of ${\bet}_{v,n}(k)$. 
We define the FPT  $\tau_x$  of the BRW to $B_x$, that is,
$$\tau_x:=\min\{n\geq 0 : \exists\, v\in V_n,~ \bet_{v,n}(n)\in B_x\}.$$
Let $\bxi$ be an $\R^d$-valued random variable representing the increment distribution of the BRW. 
Denote the first coordinate of $\bxi$ by $\xi$, which is a real-valued random variable. We  introduce the large deviation rate function
\begin{align}
    I(x):=\sup_{\lambda>0}\Big(\lambda x-\log\phi_\xi(\lambda)\Big),\label{eq:ratef}
\end{align} 
where $\phi_\xi(\lambda):=\E[e^{\lambda\xi}]$ is the moment generating function for $\xi$. Consider the following assumptions:\footnote{ These assumptions are the same as (A1)--(A4) in \citep{blanchet2024first}.}
\begin{itemize}
     \item [(A1)] the offspring distribution has a finite second moment, i.e., $\sum_j j^2\,p_j<\infty$;\label{A1}
     \item [(A2)] the law of $\bxi$ is integrable and centered, i.e., $\E[\bxi]=\z$;
      \item [(A3)] the law of $\bxi$ is spherically-symmetric in $\R^d$,\footnote{This means that the law of the jump $\bxi$ is invariant under any orthonormal transformation in $\R^d$.} and $\p(\bxi=\z)<1$;
    \item [(A4)] $\log\rho\in(\mathrm{ran}I)^\circ$, where $(\mathrm{ran}I)^\circ$ is the interior of the range of $I$.  In other words, there exists $c_1 > 0$ such that $I(c_1) = \log \rho$.  It can be shown that $c_1\in (\mathrm{ran}(\log\phi_\xi)')^\circ$. Let $c_2=I'(c_1)$. (The constants $c_1, c_2$ arise naturally from the large deviations analysis of the random walk generated by the first-coordinate increment distribution.)
 \end{itemize}


\begin{theorem}\label{thm:main}
     Assume (A1)--(A4). Conditioned upon survival, the first passage time for BRW in dimension $d$ to $B_x$ satisfies
\begin{align}
    \tau_x=\frac{x}{{c}_1}+\frac{d+2}{2c_1\bl}\,\log x+O_\p(1),\label{eq:taux asymp}
\end{align}where the $O_\p(1)$ is tight. 
\end{theorem}

As a consequence of Theorem \ref{thm:main}, the first passage times are tight around the deterministic asymptotics $\frac{x}{{c}_1}+\frac{d+2}{2c_1\bl}\,\log x$. {{A modified second‑moment method could potentially yield tightness but loses the genealogy needed to identify inter‑cluster independence and the decorations surrounding frontier particles. This type of analysis is key in the convergence in law results of the re-centered one-dimensional maxima of spatial branching processes in the work of, for example, \citep{aidekon2013convergence,aidekon2013branching}, which describe the decorations. These results and observations naturally invite the question:}} does the aforementioned $O_\p(1)$ term converge in law? If the answer is affirmative, how can we characterize the limit law? Typically, answering these types of questions requires a detailed analysis of the cluster structure formed by the extremal particles (as we do here). Thus, the results in this paper may pave the way towards these answers through a careful multi-scale analysis of those clusters. We leave these challenging questions for future investigation.

Let us now proceed to the non-spherically-symmetric case. Let
\begin{align}
    \widehat{I}(\bx):=\sup_{\bla\in\R^d}\Big(\bla\cdot\bx-\log\phi_\bxi(\bla)\Big)=\sup_{\bla\in\R^d}\Big(\bla\cdot\bx-\log\E[e^{\bla\cdot\bxi}]\Big)\label{eq:I long}
\end{align}
denote the large deviation rate function for $\bxi$. We impose the following assumptions:
\begin{itemize}
    \item [(A5)] the law of $\bxi$ is non-lattice in the sense that for all $\bx\in\R^{d}\setminus\{\z\}$, $|\E[e^{\mathrm{i}\bx\cdot\bxi}]|< 1$; 
    \item [(A6)] $\log\rho\in(\mathrm{ran}\widehat{I}(\cdot,\z))^\circ$, where $\widehat{I}(\cdot,\z)$ refers to the function $\widehat{I}$ with the last $d-1$ variables fixed at zero; let $\widehat{c}_1>0$ satisfy $\widehat{I}(\widehat{c}_1,\z)=\log\rho$, it holds $(\widehat{c}_1,\z)\in (\mathrm{ran}\nabla\log\phi_\bxi)^\circ$.
   \end{itemize}   
   Denote by $\bc_2=\nabla \widehat{I}((\widehat{c}_1,\z))$, which is the value of $\bla$ where the supremum \eqref{eq:I long} is attained at $\bx=(\widehat{c}_1,\z)$. 
The assumption (A5) is crucial for the non-degeneracy of the dimension $d$ of the jumps.
More precisely, (A5) implies that any projection of $\bxi$ (for instance, $\bxi\cdot\bc_2$) must also be non-lattice. 
Both conditions are useful for guaranteeing that the BRW reaches the target ball $B_x$. Another consequence of (A6) is that $\phi_\bxi$ is well-defined in a neighborhood of $\bc_2$.

\begin{theorem}\label{thm:main2}
     Assume (A1), (A2), (A5), and (A6). Conditioned upon survival, the first passage time for BRW in dimension $d$ to $B_x$ satisfies
\begin{align}
    \tau_x=\frac{x}{\widehat{c}_1}+\frac{d+2}{2\,\widehat{c}_1\be_1\cdot\bc_2}\,\log x+O_\p(1),\label{eq:fpt2}
\end{align}
where $\be_1:=(1,0,\dots,0)\in\R^d$ and the $O_\p(1)$ term is tight. 
\end{theorem}
The assumptions (A5) and (A6) are weaker than (A3) and (A4). Indeed, it is easy to see that  (A3) implies (A5) for $d\geq 2$, as a consequence of Lemma 24 of \citep{blanchet2024first}; Proposition 21 therein also shows that (A3) and (A4) together imply (A6). 
Consequently, Theorem \ref{thm:main2} is more general than Theorem \ref{thm:main}. However, in this paper, we will mainly focus on the proof of Theorem \ref{thm:main}. The proof of Theorem \ref{thm:main2} is mostly verbatim, where the major changes will be pointed out in Section \ref{sec:non-sphere}.

\begin{remark}\label{rem:gartner}
    There is a relatively simple way to infer the asymptotics \eqref{eq:taux asymp} and \eqref{eq:fpt2} from known results in the PDE literature, as noted in \citep{blanchet2024first,zhang2024modeling}. We briefly summarize the main arguments here for completeness. 
    
    First, it is well-known (see e.g., \citep{mallein2015maximal}) that if $v(t,\bx)$ denotes the probability that there exists a particle of BBM located in $B_\bx$ (the unit ball centered at $\bx\in\R^d$) at time $t$, then $v(t,\bx)$ solves the multi-dimensional Fisher-KPP equation 
\begin{align}
    \begin{cases}
        v_t=\frac{1}{2}\Delta v+v-v^2,~t>0,~\bx\in \R^d,\\
        v(0,\bx)=\bone_{\{\bx\in B_{\z}(1)\}}.
    \end{cases}\label{eq:fkpp2}
\end{align} 
The large-time behavior of \eqref{eq:fkpp2} has been well-studied in the literature. For instance, \citep{gartner1982location} showed that for every fixed $\ee>0$, elements in the set $\{(t,\bx):v(t,\bx)\in(\ee,1-\ee)\}$ satisfy
\begin{align}
    \n{\bx}= \sqrt{2}\,t-\frac{(d+2)\log t}{2\sqrt{2}}+O(1)\label{eq:gartner asymp}
\end{align}
uniformly as $t\to\infty$. See also \citep{roquejoffre2019sharp} for finer estimates. 
If $\bx=(x,\z)$, the probability $v(t,\bx)$ is almost equivalent to $\p(\tau_x\leq t)$, modulo certain events that can be absorbed into the $O(1)$ term (e.g., a particle may enter $B_x$ before $t$ but all its descendants conspiratorially stay outside $B_x$ at time $t$). From here, it is not hard to conclude that the right-hand side of \eqref{eq:gartner asymp} provides the correct asymptotics of $\tau_x$ for BBM; see Appendix A.1 of \citep{zhang2024modeling} for further details. 

\sloppy Next, observe that inserting $c_1=c_2=\sqrt{2}$ in the asymptotics $M_n=c_1n-\frac{3}{2c_2}\log n+O_\p(1)$ for the maximum of BRW yields the asymptotics $M_t=\sqrt{2}t-(3\log t)/(2\sqrt{2})+O_\p(1)$ of the maximum $M_t$ of BBM at time $t$. Along with the right-hand side of \eqref{eq:gartner asymp}, this heuristically leads to \eqref{eq:taux asymp}. 

To infer the form of \eqref{eq:fpt2}, one way is to assume that the jump law $\bxi$ satisfies $\bxi\dd T(\bzeta)$ for some invertible transformation $T$ and a spherically symmetric law $\bzeta$. The first passage problem then reduces to the first passage problem to $T^{-1}(B_x)$ (which is equivalent to $B_{T^{-1}(\bx)}$ without sacrificing the asymptotics) for the BRW with jump law $\bzeta$. A standard analysis of the large deviation rate function $\widehat{I}$ under a linear transform then leads to \eqref{eq:fpt2}. We refer to Section 3.2 of \citep{blanchet2024first} for details.

\end{remark}

\begin{remark}
    Our first passage time is defined in terms of a shifted ball of \textit{radius one}. We expect that the same proof techniques apply to targets of varying shapes. We refer to Section \ref{sec:targets} for a more detailed discussion. 
\end{remark}

\begin{remark}
    We also expect that the same technique applies to the branching Brownian motion, and hence leads to a probabilistic proof of the asymptotic behavior of the multi-dimensional Fisher--KPP equation with boundary value conditions (as a complement of the works \citep{ducrot2015large,gartner1982location,roquejoffre2019sharp}). Compared to the probabilistic proof of \citep{gartner1982location}, our approach does not require that the underlying process is Gaussian.
\end{remark}

Theorem \ref{thm:main2} also has an immediate consequence on the tightness of the maximum along the first coordinate, among particles whose locations in the remaining $d-1$ dimensions lie within unit distance from the origin. We formulate it as the next corollary, whose proof is deferred until Section \ref{sec:proof of coro}.

\begin{corollary}\label{coro}
   Assume (A1), (A2), (A5), and (A6).  Define 
    $$\widetilde{M}_n:=\max\{\eta_{v,n}(n):v\in V_n,\,\|\widehat{\bet}_{v,n}(n)\|\leq 1\}.$$
    Then conditioned upon survival, 
    \begin{align}
        \widetilde{M}_n=\widehat{c}_1n-\frac{d+2}{2\,\be_1\cdot\bc_2}\,\log n+O_\p(1),\label{eq:wMn}
    \end{align}
    where the $O_\p(1)$ term is tight. 
\end{corollary}

The proofs of Theorems \ref{thm:main} and \ref{thm:main2} are based on a novel result (Proposition \ref{lemma:transition} below) on the genealogical structure of the frontier particles for the one-dimensional BRW. Roughly speaking, it states that particles located beyond level $M_n-\log n$ at time $n$ are likely separated either very early (in time $[0,O((\log n)^2)]$) or very late (in time $[n-O((\log n)^2),n]$) in the underlying genealogy. The result extends previous works of \citep{arguin2011genealogy,mallein2016asymptotic} and is easily generalizable --- $\log n$ can be replaced by functions that are $o(\sqrt{n})$. In the next section, we demonstrate the relevance of this result in our analysis.

\subsection{Outline of the proof}
\label{sec:outline proof}
In the following, we write $A\ll B$ if there exists a constant $C>0$ possibly depending on the law of the BRW such that $A\leq CB$, and $A\asymp B$ if $A\ll B\ll A$.

\textbf{Setup and main intuition.}
 Let us assume that the BRW is spherically-symmetric, and we condition upon survival. 
In the spherically-symmetric case, let us define
\begin{align}
    t_x:=\frac{x}{c_1}+\frac{d+2}{2\bl c_1}\,\log x=\Big(\frac{x}{c_1}+\frac{3}{2\bl c_1}\,\log x\Big)+\frac{d-1}{2\bl c_1}\,\log x,\label{eq:txdef}
\end{align}
which is the anticipated asymptote of the FPT $\tau_x$ (c.f.~\eqref{eq:taux asymp}), where we recall that $I(c_1)=\log\rho$, $I'(c_1)=c_2$, and $I$ is the rate function for the first coordinate of $\bxi$.  Denote by $M_n$ the maximum at level $n$ of a BRW with jump $\xi$, and its asymptote $m_n:=c_1n-\frac{3}{2c_2}\log n,\,n\geq 1$ (with $m_0:=0$).

We first explain the intuition behind the decomposition \eqref{eq:txdef} and the main difficulties behind the proof of Theorem \ref{thm:main}. By inverting the expression for $m_n$, one expects that the FPT to $\H_x:=[x,\infty)\times\R^{d-1}$ is around 
$$t^{(1)}_x:=\frac{x}{c_1}+\frac{3}{2\bl c_1}\,\log x,$$
because $m_{t^{(1)}_x}=x+O(1)$. The other term $\frac{d-1}{2\bl c_1}\,\log x$ can be understood as the extra waiting time for more particles to enter $\H_x$ so that one of them would enter $B_x$. More precisely, we note the following.
\begin{itemize}
    \item As a result of the local CLT,\footnote{Strictly speaking, we will apply the local CLT under a change of measure and possibly under barrier constraints; see Lemmas \ref{lemma:conditioned local CLT} and \ref{lemma:uniform local CLT}.} each particle that arrives in $\H_x$ has a chance around $x^{-\frac{d-1}{2}}$ of landing in $B_x$.
    \item If we wait until time $t_x$, then there are $\asymp (\log x)x^{\frac{d-1}{2}}$ particles hitting $\H_x$, a fact justified by Lemma \ref{prop:number of particles} below.
\end{itemize}
 Therefore, on average, there would be $\asymp \log x$ particles found in $B_x$ at time $t_x$. Meanwhile, it is certainly not the case that the displacements in the last $d-1$ coordinates are almost independent among the ($(\log x)x^{\frac{d-1}{2}}$ many) \textit{frontier particles} in the first coordinate, i.e., those that travel fast in the first dimension. 
For example, given two frontier particles,
\begin{itemize}
    \item if they are separated very early genealogically, then the events that they lie in $B_x$ at time $t_x$ are close to being independent;
    \item if they are separated very late genealogically, then the events that they lie in $B_x$ at time $t_x$ are strongly dependent.
\end{itemize}
In other words, to obtain bounds of $\tau_x$, one needs to understand the genealogy (or the dependence structure) of the (roughly $(\log x)x^{\frac{d-1}{2}}$ many) frontier particles that are the fastest in the first dimension around time $t_x$. 

The decomposition \eqref{eq:txdef} is also identified in Section 1.3 of \citep{blanchet2024first}, where weaker versions of Theorems \ref{thm:main} and \ref{thm:main2} are proved. The techniques in \citep{blanchet2024first} are confined to the modified second moment method, without an in-depth study of the underlying genealogy of the frontier particles, which we explain next.

\textbf{Prelude: introducing the role of clusters in the genealogy of the frontier particles.} 
The goal of the paragraphs in this subsection is to expose a cluster structure that is useful in our analysis. We do this by introducing a conditional probability (see \eqref{eq:sbp model} below) that captures key elements in our analysis, although we do not directly study this probability in our future development, it is useful as a device to quickly see the main ingredients that will come into play. As we shall discuss later, the (non-trivial) genealogy of the frontier particles is reasonably well understood. 
Suppose that we condition on both the genealogy of the BRW up to time $t_x$ and the BRW but only its projection onto the first coordinate. Consider the set $T=\{v\in V_{t_x}:\,\bet_{v,t_x}(t_x)\in\H_x\}$, which is measurable and, by the above discussion, it has cardinality $\# T\asymp (\log x)x^{\frac{d-1}{2}}$ with high probability. Next, define the (conditioned) locations of the particles $v\in T$ in the remaining $d-1$ dimensions, written as $\{\bX_v\}_{v\in T}:=\{\widehat{\bet}_{v,t_x}(t_x)\}_{v\in T}$. 
Recall that our goal is to find a particle in $B_x$ at time $t_x$, which is, roughly speaking, equivalent to finding a particle $v\in T$ such that $\n{\bX_v}\leq 1$.\footnote{This is because a non-trivial proportion of particles in $T$ will be located in $[x,x+1/2]$ in the first coordinate, so we may simply look at particles in $\H_x$ at time $t_x$.} Here and later, we use $\n{\cdot}$ to denote the Euclidean norm. 
We are then reduced to the following problem (after conditioning on the genealogy and the first BRW coordinate of the frontier particles): given an $\R^{d-1}$-valued stochastic process $\bX=\{\bX_v\}_{v\in T}$ where $T$ is a finite set, 
how to characterize the probability 
\begin{align}
    \p(\exists\, v\in T,\,\n{\bX_v}\leq 1 )\label{eq:sbp model}
\end{align}
up to multiplicative constants, in terms of the dependence structure of $\{\bX_v\}_{v\in T}$ (keep in mind that the probability in \eqref{eq:sbp model} is conditional, as stated earlier, so bounds on \eqref{eq:sbp model} are to be understood as high-probability bounds).  Intuitively, $\bX_v$ and $\bX_w$ are strongly dependent if $v$ and $w$ are separated very late in the underlying genealogy, and vice versa. 

While the problem \eqref{eq:sbp model} appears fundamental, we are unaware of a general solution.\footnote{A particularly interesting problem would be, for instance, assuming $\bX$ is one-dimensional centered Gaussian (and hence written as $X$), characterize \eqref{eq:sbp model} in terms of the distance $d_X(v,w)=\sqrt{\E[(X_v-X_w)^2]}$. Relevant studies include \citep{BiermeLacauxXiao2009,NualartViens2014}, which provide upper and lower bounds in special cases using capacity and Hausdorff measure.} There are a few natural ideas for upper and lower bounding the quantity \eqref{eq:sbp model}:
\begin{itemize}
    \item[A)] To bound $\p(\exists\, v\in T,\,\n{\bX_v}\leq 1 )$ from below, pick a subset $T'\subseteq T$ such that $\bX_v$ and $\bX_w$ are approximately independent for all $v,w\in T',\,v\neq w$, and use $\p(\exists\, v\in T',\,\n{\bX_v}\leq 1 )$ as a lower bound.
    \item[B)] To bound $\p(\exists\, v\in T,\,\n{\bX_v}\leq 1 )$ from above, partition $T$ into ``well-separated blocks'' $T_1,\dots,T_m$ that are approximately independent (i.e., for any $v\in T_i,\,w\in T_j,\,i\neq j$, $\bX_v$ and $\bX_w$ are roughly independent) and such that the ``size'' of each block is small (and hence for each $1\leq j\leq m$,$~\p(\exists\, v\in T_j,\,\n{\bX_v}\leq 1 )$ is small).\footnote{This can be viewed as a simplified version of the generic chaining technique \citep{talagrand2022upper}.}
\end{itemize}
In practice, these bounds work well if the set $T$ has the following \textit{cluster structure}: $T$ can be partitioned into well-separated blocks/clusters with very small sizes (for instance, consider the extreme case where $\{\bX_v\}_{v\in T}$ forms an i.i.d.~sequence, or an identical sequence). Obtaining a characterization of \eqref{eq:sbp model} for a general process $\{\bX_v\}_{v\in T}$ is difficult, but is not necessary for our purpose since we are interested in the special class of processes $\bX$ that describes the genealogy of the extremal particles for a one-dimensional BRW. This highlights the importance of understanding the cluster structure of extremal particles.

\textbf{The cluster structure of extremal particles.}
Conditioning on the particle genealogy of those in $\H_x$ at time $t_x$, we effectively obtain a stochastic process $\{\bX_v\}_{v\in T}$ describing the displacements in the last $d-1$ dimensions of the frontier particles. 
In this case, we give upper and lower bounds for \eqref{eq:sbp model} that may not match in general, but surprisingly, they coincide (up to multiplicative constants) with high probability.  
This is because of the following nice feature of the \textit{one-dimensional} BRW: 
\begin{align}
    \begin{split}
        &\text{the frontier particles are most likely to be separated either}\\
        &\hspace{0.63cm}\text{\textit{very early} or \textit{very late} in the underlying genealogy}.
    \end{split}
\label{eq:quote}
\end{align}
 A precise formulation concerning a number of $O(1)$ many frontier particles can be found in Theorem 4.5 of \citep{mallein2016asymptotic} in the context of BRW, and Theorem 2.1 of \citep{arguin2011genealogy} in the context of BBM. Loosely speaking, as $n\to\infty$, the particles beyond $m_n$ at time $n$ are separated in either the first $O(1)$ steps or the last $O(1)$ steps with high probability. This observation can be generalized to the study of around $(\log x)x^{\frac{d-1}{2}}$ many frontier particles (or particles beyond $x$ at time $t_x$), using an extension of Proposition 8 of \citep{blanchet2024first}. For $0\leq n\leq t_x$, we define the \textit{production number} $P_n$ as the number of particles at level $n$ that allow a descendant beyond $x$ in the first coordinate, at time $t_x$. For instance, $P_0=1$ and $P_{t_x}\asymp (\log x)x^{\frac{d-1}{2}}$ with high probability. 
 We will prove in Proposition \ref{lemma:transition} that $P_n$ only has non-trivial increase on the intervals $n\in[0,O((\log x)^2)]\cup [t_x-O((\log x)^2),t_x]$; see Figure \ref{fig:Pn} for an illustration. In other words, the \textit{majority} of the (roughly $(\log x)x^{\frac{d-1}{2}}$ many) particles beyond $x$ at time $t_x$ are either separated in the first $O((\log x)^2)$ steps or the last $O((\log x)^2)$ steps with a non-trivial probability.\footnote{Due to the nature of the second moment method we apply, we cannot conclude a with-high-probability statement. While we resolve this issue by using the exponential concentration of the FPT (Lemma \ref{lemma:concentration} below), we conjecture that such a property holds with high probability. }
\begin{figure}[h!]
    \centering
    \includegraphics[width=0.6\textwidth]{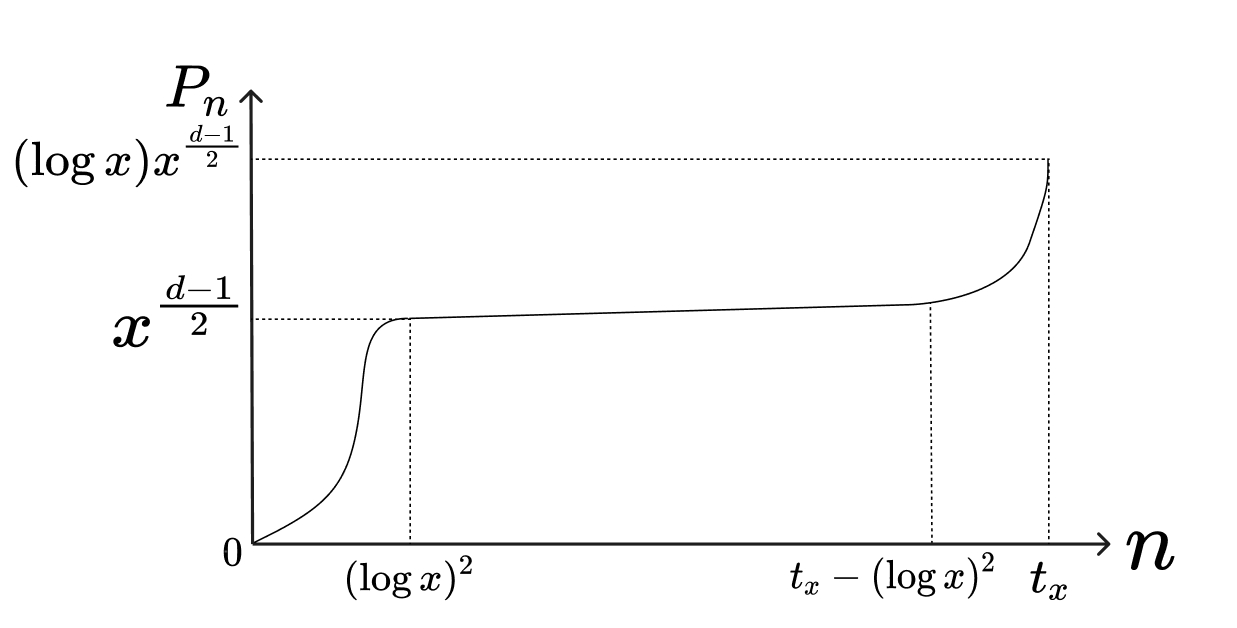}
    \caption{Typical growth pattern of the production number $P_n$. The plot in the interval $[O((\log x)^2),t_x-O((\log x)^2)]$ is almost flat, reflecting the fact that the only non-trivial increase arises from $n\ll (\log x)^2$ or $t_x-n\ll (\log x)^2$.}
    \label{fig:Pn}
\end{figure}

The effect \eqref{eq:quote} is related to \textit{entropic repulsion}, which describes the phenomenon that a typical path leading to maximum lies well below the interpolating line in most intermediate times. For this reason, those locations well below the interpolating line are not favorable as a branching location that leads to another extremal particle. Thus, the overwhelming majority of the leading frontier particles will exhibit the entropic repulsion phenomenon. This means that early on in the history of the BRW, leading particles' induced ``clusters'' in the genealogy start being formed early on (within $(\log x)^2$ time), and by time $t_x$ there are roughly $x^{(d-1)/2}$ many clusters\footnote{Note that this is a rather coarse approximation, each leading particle could form early on itself a random number of clusters, but the expectation of this number is finite.} that are well separated in the metric of the tree generated by the genealogy.  
More comprehensive discussions of the leading particles' genealogy can be found in \citep{arguin2011genealogy,cortines2019structure,cortines2021more,hartung2024growth}. However, a major difference is that these works focused on $O(1)$ many frontier particles of the BBM, instead of $\asymp (\log x)x^{\frac{d-1}{2}}$ many frontier particles of the BRW. Another technical difference is that their {clusters} classify \textit{all} particles in the branching genealogy by the genealogical distance and are re-centered by the maximum location in each cluster; in our case, we focus only on particular particles lying in $\H_x$ at time $t_x$ (instead of collecting all of them).

Turning to the picture of the process $\{\bX_v\}_{v\in T}$, this suggests that the set $T$ enjoys the cluster structure suggested earlier. Figure \ref{fig:cluster} illustrates this phenomenon. Roughly speaking, each element in the partition of $T$ into blocks then corresponds to a collection of particles that do not separate until time $t_x-O((\log x)^2)$, or equivalently until $O((\log x)^2)$ (since $P_n$ barely increases in the time interval $[O((\log x)^2),t_x-O((\log x)^2)]$), meaning that these blocks are well-separated.  As a consequence of Proposition \ref{lemma:transition}, the number of such blocks is around $x^{\frac{d-1}{2}}$. By a conditional local CLT we establish below (Lemma \ref{lemma:conditioned local CLT}), each block has a chance of around $x^{-\frac{d-1}{2}}$ of having a particle located in $B_x$, and these events for each block are approximately independent {(a key ingredient in A))}. The upper bound for $\tau_x$ then follows, which we elaborate on in Section \ref{sec:UB}.

On the other hand, {turning to the block size mentioned in B)}, it is nontrivial to show that the sizes of these individual clusters are small. One can show that each cluster has size $\ll \log x$, most have size $\ll 1$, and on average has cardinality $\asymp\log x$, but these pieces of information are not sufficient to conclude a matching upper bound for \eqref{eq:sbp model}. To proceed further, one needs the following crucial observation. There are {two} fundamentally different ways to upper bound the ``sizes'' of individual clusters: \begin{itemize}
    \item The size of a cluster $T_j$ is small if it contains very few elements (i.e., its cardinality is small). In this case, we use the union bound to obtain
    \begin{align}
        \p(\exists\, v\in T_j,\,\n{\bX_v}\leq 1 )\ll x^{-\frac{d-1}{2}}\#T_j.\label{eq:bound1}
    \end{align}
    \item \sloppy The size of a cluster $T_j$ is small if its ``dispersion'' is small, precisely, if $\E[\sup_{v,w\in T_j}\n{\bX_v-\bX_w}^{d-1}]$ is small. In this case, we expect that 
    \begin{align}
        \p(\exists\, v\in T_j,\,\n{\bX_v}\leq 1 )\ll x^{-\frac{d-1}{2}}\E\Big[\sup_{v,w\in T_j}\n{\bX_v-\bX_w}^{d-1}\Big].\label{eq:bound2}
    \end{align}
\end{itemize}

\begin{figure}[h!]
    \centering  \includegraphics[width=0.7\textwidth]{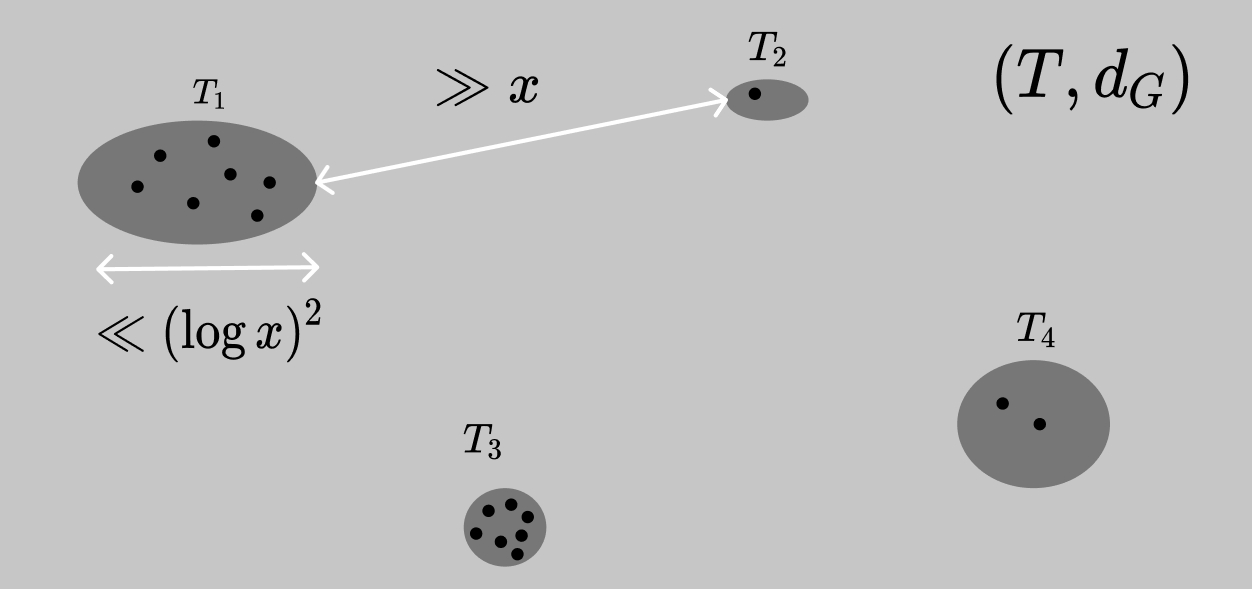}
    \caption{A typical cluster structure of the (random) set $T$. Consider the metric $d_G$  on $T$ defined as the genealogical distance of two particles $v,w\in T$ (i.e., if $u\in V_\ell$ is the latest common ancestor of $v,w\in V_{t_x}$, then $d_G(v,w)=t_x-\ell$).  Solid ellipses indicate the clusters. The set $T$ consists of $\asymp x^{\frac{d-1}{2}}$ clusters that are well-separated by distances of order $\gg x$; each of the clusters has diameter $\ll (\log x)^2$, in the metric space $(T,d_G)$. The dispersion of a cluster $T_i$ can be measured as its radius in the metric $d_G$. The set $T_1$ has a large cardinality and a large dispersion; $T_2$ has a small cardinality and a small dispersion; etc. }
    \label{fig:cluster}
\end{figure}

In summary, our goal is to show that, with high probability, the random set $T$ exhibits the cluster structure explained, and most of the $\asymp x^{\frac{d-1}{2}}$ many clusters satisfy the following: either its cardinality is small, or its dispersion is small. Equivalently, consider the collection $\mathcal P$ of particles at time 
\begin{align}
    \wtx:=t_x-(\log x)^2\label{eq:wtx def 1}
\end{align}
that lead to a descendant beyond $x$ at time $t_x$ (in the first coordinate). We need to show that with high probability, for \textit{most} particles in $\mathcal P$, either each of these has very few descendants reaching $x$ at time $t_x$, or all of its descendants that reach $x$ have a very young common ancestor (so the dispersion is controlled with the help of a suitable conditional local CLT). Achieving this goal is the most technical part of this paper. In the following, we attempt to sketch the intuition without going into too many details. 

\textbf{Bounding the size of the clusters.}
Recall \eqref{eq:wtx def 1}. 
The plan is to condition on an ancestor at time $\wtx $ (as well as the first coordinate of its location), discretize the space, and perform the following multi-step conditioning analysis of the BRW in time $[\wtx ,t_x]$:
\begin{itemize}
    \item Look at a particle $v\in V_{\wtx }$ that is near the location $x-m_{(\log x)^2}-\ell,\,\ell\in\bZ$.
    \item Consider the BRW process initiated at $v$. Condition on the \textit{heterogeneity index} $h$ of $v$, defined as the age of the \textit{latest common ancestor} of all particles present in $[x,\infty)$ at time $t_x$, in the sub-tree initiated at $v$. For example, if only one descendant of $v$ reaches $[x,\infty)$ at time $t_x$, then $h=1$. If none reaches, then we set $h=0$. Obviously, $h\in[0,(\log x)^2]\cap\bZ$.
     \item Condition on the event that the {location} of the latest common ancestor at time $t_x-h$ is near $x-m_h+g,\,g\in\bZ$. 
\end{itemize}
Figure \ref{fig:lhg} below illustrates the three parameters $\ell,h,g$.

\begin{figure}[h!]
    \centering
    \includegraphics[width=0.7\textwidth]{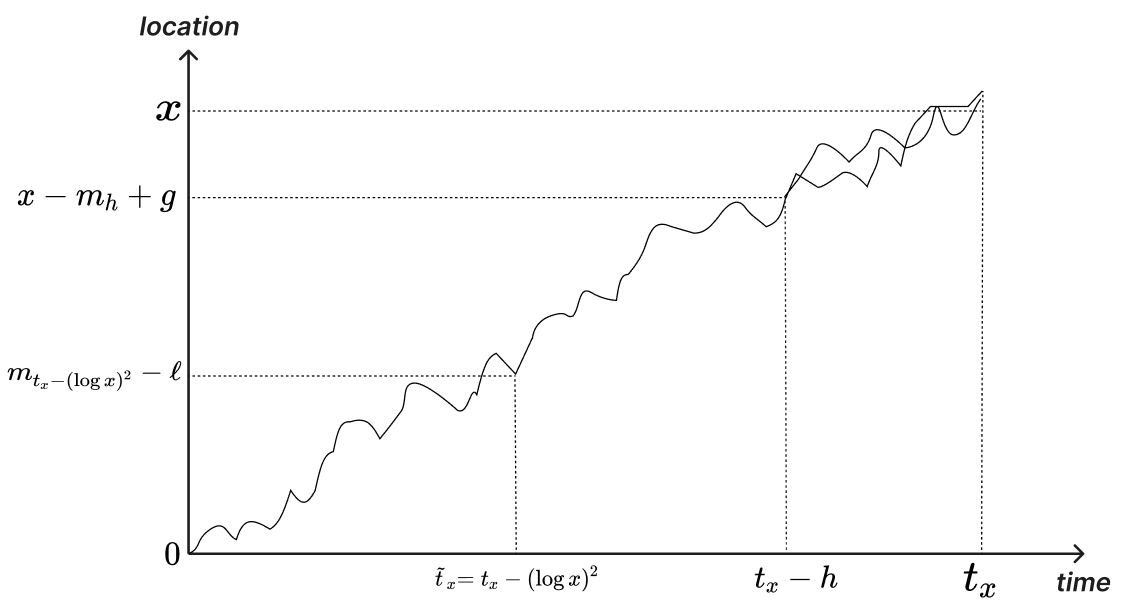}
    \caption{Illustration of the parameters $\ell,h,g$. Solid curves indicate the trajectories of the BRW.}
    \label{fig:lhg}
\end{figure}

More precisely, we will apply a first moment method conditionally on $\ell$, and apply a union bound on $h,g$. 
In this way, all particles $v\in V_{\wtx }$ have been classified by the indices $(\ell,h,g)$. Following the above discussion, we first consider the following three cases.
\begin{enumerate}[(a)]
 \item If $\ell$ is small,\footnote{Here and below, the \textit{smallness} of $\ell,g$ refers to having a (very) negative value, instead of having a small absolute value.} the number of such particles $v$ will be small. This will be shown in Proposition \ref{prop:particle density}.
    \item If $h$ is small, the size of the cluster corresponding to the particle $v$ is small in the set $T$. This corresponds to the case \eqref{eq:bound2} and will be proved in Lemma \ref{lemma:uniform bd}.
    \item If $g$ is small, the number of descendants of $v$ reaching $x$ at time $t_x$ is small (note that we condition on two particles separated at time $t_x-h$ that reach $x$ at time $t_x$, so such a number must be positive), meaning that the cardinality of the cluster corresponding to $v$ is small. This corresponds to the case \eqref{eq:bound1} and will be proved in Lemma \ref{lemma:2prob2case}. 
   
\end{enumerate}
It remains to consider case (d): $\ell,h,g$ are all large. There are two sub-cases.
\begin{enumerate}
    \item [(d1)] If $\ell,g$ are large and $h\approx (\log x)^2$, this means that in the small time period $[\wtx ,t_x-h]$, the trajectory travels a distance of $\ell+g+m_{(\log x)^2}-m_h$ which is significantly larger than $m_{(\log x)^2-h}$, and thus happens with a tiny probability. This is the goal of Lemma \ref{lemma:E_1}. 
    \item [(d2)] If $\ell,g$ are large and $h$ is close to neither $0$ nor $(\log x)^2$, we are in a situation where the BRW initiated by $v$ satisfies that two descendants of $v$ located beyond $m_{(\log x)^2}+\ell$ in time $(\log x)^2$ have a common ancestor that is neither too early nor too late. This must happen rarely, since it contradicts the philosophy \eqref{eq:quote}. The analysis is hidden in the computation of sums over $h$ in the proofs of Lemmas \ref{lemma:ell<loglog} and \ref{lemma:ell>loglog}.
\end{enumerate}
Therefore, in all cases, the size of the cluster corresponding to the particle $v$ can be controlled with high probability. 
One extra technicality comes into play since the statement \eqref{eq:quote} requires removing ballot-type events where the random walks cross a certain barrier. When applying \eqref{eq:quote} in case (d2), one needs to remove the barrier events for each $v\in V_{\wtx }$ within our consideration. Clearly, removing the events for all $v\in V_{\wtx }$ is extremely costly because there are exponentially many such particles. To overcome this issue, we remove barrier events only for those \textit{relevant} $v\in V_{\wtx }$. These are the particles $v$ where the last $(d-1)$-dimensional location of the latest common ancestor at time $t_x-h$ is close enough to the origin (say within a distance of $\asymp h$, so that it has a sufficient chance to reach $B_x$), in addition to satisfying the prescribed events. 
This finishes the upper bound for \eqref{eq:sbp model} and consequently the desired lower bound of $\tau_x$.

Finally, we remark that in addition to nailing down the precise asymptotic of the FPT, our approach naturally leads to high probability properties of the trajectory that first realizes the FPT. We may identify the main contribution to the total size of the clusters emanating from distinct values of $(\ell,h,g)$---it will become apparent from our proof that the main contribution stems from $\ell\asymp\log x, h=O(1),$ and $g=O(1)$. In other words, one can show that with high probability, the trajectory that realizes the FPT satisfies:
\begin{itemize}
    \item its location at time $\wtx$ belongs to $[m_{t_x-(\log x)^2}-(\log x)/\ee,m_{t_x-(\log x)^2}-\ee\log x]$ for some small $\ee>0$;
    \item the collection of descendants of its ancestor at time $\wtx$ that reach $\H_x$ at time $t_x$ has a latest common ancestor at time $\wtx-O(1)$;
    \item if $h$ denotes the age of that latest common ancestor, then its location at time $t_x-h$ is around $x-m_h+O(1)$.
\end{itemize}



\sloppy\textbf{Notation.} We typically use (possibly with subscripts) $u,v,w$ to denote particles; $P,V,W$ to denote collections of particles; $E,G,H,\sA,\sB,\sC,\dots,\sK$ to denote events; $t,\tau,h$ to denote time; $x,g,\ell,\bx,\bu$ to denote locations or distances in Euclidean spaces; $\psi,\widehat{\psi},\widetilde{\psi},\overline{\psi}$ to denote barrier functions. Vectors are typically denoted by bold symbols. The notation $\delta>0$ (resp.~$L,C>0$) typically refers to a small (resp.~large) constant depending on the law of the BRW that may vary from line to line; $K_1,K_2,\dots,K_{12}$ denote large constants that may depend on each other (in a permissible order) and the law of the BRW (including the underlying dimension $d$). Denote by $\bone_A$ the indicator of an event $A$. The first time a definition appears is always followed by the ``$:=$'' sign. We refer to Appendix \ref{appendix:notation} for a glossary of frequently used notation and definitions throughout this paper.

\textbf{Outline of the paper.} Section \ref{sec:prelim} collects a few useful results on one-dimensional BRW and applies them to study the transition in the production number $P_n$, concluded by Section \ref{sec:UB} that proves the desired upper bound of the FPT in Theorem \ref{thm:main}. The proof of the corresponding lower bound takes up the entire Section \ref{sec:LB}, where we gradually carry out the multi-step conditioning plan outlined above. Section \ref{sec:discussions} contains a sketch of the extra arguments required for the proof of Theorem \ref{thm:main2} and discussions on varying targets. Appendices \ref{sec:escape}--\ref{sec:cond local CLT proof} are devoted to several preliminary tools involving the escape probability of BRW, ballot theorems, and a conditional local CLT.

\section{Preliminary results}\label{sec:prelim}

\subsection{Useful results for the extremal behavior of one-dimensional BRW}\label{sec:one-dim results}
This section contains a few useful lemmas that are \textit{established} results for one-dimensional BRW. A few other results that need further verification will be collected in Appendix \ref{sec:upper ballot}.
We assume throughout this section that the BRW satisfies assumptions (A1)--(A4) with $d=1$, except for Lemma \ref{lemma:concentration}. Let $S$ denote the event that the underlying BRW survives at all times. In all events considered below, we omit the conditioning on the survival event $S$ for brevity.



For $\beta>0$ and $n\in\N$, we define the barrier event
\begin{align}
    \sG_{n,\beta}:=\bigcup_{v\in V_n}\bigcup_{0\leq k\leq n}\left\{\eta_{v,n}(k)\geq \frac{km_n}{n}+\beta+\frac{6}{\bl}(\log\min\{k,n-k\})_+\right\},\label{eq:gnb}
\end{align}
where $(\cdot)_+$ denotes the positive part of an extended real number and by definition $(\log 0)_+=(-\infty)_+=0$. Let us also define 
\begin{align}
    \varphi_{n,\delta}(i):=e^{-\delta|i|\min(\frac{|i|}{n},1)}.\label{eq:phind}
\end{align}

\begin{lemma}[Lemma 2.4 of \citep{bramson2016convergence}]\label{lemma:barrier}
There exists $\delta>0$ such that 
\begin{align}
    \p(\sG_{n,\beta})\ll \beta e^{-c_2\beta}\varphi_{n,\delta}(\beta).\label{eq:barrier}
\end{align}
Moreover, if $\tau_{n,\beta}$ denotes the smallest $k$ such that the event $\sG_{n,\beta}$ occurs, we have
$$\p(\tau_{n,\beta}=j)\ll \min\{j,n+1-j\}^{-3}\beta e^{-c_2\beta}\varphi_{n,\delta}(\beta).$$
\end{lemma}

\begin{proof}
    The first claim is precisely Lemma 2.4 of \citep{bramson2016convergence}. The second claim is a restatement of equation (23) therein, where the power is $-3$ instead of $-2$ because we changed the coefficient of the logarithm term in the definition \eqref{eq:gnb} of $\sG_{n,\beta}$.
\end{proof}

We recall that $M_n$ stands for the maximum of the BRW in generation $n$.

\begin{lemma}[Corollary 2.5 and Lemma 2.7 of \citep{bramson2016convergence}]\label{lemma:brw small deviation prob}
There exist $\delta,C>0$ such that for  $z\geq 0$,
\begin{align}
    \p(M_n>m_n+z)\leq C(z+1)e^{-c_2z}\varphi_{n,\delta}(z).\label{eq:BRW max tail}
\end{align}
Moreover, for $z\leq \sqrt{n}$,
\begin{align}
    \p(M_n>m_n+z)\geq \frac{1}{C}ze^{-c_2z}.\label{eq:BRW max tail2}
\end{align}
In other words, \eqref{eq:BRW max tail} is tight up to constants for $z\leq \sqrt{n}$.
    \end{lemma}

    \begin{remark}
        The proof of \eqref{eq:BRW max tail2} proceeds by first applying the simple inequality $$ \p(M_n>m_n+z)\geq \p(\exists v\in V_n:\,\eta_{v,n}(n)\in[m_n+z,m_n+z+1)).$$ 
As a consequence, by slightly modifying the proof in \citep{bramson2016convergence}, it holds that for $1\leq z\leq \sqrt{n}$,
\begin{align}
   \p\Big(\exists v\in V_n,\,\Big|\eta_{v,n}(n)-(m_n+z)\Big|\leq \frac{1}{4}\Big) \geq \frac{1}{C}ze^{-c_2z}.\label{eq:BRW max tail3}
\end{align}
        
    \end{remark}
    
\begin{remark}
    The estimates \eqref{eq:barrier} and \eqref{eq:BRW max tail} were stated in \citep{bramson2016convergence} in the form 
    $$\p(M_n>m_n+z)\leq C(z+1)e^{-c_2z}e^{-\delta|z|\min(\frac{|z|}{n\log n},1)}.$$
    On the other hand, the authors of \citep{bramson2016convergence} remarked below the statement of Lemma 2.4 therein that \eqref{eq:BRW max tail} holds with a slightly modified argument. We sketch the missing argument below for completeness. The only missing piece therein is the validity of equation (20) in \citep{bramson2016convergence}, uniformly in $i<\beta-C\sqrt{n}$ instead of $i<\beta-C\sqrt{n\log n}$. After a proper change of measure using Lemma 2.2 therein, it suffices to show that for some $\delta>0$,
    $\p(S_k<i)\ll e^{-\frac{\delta |i|^2}{n}}$
    uniformly in $i\in [\beta-C\sqrt{n\log n},\beta-C\sqrt{n}]$ and $1\leq k\leq n$, where $S_k=\sum_{j=1}^k\xi_j$ is partial sum of an i.i.d.~sequence with law given by the jump of the BRW. Using the Skorohod embedding theorem, we may write $S_k\dd B_{\tau_k}$ where $B$ is Brownian motion and $\tau_k$ is a sum of $k$ i.i.d.~nonnegative random variables with a finite second moment. It then follows that uniformly for $1\leq k\leq n$,
\begin{align*}
    \begin{split}
        \p(S_k<i)=\p(B_{\tau_k}<i)&\leq \p\Big(\tau_k\geq \frac{n}{4\delta}\Big)+\p\Big(\sup_{0\leq s\leq n/(4\delta)}B_s>|i|\Big)\\
        &\ll n^{-2}+e^{-\frac{|i|^24\delta}{2n}}\ll e^{-\frac{\delta|i|^2}{n}},
    \end{split}
\end{align*}
    where we have used Remark 8.3 of \citep{karatzas2014brownian} in the second inequality and the $\ll $ may depend on $\beta,\delta$.
    
\end{remark}

\begin{remark}
    Note also that the paper \citep{bramson2016convergence} assumed that the underlying BRW cannot terminate, i.e., $p_0=0$. However, as long as $\rho>1$, the general case follows from the case $p_0=0$ by the same proof therein (which only relies on the first and second moments of the number of particles) and a decomposition theorem of supercritical branching processes (Theorem 1 of Section 12 of \citep{athreya2004branching}), which states that a supercritical branching process conditioned on $S$ has the same finite-dimensional distributions (in terms of particle counts) as another supercritical branching process satisfying $p_0=0$ and having the same mean offspring distribution.
\end{remark}

Define the collection of particles
$$Q_{n,\beta}:=\left\{v\in V_n : \text{ for any }0\leq k\leq n,~\eta_{v,n}(k)<\frac{km_n}{n}+\beta+\frac{6}{\bl}(\log\min\{k,n-k\})_+\right\}.$$
\begin{lemma}[Proposition 9 of \citep{blanchet2024first}]\label{lemma:gnb}Uniformly in $y\in[2,\sqrt{n}]$, 
    $$\E[\#\{v\in Q_{n,\beta} : \eta_{v,n}(n)\geq m_n-y\}]\ll \beta (y+\beta)e^{\bl y}.$$
\end{lemma}

    \begin{lemma}[Proposition 8 of \citep{blanchet2024first}]
        \label{prop:number of particles} There exists $L>0$ depending only on the law of the BRW such that the following holds conditioned upon survival. Given any $\ee>0$, there exists $C>0$ independent from $n$ and $y$   such that  uniformly for $n$ large enough and for $y\in[2,\sqrt{n}]$, 
    \begin{align}
        \p\left(\#\{v\in V_n : \eta_{v,n}(n)\geq m_n-y\}>Cye^{\bl y}\right)<\ee\label{eq:ub number}
    \end{align}
    and 
    \begin{align}
        \p\left(\#\{v\in V_n : \eta_{v,n}(n)\geq m_n-y\}>\frac{1}{C}ye^{\bl y}\right)>\frac{1}{L}.\label{eq:lb number}
    \end{align}
    \end{lemma}
    The above results are closely related. For instance, Lemmas \ref{lemma:barrier} and \ref{lemma:gnb} together yield \eqref{eq:ub number} as $\beta (y+\beta)e^{\bl y}\asymp_\beta ye^{\bl y}$. 
The proofs of Lemmas \ref{lemma:gnb} and \ref{prop:number of particles} are based on a ``ballot theorem under a change of measure'' argument, which is standard for the study of extrema of spatial branching processes and will be frequently used in this work.  We refer to \citep{blanchet2024first,bramson2016convergence} for further details.  

The lower bound $1/L$ (instead of the anticipated stronger lower bound $1-\ee$) of the probability in \eqref{eq:lb number} is an artifact of the second moment method. This bound alone does not suffice for proving high-probability upper bounds of $\tau_x$. To resolve this issue, the work \citep{blanchet2024first} established a concentration bound for $\tau_x$ around its median. 
Let $\mathrm{Med}(\cdot)$ denote the median of a random variable.
\begin{lemma}[Theorem 2 of \citep{blanchet2024first}]\label{lemma:concentration}Let $\tau_x$ be the first passage time to $B_x$ for a $d$-dimensional BRW satisfying conditions (A1), (A2), and (A6). There exist constants $C,c>0$ independent of $x$ such that for each $y\in[0,x]$,
    \begin{align*}
        \p\left(|\tau_x-\mathrm{Med}(\tau_x\mid S)|>y\right)\leq Ce^{-cy}.
    \end{align*}    
\end{lemma}

   

\subsection{Transition in the production number \texorpdfstring{$P_n$}{}}
  In this subsection, we formulate the quote \eqref{eq:quote} in the form we need using the notion of production numbers, keeping in mind that we look at a neighborhood of length $\log x$ near extrema.
Recall \eqref{eq:txdef} and that the {production number} $P_n$ is defined as
$$P_n:=\#\{v\in V_n:\,\exists w\in V_{t_x},\,w\succ v,\,\eta_{w,t_x}(t_x)\geq x\},~0\leq n\leq t_x,$$
where $w\succ v$ means that particle $w$ is a descendant of $v$ (including the case $w=v$). 
For our purpose, it is also useful to bound from below a similar quantity $P_n'$ as $P_n$, defined as 
$$P_n':=\#\Big\{v\in V_n:\,\exists w\in V_{t_x},\,w\succ v,\,\eta_{w,t_x}(t_x)\in[x-\frac{1}{2},x+\frac{1}{2}]\Big\},~0\leq n\leq t_x.$$
Here and later, the upper and lower limits of a sum are always interpreted as integers, without loss of generality. The main result in this subsection is the following extension of Lemma \ref{prop:number of particles}.

\begin{proposition}[Transition in $P_n$]\label{lemma:transition}
  (i)  For any $\ee>0$, there is $C>0$ such that uniformly for $x$ large enough, $$\p(P_{\wtx }\geq Cx^{(d-1)/2}\mid S)<\ee.$$
  (ii)  There are $L,C>0$ such that uniformly for $x$ large enough,
  $$\p\Big(P'_{(\log x)^2}\geq \frac{1}{C} x^{(d-1)/2}\mid S\Big)>\frac{1}{L}.$$
\end{proposition}
\begin{remark}\label{rem:no S}
    A well-known fact of supercritical branching processes is that conditioned on extinction, the lifespan has an exponential tail. In particular, $\p(\#V_n>0\mid S^c)=o(1)$ (see Theorem 13.3 of \citep{athreya2004branching}). It follows that the statement of Proposition \ref{lemma:transition} is essentially equivalent to the same statement without conditioning upon survival. The general idea behind proving Proposition \ref{lemma:transition} is to first condition on the configuration at the time of interest (say, $(\log x)^2$), classify the particles at such a time according to their locations (while discretizing the space), and finally evolve these particles independently until time $t_x$. 
\end{remark}

\begin{remark}\label{rem:bootstrap}
    The $1/L$ lower bound of the probability in Proposition \ref{lemma:transition}(ii) can be improved to a $1-\ee$ lower bound by a bootstrapping argument: first run the BRW process until a large time $T$ to obtain sufficiently many (say $N$) alive particles near the origin,\footnote{This can be seen from a simple adaptation of the proof of Lemma \ref{lemma:escape} below.} and then independently evolve them to boost the lower bound $1/L$ to $1-(1-1/L)^N$, while applying a slightly stronger form than Proposition \ref{lemma:transition}(ii): there are $L,C>0$ such that for every fixed $T>0$, uniformly for $x$ large enough and $t\in\{0,1,\dots,T\}$,
  $$\p\Big(P'_{(\log x)^2-t}\geq \frac{1}{C} x^{(d-1)/2}\mid S\Big)>\frac{1}{L},$$which can be verified using essentially the same proof. We omit the details.
\end{remark}

\begin{proof}In the following, we use frequently the fact that
\begin{align}
    \begin{split}
        m_{\wtx}&=x+\frac{d-1}{2c_2}\log x-c_1(\log x)^2+o(1)\\
    &=x-\Big(m_{(\log x)^2}-\frac{d-1}{2c_2}\log x+\frac{3}{c_2}\log\log x\Big)+o(1),
    \end{split}\label{eq:mwtx}
\end{align}
which follows from a direct computation.

(i) Denote by $n=\wtx$. We first exclude a barrier event of arbitrarily small probability. Let $\ee>0$.
 By Lemma \ref{lemma:barrier}, $\p(Q_{n,\beta}\neq V_n)<\ee/2$ for some $\beta$ large enough.  
Therefore, we may without loss of generality assume that the event $\{Q_{n,\beta}=V_n\}$ holds.
 Define a collection of independent $\{0,1\}$-valued random variables $\{\delta_{v,y}\}_{v\in V_{\wtx},\,y\in(-\infty,x]\cap\bZ}$, independent from everything else, and such that
$$\p(\delta_{v,y}=1)=\p(M_{(\log x)^2}\geq x-y).$$
These random variables indicate whether a particle located within the interval $[y,y+1]$ at time $\wtx $ will have a descendant beyond $x$ at time $t_x$ (note that the evolution of the particles in time $[\wtx ,t_x]$, given the configuration at time $\wtx $, are independent).
For $u\in[2,\sqrt{n}]$ and $v\in V_n$, define the event
$$H_{v,n}(u):=\{\eta_{v,n}(n)\in[m_n-u,m_n-u+1)\}.$$
Now on the event $\{Q_{n,\beta}=V_n\}$,
$$P_{\wtx }\lst \sum_{y=-\infty}^{m_{\wtx }+\beta}\sum_{v\in V_{\wtx }}\delta_{v,y}\bone_{H_{v,\wtx }(m_{\wtx }-y)},$$
where $\lst$ denotes stochastic dominance. 
We then compute 
\begin{align*}
    &\hspace{0.5cm}\E\bigg[\bone_{\{Q_{n,\beta}=V_n\}}\sum_{y=-\infty}^{m_{\wtx }+\beta}\sum_{v\in V_{\wtx }}\delta_{v,y}\bone_{H_{v,\wtx }(m_{\wtx }-y)}\bigg]\\
    &=\rho^{\wtx }\sum_{y=-\infty}^{m_{\wtx }+\beta}\p(M_{(\log x)^2}>x-y)\,\p(H_{v,\wtx }(m_{\wtx }-y)\cap\{Q_{n,\beta}=V_n\}).
    \end{align*}
Using the change of variable $j=m_{\wtx }-y$ and \eqref{eq:mwtx}, the above is equal to
\begin{align}
    \begin{split}
       & \sum_{j=-\beta}^{\infty}\p\Big(M_{(\log x)^2}>m_{(\log x)^2}+(-\frac{d-1}{2c_2}\log x+j+\frac{3}{c_2}\log\log x)\Big)\\
       &\hspace{4cm}\times\rho^{\wtx }\p(H_{v,\wtx }(j)\cap\{Q_{n,\beta}=V_n\}).
    \end{split}\label{eq:summand}
\end{align}
Let $C>0$ be a large constant.  We divide the sum over $j$ in \eqref{eq:summand} into various ranges:
\begin{itemize}
    \item $-\beta\leq j\leq \frac{d-1}{2c_2}\log x-\frac{3}{c_2}\log\log x$. By a union bound and a large deviation estimate, the total contribution is controlled by
    $$\sum_{j=-\beta}^{\frac{d-1}{2c_2}\log x-\frac{3}{c_2}\log\log x}\rho^{\wtx }\p(H_{v,\wtx }(j))\ll \sum_{j=-\beta}^{\frac{d-1}{2c_2}\log x-\frac{3}{c_2}\log\log x}(\log x)^3e^{c_2j}\ll x^{\frac{d-1}{2}}.$$
    \item $\frac{d-1}{2c_2}\log x-\frac{3}{c_2}\log\log x\leq j\leq C\log x$. By Lemmas \ref{lemma:brw small deviation prob} and \ref{lemma:gnb}, this part contributes at most
    \begin{align*}
        &\hspace{0.5cm}\sum_{j=\frac{d-1}{2c_2}\log x-\frac{3}{c_2}\log\log x}^{C\log x} \Big(-\frac{d-1}{2c_2}\log x+j+\frac{3}{c_2}\log\log x\Big)e^{-c_2(-\frac{d-1}{2c_2}\log x+j+\frac{3}{c_2}\log\log x)}\\
        &\hspace{10cm}\times je^{c_2j}\\
        &\ll x^{\frac{d-1}{2}}(\log x)^{-3}\sum_{j=\frac{d-1}{2c_2}\log x-\frac{3}{c_2}\log\log x}^{C\log x}j\Big(-\frac{d-1}{2c_2}\log x+j+\frac{3}{c_2}\log\log x\Big)\\
        &\ll x^{\frac{d-1}{2}}.
    \end{align*}
    \item $C\log x\leq j\leq \sqrt{x}$. Again applying Lemmas \ref{lemma:brw small deviation prob} and \ref{lemma:gnb} leads to an upper bound of
    $$x^{\frac{d-1}{2}}(\log x)^{-3}\sum_{j=C\log x}^\infty je^{-c_2j-\frac{\delta j^2}{(\log x)^2} }je^{c_2j}.$$
    The sum can be bounded using an integral approximation:
\begin{align*}
    \sum_{j=C\log x}^\infty j\varphi_{(\log x)^2,\delta}(j)j&\ll \int_{C\log x}^{(\log x)^2} y^2e^{-\frac{\delta y^2}{(\log x)^2}}\d y+\int_{(\log x)^2}^\infty y^2e^{-\delta y}\d y\\
    &\ll (\log x)^3\int_C^\infty z^2e^{-\delta z}\d z\ll (\log x)^3.
\end{align*}
Therefore, this part of the contribution gives $\ll x^{\frac{d-1}{2}}$.

\item $j\geq \sqrt{x}$.  We directly apply the upper bound part of Cram\'{e}r's theorem along with the first moment method. Using convexity of $I$, we obtain
\begin{align*}
    &\hspace{0.5cm}\sum_{j=\sqrt{x}}^\infty \p\Big(M_{(\log x)^2}>m_{(\log x)^2}+j-\frac{d-1}{2c_2}\log x+\frac{3}{c_2}\log\log x\Big)\\
    &\hspace{5cm}\times\rho^{\wtx }\p(H_{v,\wtx }(j)\cap\{Q_{n,\beta}=V_n\})\\
    &\ll \sum_{j=\sqrt{x}}^\infty \rho^{t_x}e^{-(\log x)^2I(\frac{m_{(\log x)^2}+j-\frac{d-1}{2c_2}\log x+\frac{3}{c_2}\log\log x}{(\log x)^2})}e^{-\wtx I(\frac{m_{\wtx }-j}{\wtx })}\\
    &\ll \sum_{j=\sqrt{x}}^\infty x^{\frac{d-1}{2}}\rho^{(\log x)^2}e^{-(\log x)^2I(\frac{m_{(\log x)^2}+j-\frac{d-1}{2c_2}\log x+\frac{3}{c_2}\log\log x}{(\log x)^2})}e^{c_2j}.
\end{align*}
     If $x$ is large enough, then for some $\delta,\delta'>0$,
 $$\frac{m_{(\log x)^2}+j-\frac{d-1}{2c_2}\log x+\frac{3}{c_2}\log\log x}{(\log x)^2}\geq c_1+\delta'+\frac{(c_2+\delta)j}{I'(c_1+\delta')(\log x)^2},$$where we have used the fact that $\phi_\xi$ is well-defined in a neighborhood of $c_2$. This means 
\begin{align*}
    &\hspace{0.5cm}\sum_{j=\sqrt{x}}^\infty x^{\frac{d-1}{2}}\rho^{(\log x)^2}e^{-(\log x)^2I(\frac{m_{(\log x)^2}+j-\frac{d-1}{2c_2}\log x+\frac{3}{c_2}\log\log x}{(\log x)^2})}e^{c_2j}\\
    &\ll x^{\frac{d-1}{2}}e^{(\log x)^2(I(c_1)-I(c_1+\delta'))}\sum_{j=\sqrt{x}}^\infty e^{-\delta j}\ll x^{\frac{d-1}{2}}.
\end{align*}
\end{itemize}
Combining the above four cases with \eqref{eq:summand}, we conclude that
    $$\E\bigg[\bone_{\{Q_{n,\beta}=V_n\}}\sum_{y=0}^{m_{\wtx }+\beta}\sum_{v\in V_{\wtx }}\delta_{v,y}\bone_{H_{v,\wtx }(m_{\wtx }-y)}\bigg]\ll x^{\frac{d-1}{2}}.$$
    The remaining follows from Markov's inequality and Remark \ref{rem:no S}.

    (ii) 
 Similarly as in (i), we define a collection of independent $\{0,1\}$-valued random variables $\{\delta'_{v,y}\}_{v\in V_{(\log x)^2},\,y\in[0,x]\cap\N}$, independent from everything else, and such that
$$\p(\delta'_{v,y}=1)=\p\Big(\exists v\in V_{\wtx},\,\eta_{v,\wtx}(\wtx)\in[x-y-\frac{1}{4},x-y+\frac{1}{4}]\Big).$$
These random variables describe whether a particle located inside $[y-\frac{1}{4},y+\frac{1}{4}]$ at time $(\log x)^2$ will end up with a descendant in $[x-\frac{1}{2},x+\frac{1}{2}]$ at time $t_x$. 
Define$$U_n:=\left\{v\in V_n : \text{ for any }0\leq k\leq n,~\eta_{v,n}(k)<\frac{km_n}{n}\right\}.$$
For $u\in[2,\sqrt{n}]$ and $v\in V_n$, define the event
$$H'_{v,n}(u):=\Big\{v\in U_n,~\eta_{v,n}(n)\in[m_n-u-\frac{1}{4},m_n-u+\frac{1}{4})\Big\}.$$ It follows that, by considering particles located in $[y-\frac{1}{4},y+\frac{1}{4}]$ for $m_{(\log x)^2}-\log x\leq y\leq m_{(\log x)^2},~y\in\bZ$ at time $(\log x)^2$, 
\begin{align}
    P'_{(\log x)^2}\gst \sum_{y=m_{(\log x)^2}-\frac{d-1}{c_2}\log x}^{m_{(\log x)^2}-\frac{d-1}{2c_2}\log x}\sum_{v\in V_{(\log x)^2}}\delta'_{v,y} \bone_{H'_{v,(\log x)^2}(m_{(\log x)^2}-y)}.\label{eq:gst1}
\end{align}
We apply the second moment method to give a lower bound of the right-hand side of \eqref{eq:gst1}. Let us emphasize that the events $H'_{v,n}(u)$ and the random variables $\delta'_{v,y}$ are independent. We have
\begin{align}\begin{split}
   &\hspace{0.5cm} \E\bigg[\sum_{y=m_{(\log x)^2}-\frac{d-1}{c_2}\log x}^{m_{(\log x)^2}-\frac{d-1}{2c_2}\log x}\sum_{v\in V_{(\log x)^2}}\delta'_{v,y} \bone_{H'_{v,(\log x)^2}(m_{(\log x)^2}-y)}\bigg]\\
   &=\rho^{(\log x)^2}\sum_{y=m_{(\log x)^2}-\frac{d-1}{c_2}\log x}^{m_{(\log x)^2}-\frac{d-1}{2c_2}\log x}\p\Big(\exists v\in V_{\wtx},\,\eta_{v,\wtx}(\wtx)\in[x-y-\frac{1}{4},x-y+\frac{1}{4}]\Big)\\
   &\hspace{4cm}\times\p(H'_{v,(\log x)^2}(m_{(\log x)^2}-y)).
\end{split}\label{eq:t1}
    \end{align}
    It follows from the same argument leading to (17) in \citep{blanchet2024first} that for $u\in[2,\sqrt{n}]$, $\rho^n\p(H'_{v,n}(u))\gg ue^{\bl u}$. 
    With a change of variable $j=m_{(\log x)^2}-y$ and applying \eqref{eq:BRW max tail3} of Lemma \ref{lemma:brw small deviation prob} and \eqref{eq:mwtx}, the quantity in \eqref{eq:t1} is at least
  \begin{align*}
      &\hspace{0.5cm} \sum_{j=\frac{d-1}{2c_2}\log x}^{\frac{d-1}{c_2}\log x}\p\Big(\Big|M_{\wtx}-(m_{\wtx}+(j+\frac{3}{c_2}\log\log x-\frac{d-1}{2c_2}\log x))\Big|\leq\frac{1}{4}\Big)\\
      &\hspace{7cm}\times\rho^{(\log x)^2}\p(H'_{v,(\log x)^2}(j))\\
    &\gg \sum_{j=\frac{d-1}{2c_2}\log x}^{\frac{d-1}{c_2}\log x}\Big((j+\frac{3}{c_2}\log\log x-\frac{d-1}{2c_2}\log x)e^{-c_2(j+\frac{3}{c_2}\log\log x-\frac{d-1}{2c_2}\log x)}\Big)  (je^{c_2j})\\
    &\gg (\log x)^{-3}x^{\frac{d-1}{2}}\sum_{j=\frac{d-1}{2c_2}\log x}^{\frac{d-1}{c_2}\log x}(j-\frac{d-1}{2c_2}\log x)j\\
    &\gg x^{\frac{d-1}{2}}.
\end{align*}
We next compute the second moment of the right-hand side of \eqref{eq:gst1}. Expanding the square leads to
\begin{align*}
     &\hspace{0.5cm} \E\bigg[\bigg(\sum_{y=m_{(\log x)^2}-\frac{d-1}{c_2}\log x}^{m_{(\log x)^2}-\frac{d-1}{2c_2}\log x}\sum_{v\in V_{(\log x)^2}}\delta'_{v,y} \bone_{H'_{v,(\log x)^2}(m_{(\log x)^2}-y)}\bigg)^2\bigg]\\
     &=\sum_{s=0}^{(\log x)^2} \rho^{(\log x)^2+s}\sum_{y=m_{(\log x)^2}-\frac{d-1}{c_2}\log x}^{m_{(\log x)^2}-\frac{d-1}{2c_2}\log x} \sum_{y'=m_{(\log x)^2}-\frac{d-1}{c_2}\log x}^{m_{(\log x)^2}-\frac{d-1}{2c_2}\log x}\E[\delta'_{v,y}\delta'_{w,y'}]\\
     &\hspace{4cm}\times\p(H'_{v,(\log x)^2}(m_{(\log x)^2}-y)\cap H'_{w,(\log x)^2}(m_{(\log x)^2}-y')),
     \end{align*}
     where $v,w$ have genealogical distance equal to $2s$. Meanwhile, applying the same argument leading to (19) in \citep{blanchet2024first} gives that
     \begin{align*}
         &\sum_{s=0}^{(\log x)^2} \rho^{(\log x)^2+s}\p(H'_{v,(\log x)^2}(m_{(\log x)^2}-y)\cap H'_{w,(\log x)^2}(m_{(\log x)^2}-y'))\\
         &\hspace{3cm}\ll (m_{(\log x)^2}-y)(m_{(\log x)^2}-y')e^{c_2((m_{(\log x)^2}-y)+(m_{(\log x)^2}-y'))}.
     \end{align*}
     Combining the above steps and applying \eqref{eq:BRW max tail} of Lemma \ref{lemma:brw small deviation prob} with the change of variables $j=m_{(\log x)^2}-y,~j'=m_{(\log x)^2}-y'$, we have
     \begin{align*}
     &\hspace{0.5cm} \E\bigg[\bigg(\sum_{y=m_{(\log x)^2}-\frac{d-1}{c_2}\log x}^{m_{(\log x)^2}-\frac{d-1}{2c_2}\log x}\sum_{v\in V_{(\log x)^2}}\delta'_{v,y} \bone_{H'_{v,(\log x)^2}(m_{(\log x)^2}-y)}\bigg)^2\bigg]\\
         &\ll\sum_{j=\frac{d-1}{2c_2}\log x}^{\frac{d-1}{c_2}\log x}\sum_{j'=\frac{d-1}{2c_2}\log x}^{\frac{d-1}{c_2}\log x}\p(M_{\wtx}>m_{\wtx}+(j+\frac{3}{c_2}\log\log x-\frac{d-1}{2c_2}\log x))\\
     &\hspace{3cm}\times\p(M_{\wtx}>m_{\wtx}+(j'+\frac{3}{c_2}\log\log x-\frac{d-1}{2c_2}\log x))jj'e^{c_2(j+j')}\\
     &\ll (\log x)^{-6}x^{{d-1}}\sum_{j=\frac{d-1}{2c_2}\log x}^{\frac{d-1}{c_2}\log x}\sum_{j'=\frac{d-1}{2c_2}\log x}^{\frac{d-1}{c_2}\log x}(j+\frac{3}{c_2}\log\log x-\frac{d-1}{2c_2}\log x)\\
     &\hspace{6cm}\times(j'+\frac{3}{c_2}\log\log x-\frac{d-1}{2c_2}\log x)jj'\\
     &\ll x^{{d-1}}.
\end{align*}
    We conclude with the Paley--Zygmund inequality that there exist $C,L>0$ such that
    $$\p\Big(P'_{(\log x)^2}\geq \frac{1}{C} x^{(d-1)/2}\Big)>\frac{1}{L}.$$
    The proof is then complete in view of Remark \ref{rem:no S}.
\end{proof}

\section{Proof of the upper bound of FPT}\label{sec:UB}
The idea is rather simple: given Proposition \ref{lemma:transition}, we obtain $x^{\frac{d-1}{2}}/C$ many independent trajectories in the time period $[(\log x)^2,t_x]$ that lead to $[x-\frac{1}{2},x+\frac{1}{2}]$ at time $t_x$. It remains to argue that each of them has roughly a chance of $x^{-\frac{d-1}{2}}$ to reach $B_\z(\frac{1}{2})$ in the last $d-1$ coordinates.
To justify this claim, we need a conditional local central limit theorem, as we have already conditioned on the displacement in the first dimension of the trajectories. The proof will be deferred to Appendix \ref{sec:cond local CLT proof}. 
Let $\{\bxi_i\}_{i\in\N}$ be an i.i.d.~sequence of random vectors with the same law as $\bxi$ and consider its partial sum $\bS_n=\sum_{i=1}^n\bxi_i$. Let $\lambda(x)$ and $S_n$ be the first coordinates of $\bla(x)$ and $\bS_n$ respectively.

\begin{lemma}[conditional local CLT]\label{lemma:conditioned local CLT}Fix a large constant $L>0$. Uniformly for $\bla(x)=O((\log x)^L)$,
$$\p\Big(\bla(x)+\bS_{\wtx }\in [x-\frac{1}{2},x+\frac{1}{2}]\times B_\z(\frac{1}{2})~\Big|~ \lambda(x)+S_{\wtx }\in [x-\frac{1}{2},x+\frac{1}{2}]\Big)\asymp x^{-\frac{d-1}{2}},$$
where $B_\z(\frac{1}{2})$ is the ball of radius $1/2$ centered at $\z\in\R^{d-1}$.
\end{lemma}

\begin{proof}[Proof of the upper bound of Theorem \ref{thm:main}]
On the event $\{P'_{(\log x)^2}\geq \frac{1}{C} x^{(d-1)/2}\}\cap S$, we may label particles $\{v_j\}_{1\leq j\leq x^{(d-1)/2}/C}$ at time $(\log x)^2$ that allow for descendants $\{w_j\}_{1\leq j\leq x^{(d-1)/2}/C}$ in $[x-\frac{1}{2},x+\frac{1}{2}]$ at time $t_x$.
By a union bound and a rough large deviation estimate (recalling that $\widehat{\bet}_{v,n}(k)\in\R^{d-1}$ is the last $d-1$ coordinates of $\bet_{v,n}(k)$), 
\begin{align*}
    \p(\exists j\in\{1,\dots,x^{(d-1)/2}/C\},~\n{\widehat{\bet}_{v_j,(\log x)^2}((\log x)^2)}\geq (\log x)^3)&\ll \rho^{(\log x)^2}e^{-\delta'(\log x)^3}=o(1)
\end{align*}for some $\delta'>0$, 
and hence we may without loss of generality assume that $$\n{\widehat{\bet}_{w_j,t_x}((\log x)^2)}=\n{\widehat{\bet}_{v_j,(\log x)^2}((\log x)^2)}\leq (\log x)^3$$ for all $j$.
Let $\widetilde{\p}$ be the conditional law upon the above setting (i.e.,~on the event $\{P'_{(\log x)^2}\geq \frac{1}{C} x^{(d-1)/2}\}\cap S$, the configuration up to time $(\log x)^2$, and the event that $\eta_{w_j,t_x}(t_x)\in [x-\frac{1}{2},x+\frac{1}{2}]$ and $\n{\widehat{\bet}_{w_j,t_x}((\log x)^2)}\leq (\log x)^3$ for all $j$). 
By Lemma \ref{lemma:conditioned local CLT}, uniformly in $j$, 
\begin{align}
    \widetilde{\p}\Big(\widehat{\bet}_{w_j,t_x}(t_x)-\widehat{\bet}_{w_j,t_x}((\log x)^2)\in B_\z(\frac{1}{2})-\widehat{\bet}_{w_j,t_x}((\log x)^2)\Big)\gg x^{-\frac{d-1}{2}}.\label{eq:j}
\end{align}
 
To prove the upper bound of $\tau_x$, we show that one of the descendants of these particles $\{v_j\}$ realizes the FPT with an asymptotically positive probability. It suffices then to consider the sub-event that for some $j$, $\bet_{w_j,t_x}(t_x)\in[x-\frac{1}{2},x+\frac{1}{2}]\times B_\z(\frac{1}{2}) $.  Note that under the law $\widetilde{\p}$, the random variables $\{\widehat{\bet}_{w_j,t_x}(t_x)-\widehat{\bet}_{w_j,t_x}((\log x)^2)\}_{1\leq j\leq x^{(d-1)/2}/C}$ are independent. Therefore, by \eqref{eq:j},
\begin{align*}
    &\hspace{0.5cm}\p\Big(\tau_x\leq t_x\mid \{P'_{(\log x)^2}\geq \frac{1}{C} x^{(d-1)/2}\} \cap S\Big)\\
    &\geq \widetilde{\p}\Big(\exists j,~\widehat{\bet}_{w_j,t_x}(t_x)-\widehat{\bet}_{w_j,t_x}((\log x)^2)\in B_\z(\frac{1}{2})-\widehat{\bet}_{w_j,t_x}((\log x)^2)\Big)\\
    &=1-\prod_{1\leq j\leq x^{(d-1)/2}/C}\bigg(1-\widetilde{\p}\Big(\widehat{\bet}_{w_j,t_x}(t_x)-\widehat{\bet}_{w_j,t_x}((\log x)^2)\in B_\z(\frac{1}{2})-\widehat{\bet}_{w_j,t_x}((\log x)^2)\Big)\bigg)\\
    &\gg 1- \prod_{1\leq j\leq x^{(d-1)/2}/C}(1-x^{-\frac{d-1}{2}})\gg \frac{1}{C}.
\end{align*}
By Proposition \ref{lemma:transition},
$$\p(\tau_x\leq t_x\mid S)\gg \frac{1}{LC}.$$
Using Lemma \ref{lemma:concentration}, we conclude that for any $\ee>0$, there exists $K>0$ such that
    $$\p(\tau_x\geq t_x+K\mid S)<\ee.$$
    This completes the proof.
\end{proof}

\section{Proof of the lower bound of FPT}\label{sec:LB}
Recall \eqref{eq:txdef} and $\widetilde{t}_x=t_x-(\log x)^2$. Let us define also $\widetilde{x}:=x-m_{(\log x)^2}$. This section aims to prove that $\tau_x\geq t_x-O_\p(1)$. That is, for a fixed $\ee>0$, we find a lag time $K>0$ such that for $x$ large enough,
\begin{align}
    \p(\tau_x\leq t_x-K)<\ee.\label{eq:toshow lb}
\end{align}
In all asymptotic upper bounds below, the asymptotic constant does not depend on $K$. We use the short-hand notation $t_{x,K}:=t_x-K$. We omit the conditioning on the survival event $S$ for notational brevity in all probabilities and expectations below.

\subsection{Reducing the proof to the analysis of particles with a fixed ancestor at time \texorpdfstring{$\tilde{t}_x$}{}}\label{sec:reducing proof}

\subsubsection{The key conditioning step}
To prove the lower bound of $\tau_x$, we follow the strategy outlined in Section \ref{sec:outline proof}: classify the particles near frontier at time $\wtx$ according to the locations in the first dimension, and then analyze the chances that their descendants reach $B_x$ at time $t_{x,K}=t_x-K$ (\textit{local hitting probabilities}). Consequently, the total hitting probability of $B_x$ can be bounded from above using a first moment method, by weighting the local hitting probabilities by the density of the particles at a location near the frontier. These two key ingredients are summarized by the following two results in this subsection. 
\begin{itemize}
    \item Proposition \ref{prop:particle density} controls the desired weights (or the density of the particles near frontier).
    \item Theorem \ref{thm:particle probability} controls the local hitting probabilities.
\end{itemize}
Let $K_2$ be a large constant such that $\p(\sG_{\wtx,K_2})<\ee/2$, by Lemma \ref{lemma:barrier}. 

\begin{proposition}[density of particles at time $\wtx$]\label{prop:particle density}
    It holds that for $\ell\leq x/\log x$,
    $$\E\Big[\#\{v\in V_{\widetilde{t}_x}:\eta_{v,\widetilde{t}_x}(\widetilde{t}_x)\in[\wx-\ell-1,\wx-\ell)\}\bone_{\sG_{\wtx,K_2}^{\mathrm{c}}}\Big]\ll e^{c_2\ell}(\log x+\ell_+)x^{\frac{d-1}{2}}(\log x)^{-3}. $$
    For $\ell\geq x/\log x$, there exists $L>0$ such that
    \begin{align*}
        \E\Big[\#\{v\in V_{\widetilde{t}_x}:\eta_{v,\widetilde{t}_x}(\widetilde{t}_x)\in[\wx-\ell-1,\wx-\ell)\}&\bone_{\sG_{\wtx,K_2}^{\mathrm{c}}}\Big]\\
        &\ll e^{c_2\ell}(\log x+\ell_+)x^{\frac{d-1}{2}}(\log x)^{-3}e^{L\ell(\log x)/x}. 
    \end{align*}
    Here, we allow the constant in $\ll$ to depend on $K_2$.
\end{proposition}
\begin{proof}
    We apply the ``ballot theorem under a change of measure'' argument similarly as done in the proof of \eqref{eq:ub number} (see Proposition 8 of \citep{blanchet2024first}). 
    Before performing the change of measure, we introduce $\widehat{\lambda}=I'(m_{\wtx}/\wtx)$, which is the value of $\lambda$ where the supremum of \eqref{eq:ratef} is attained with $x=m_{\wtx}/\wtx$. By (15) of \citep{bramson2016convergence}, it holds $0\leq \bl-\widehat{\lambda}\ll (\log x)/x$. 
Let $\q$ be defined by
\begin{align}
    \frac{\d\p}{\d\q}:=e^{-\widehat{\lambda}(\eta_{v,{\wtx}}({\wtx})-m_{\wtx})-{\wtx}I(m_{\wtx}/{\wtx})}\asymp (\wtx)^{3/2}\rho^{-{\wtx}}e^{-\widehat{\lambda}(\eta_{v,{\wtx}}({\wtx})-m_{\wtx})}.\label{eq:dpdqn}
\end{align}
  It follows that under $\q$, $\{\eta_{v,{\wtx}}(k)-km_{\wtx}/{\wtx}\}_{k=0,\dots,{\wtx}}$
is a mean zero random walk. The ending location of the random walk is around
    $\wx-\ell$, which is at a distance 
    \begin{align}
        m_{\wtx}+K_2-(\wx-\ell)=\frac{d-1}{2c_2}\log x-\frac{3}{c_2}\log\log x+K_2+\ell\label{eq:distance}
    \end{align}
    below the barrier.

    We therefore have, using Lemma 2.3 of \citep{bramson2016convergence}, \eqref{eq:dpdqn}, and \eqref{eq:distance},
    \begin{align*}
        &\hspace{0.5cm}\E\big[\#\{v\in V_{\widetilde{t}_x}:\eta_{v,\widetilde{t}_x}(\widetilde{t}_x)\in[\wx-\ell-1,\wx-\ell)\}\bone_{\sG_{\wtx,K_2}^{\mathrm{c}}}\big]\\
        &\ll \rho^{\wtx}\p(\eta_{v,\widetilde{t}_x}(\widetilde{t}_x)\in[\wx-\ell-1,\wx-\ell)\cap{\sG_{\wtx,K_2}^{\mathrm{c}}})\\
        &\ll (\wtx)^{3/2}e^{-\widehat{\lambda}(\wx-\ell-m_{\wtx})}\q(\eta_{v,\widetilde{t}_x}(\widetilde{t}_x)\in[\wx-\ell-1,\wx-\ell)\cap{\sG_{\wtx,K_2}^{\mathrm{c}}})\\
        &\ll K_2\Big(1+\Big(\frac{d-1}{2c_2}\log x-\frac{3}{c_2}\log\log x+K_2+\ell\Big)_+ \Big)e^{-\bl(\wx-\ell-m_{\wtx})}e^{L\ell(\log x)/x}\\
        &\ll e^{c_2\ell}(\log x+\ell_+)x^{\frac{d-1}{2}}(\log x)^{-3}e^{L\ell(\log x)/x}.
    \end{align*}
   The last term $e^{L\ell(\log x)/x}\ll 1$ if $\ell\leq x/\log x$.  This proves the claim.
\end{proof}

It remains to fix a particle $v\in V_{\widetilde{t}_x}$ such that $\eta_{v,\widetilde{t}_x}(\widetilde{t}_x)\in[\wx-\ell-1,\wx-\ell)$ and $\sG_{\wtx,K_2}^{\mathrm{c}}$ hold (recall \eqref{eq:barrier}), and bound the probability of finding a descendant $w\in V_{t_{x,K}}$ of $v$ with $\bet_{w,t_{x,K}}(t_{x,K})\in B_x$, which is the task of the next theorem. We need a few preparations before its statement. 

In the following, we use 
$\q^{\ell,v}$ to denote the probability measure on the BRW restricted to the descendants of $v$ (i.e., the sub-tree with root $v$ and we implicitly recognize $v$ as the common ancestor), conditioning on $\eta_{v,\widetilde{t}_x}(\widetilde{t}_x)\in[\wx-\ell-1,\wx-\ell)$ and $\sG_{\wtx,K_2}^{\mathrm{c}}$.\footnote{Note that the location $\bet_{v,\wtx}(\wtx)$ is not independent from the event $\sG_{\wtx,K_2}^{\mathrm{c}}$.} Let $K_3,K_6,K_8$ be large constants to be determined later in the proof. Recall \eqref{eq:phind}. 
 For future use, we consider a large constant $L$ to be determined and define the auxiliary function
\begin{align}
    \Psi_{(\log x)^2,\delta}(\ell):=\begin{cases}
        \ell&\text{ if }\ell<\frac{\log x}{L};\\
        (\log x)^{1/3}+\ell e^{-\frac{\delta\ell^2}{(\log x)^2}}&\text{ if }\frac{\log x}{L}\leq\ell\leq L\log x\log\log x;\\
        \varphi_{(\log x)^2,\delta}(\ell)&\text{ if }\ell> L\log x\log\log x.
    \end{cases}\label{eq:Psi}
\end{align}

Next, we define the key quantity $I_{\ell,x}$, which serves as an upper bound of the local hitting probabilities. We define $I_{\ell,x}$ according to different ranges of $\ell$ as follows.
\begin{itemize}
    \item If $\ell< -K_3\log\log x$,
$$I_{\ell,x}:= (\log x)^{2(d-1)}x^{-\frac{d-1}{2}}.$$
\item If $-K_3{\log\log x}\leq\ell\leq K_6\log\log x$,
$$I_{\ell,x}:= C(\ee,K) (\log\log x)^{K_7}(\log x)x^{-\frac{d-1}{2}} e^{-c_2\ell}.$$
\item If $\ell> K_6\log\log x$, 
\begin{align*}
   & I_{\ell,x}:= \ee x^{-\frac{d-1}{2}}e^{-c_2\ell}\ell \varphi_{(\log x)^2,\delta}(\ell)\\
   &\hspace{2cm}+ x^{-\frac{d-1}{2}}\ell e^{-c_2\ell}\varphi_{(\log x)^2,\delta}(\ell)(K^{d-1-K_8}+\ell^{-K_8/2}(\log x)^{2d})\\
   &\hspace{2cm}+C(\ee,K)e^{-(c_2+\delta/4)\ell}(\log x)^{3d}x^{-\frac{d-1}{2}}\\
    &\hspace{2cm}+C(\ee)e^{-c_1c_2K/4} x^{-\frac{d-1}{2}} e^{-c_2\ell}\Psi_{(\log x)^2,\delta}(\ell)\\
    &\hspace{5cm}\times\big(e^{2K_{10}\frac{\ell\log\log x}{(\log x)^2}}+(\log\log x)^{d+3}\ell^{-1/8}e^{\frac{\log\ell\log\log x}{8\ell}}\big).
\end{align*}
Here, $C(\ee),\,C(\ee,K)$ are constants to be determined, which may also depend on the constants $K_3,K_6,K_8$ but not on $x,\ell$. 
\end{itemize}

\begin{theorem}[first passage contributions of particles located in $[\wx-\ell-1,\wx-\ell)\times\R^{d-1}$ at time $\wtx$]\label{thm:particle probability}
Let $\ee>0$. Then there exist $C(\ee),C(\ee,K)>0$ such that for all $\ell\in\bZ$ and $x$ large enough,
\begin{align*}\q^{\ell,v}(\exists\, w\in V_{t_{x,K}},\,w\succ v,\,\bet_{w,t_{x,K}}(t_{x,K})\in B_x)\ll I_{\ell,x}.
\end{align*}
Here, the implicit constant in $\ll$ may depend on $K_3,K_6,K_8$ but does not depend on $\ee$ or $K$.
\end{theorem}

The proof of Theorem \ref{thm:particle probability} is deferred till later and takes up the majority of the remaining of this paper. In the next subsection, we finish the proof of the lower bound of the first passage time $\tau_x$ assuming Theorem \ref{thm:particle probability} holds.

\subsubsection{Proof of the lower bound of FPT}\label{sec:412}

Before proceeding to the proof, we need Lemma \ref{lemma:escape} below, which asserts that there exists $K_1>0$ such that
    $$\p(\n{\bet_{v,n}(n)}\geq 1\text{ for all }v\in V_n)\leq K_1e^{-\sqrt{n}/K_1}.$$
In other words, it is unlikely that all particles at generation $n$ appear outside the unit ball for $n$ large. In this way, we reduce the proof to finding an upper bound on $\p(\exists\, w\in V_{t_{x,K}},\,\bet_{w,t_{x,K}}(t_{x,K})\in B_x)$.


\begin{proof}[Proof of the lower bound of Theorem \ref{thm:main}]
With Lemma \ref{lemma:escape}, it suffices to bound from above the probability of finding a particle in $B_x$ at time $t_{x,K}$ (because on the complement of such an event, Lemma \ref{lemma:escape} shows that the probability that the FPT is smaller decays at least quasi-exponentially; see e.g.~the proof of Theorem 1 of \citep{zhang2024modeling}). The first step is to impose the global barrier constraint $\sG_{\wtx,K_2}^{\mathrm{c}}$(see \eqref{eq:gnb} for its definition). We have
\begin{align}\begin{split}
    &\p(\exists\, w\in V_{t_{x,K}},\,\bet_{w,t_{x,K}}(t_{x,K})\in B_x)\\
    &\hspace{3cm}\leq \p(\sG_{\wtx,K_2})+\p(\exists\, w\in V_{t_{x,K}},\,\bet_{w,t_{x,K}}(t_{x,K})\in B_x,\,\sG_{\wtx,K_2}^{\mathrm{c}}).
\end{split}\label{eq:3probs}
\end{align}
In the following, we use $w\succ v$ to denote that the particle $w$ is a descendant of $v$. 
 Using the first moment method, we have
\begin{align*}
    &\hspace{0.5cm}\p(\exists\, w\in V_{t_{x,K}},\,\bet_{w,t_{x,K}}(t_{x,K})\in B_x,\,\sG_{\wtx,K_2}^{\mathrm{c}})\\
    &= \p(\exists\, v\in V_{\wtx},w\in V_{t_{x,K}},w\succ v,\,\bet_{w,t_{x,K}}(t_{x,K})\in B_x,\,\sG_{\wtx,K_2}^{\mathrm{c}})\\
    &\leq \E\bigg[\bone_{\sG_{\wtx,K_2}^{\mathrm{c}}}\sum_{v\in V_{\wtx}}\bone_{\{\exists\, w\in V_{t_{x,K}},w\succ v,\,\bet_{w,t_{x,K}}(t_{x,K})\in B_x\}}\bigg]\\
&=\sum_{\ell\in\bZ}\E\bigg[\E\Big[\bone_{\sG_{\wtx,K_2}^{\mathrm{c}}}\sum_{v\in V_{\wtx}}\bone_{\{\eta_{v,\widetilde{t}_x}(\widetilde{t}_x)\in[\wx-\ell-1,\wx-\ell)\}}\bone_{\{\exists\, w\in V_{t_{x,K}},w\succ v,\,\bet_{w,t_{x,K}}(t_{x,K})\in B_x\}}\\
&\hspace{5cm}\mid \sG_{\wtx,K_2}^{\mathrm{c}},\,\{\eta_{v,\widetilde{t}_x}(\widetilde{t}_x)\in[\wx-\ell-1,\wx-\ell)\}\Big]\bigg]\\
    &= \sum_{\ell\in\bZ}\E\bigg[\#\{v\in V_{\widetilde{t}_x}:\eta_{v,\widetilde{t}_x}(\widetilde{t}_x)\in[\wx-\ell-1,\wx-\ell)\}\bone_{\sG_{\wtx,K_2}^{\mathrm{c}}}\\
    &\hspace{3cm}\times\q^{\ell,v}(\exists\, w\in V_{t_{x,K}},\,w\succ v,\,\bet_{w,t_{x,K}}(t_{x,K})\in B_x)\bigg].
    \end{align*}
Note that the inner probability is bounded by the deterministic term $I_{\ell,x}$ by Theorem \ref{thm:particle probability}. Applying also Proposition \ref{prop:particle density}, we get
    \begin{align*}
    &\hspace{0.5cm}\p(\exists\, w\in V_{t_{x,K}},\,\bet_{w,t_{x,K}}(t_{x,K})\in B_x,\,\sG_{\wtx,K_2}^{\mathrm{c}})\\
    &\ll \sum_{\ell\in\bZ}\E\Big[\#\{v\in V_{\widetilde{t}_x}:\eta_{v,\widetilde{t}_x}(\widetilde{t}_x)\in[\wx-\ell-1,\wx-\ell)\}\bone_{\sG_{\wtx,K_2}^{\mathrm{c}}}\Big]I_{\ell,x}\\
    &\ll \sum_{\ell\in\bZ}(e^{c_2\ell}(\log x+\ell_+)x^{\frac{d-1}{2}}(\log x)^{-3}e^{L\ell(\log x)/x})I_{\ell,x}.
\end{align*} Inserting the definition of $I_{\ell,x}$ yields the following upper bound on the above quantity:
\begin{align*}
  &\hspace{0.5cm}  \sum_{\ell<-K_3\log\log x}(e^{c_2\ell}(\log x+\ell_+)x^{\frac{d-1}{2}}(\log x)^{-3})((\log x)^{2(d-1)}x^{-\frac{d-1}{2}})\\
  &+\sum_{-K_3\log\log x\leq \ell\leq K_6\log\log x}(e^{c_2\ell}(\log x+\ell_+)x^{\frac{d-1}{2}}(\log x)^{-3})\\
  &\hspace{3cm}\times(C(\ee,K)(\log\log x)^{K_7}(\log x)x^{-\frac{d-1}{2}} e^{-c_2\ell})\\
  &+\sum_{\ell> K_6\log\log x}(e^{c_2\ell}(\log x+\ell)x^{\frac{d-1}{2}}(\log x)^{-3}e^{L\ell(\log x)/x})\\
  &\hspace{0.4cm}\times\Big(\ee x^{-\frac{d-1}{2}}e^{-c_2\ell}\ell\varphi_{(\log x)^2,\delta}(\ell)+ x^{-\frac{d-1}{2}}\ell e^{-c_2\ell}\varphi_{(\log x)^2,\delta}(\ell)\big(K^{d-1-K_8}+\ell^{-K_8/2}(\log x)^{2d}\big)\\
   &\hspace{1.2cm}+C(\ee,K)e^{-(c_2+\delta/4)\ell}(\log x)^{3d}x^{-\frac{d-1}{2}}\\
    &\hspace{1.2cm}+C(\ee)e^{-c_1c_2K/4} x^{-\frac{d-1}{2}} e^{-c_2\ell}\Psi_{(\log x)^2,\delta}(\ell)\\
    &\hspace{3cm}\times\big(e^{2K_{10}\frac{\ell\log\log x}{(\log x)^2}}+(\log\log x)^{d+3}\ell^{-1/8}e^{\frac{\log\ell\log\log x}{8\ell}}\big)\Big)\\
  &=:I_1+I_2+I_3.
\end{align*}
Our goal is to show that $I_1+I_2+I_3$ can be made arbitrarily small as $\ee\to 0,\,K\to\infty,$ and $x\to\infty$ (in order).

We next estimate the sums $I_1,I_2,I_3$. With $K_3>0$ picked large enough, the term $I_1$ can be controlled using
 $$I_1\ll \sum_{\ell<-K_3\log\log x}e^{c_2\ell}(\log x)^{2(d-2)}=o(1).$$
For $I_2$, we have
\begin{align*}
    I_2
    &\ll C(\ee,K)(\log\log x)^{K_7+1}(\log x)^{-2}= o(1),
\end{align*}where by convention, the $o(1)$ term converges to $0$ as $x\to\infty$, with rate possibly depending on $\ee,K$. 

For $I_3$, we further decompose into two parts. First,
\begin{align*}
   &\hspace{0.5cm} \sum_{\ell> K_6\log\log x}(e^{c_2\ell}(\log x+\ell)x^{\frac{d-1}{2}}(\log x)^{-3}e^{L\ell(\log x)/x})\\
  &\hspace{0.6cm}\times\Big(\ee x^{-\frac{d-1}{2}}e^{-c_2\ell}\ell\varphi_{(\log x)^2,\delta}(\ell)+ x^{-\frac{d-1}{2}}\ell e^{-c_2\ell}\varphi_{(\log x)^2,\delta}(\ell)\big(K^{d-1-K_8}+\ell^{-K_8/2}(\log x)^{2d}\big)\Big)\\
  &= (\log x)^{-3}\sum_{\ell> K_6\log\log x}\ell(\log x+\ell)\varphi_{(\log x)^2,\delta}(\ell)(\ee+K^{d-1-K_8}+\ell^{-K_8/2}(\log x)^{2d})e^{L\ell(\log x)/x}.
\end{align*}
With $K_8$ picked large enough, the above is $\ll \ee+o(1)$. Second, we have (the case $\ell\geq x/\log x$ being almost identical as above, we remove the term $e^{L\ell(\log x)/x}$ for brevity)
\begin{align*}
    &\hspace{0.5cm}\sum_{\ell> K_6\log\log x}(e^{c_2\ell}(\log x+\ell)x^{\frac{d-1}{2}}(\log x)^{-3})\times \Big(C(\ee,K)e^{-(c_2+\delta/4)\ell}(\log x)^{3d}x^{-\frac{d-1}{2}}\\
  &\hspace{0.4cm}+C(\ee)e^{-c_1c_2K/4} x^{-\frac{d-1}{2}} e^{-c_2\ell}\Psi_{(\log x)^2,\delta}(\ell)\big(e^{2K_{10}\frac{\ell\log\log x}{(\log x)^2}}+(\log\log x)^{d+3}\ell^{-1/8}e^{\frac{\log\ell\log\log x}{8\ell}}\big)\Big)\\
    &\ll C(\ee,K)(\log x)^{3d-3}\sum_{\ell> K_6\log\log x}e^{-\delta\ell/4}(\log x+\ell)\\
    &\hspace{1cm}+C(\ee)(\log x)^{-3}e^{-c_1c_2K/4}\sum_{\ell> K_6\log\log x}(\log x+\ell)\\
    &\hspace{4cm}\times\big(e^{2K_{10}\frac{\ell\log\log x}{(\log x)^2}}+(\log\log x)^{d+3}\ell^{-1/8}e^{\frac{\log\ell\log\log x}{8\ell}}\big)\Psi_{(\log x)^2,\delta}(\ell)\\
    &\ll C(\ee,K)(\log x)^{3d-2-K_6\delta/5}\\
    &\hspace{0.3cm} +C(\ee)(\log x)^{-3}e^{-c_1c_2K/4}\bigg(\sum_{\ell> K_6\log\log x}(\log x+\ell)\Psi_{(\log x)^2,\delta}(\ell)e^{2K_{10}\frac{\ell\log\log x}{(\log x)^2}}\\
    &\hspace{2.5cm}+(\log\log x)^{d+3}\sum_{\ell> K_6\log\log x}\ell^{-1/8}(\log x+\ell)\Psi_{(\log x)^2,\delta}(\ell)e^{\frac{\log\ell\log\log x}{8\ell}}\bigg).
\end{align*}
If $K_6$ is large enough, the first term can be controlled by $o(1)$. On the other hand, for the quantity inside the last bracket, we apply \eqref{eq:Psi} and observe that $e^{2K_{10}\frac{\ell\log\log x}{(\log x)^2}}\ll 1$ for $\ell\ll \log x\log\log x$,  $e^{\frac{\log\ell\log\log x}{8\ell}}\ll 1$ for $\ell\gg \log x$, and $e^{\frac{\log\ell\log\log x}{8\ell}}\ll \ell^{1/16}$ for $\ell> K_6\log\log x$ and $K_6$ picked large enough. Using standard integral approximations and changes of variables, we have the following upper bound:
\begin{align*}
    &\hspace{0.5cm}\sum_{\ell> K_6\log\log x}(\log x+\ell)\Psi_{(\log x)^2,\delta}(\ell)e^{2K_{10}\frac{\ell\log\log x}{(\log x)^2}}\\
    &\hspace{3cm}+(\log\log x)^{d+3}\sum_{\ell> K_6\log\log x}\ell^{-1/8}(\log x+\ell)\Psi_{(\log x)^2,\delta}(\ell)e^{\frac{\log\ell\log\log x}{8\ell}}\\
    &\ll \int_{K_6\log\log x}^{(\log x)/L} (\log x+y)y\d y\\
    &\hspace{3cm}+\int_{(\log x)/L}^{L\log x\log\log x} \Big((\log x+y)(\log x)^{1/3}+y(\log x+y)e^{-\frac{\delta y^2}{(\log x)^2}}\Big)\,\d y\\
    &\hspace{3cm}+\int_{L\log x\log\log x}^\infty \varphi_{(\log x)^2,\delta}(y)e^{2K_{10}\frac{y\log\log x}{(\log x)^2}}\,\d y\\
    &\hspace{0.8cm}+(\log\log x)^{d+3}\bigg(\int_{K_6\log\log x}^{(\log x)/L} (\log x+y)y^{7/8}e^{\frac{\log y\log\log x}{8y}}\d y\\
    &\hspace{1.5cm}+\int_{L\log x\log\log x}^\infty y^{-1/8}\varphi_{(\log x)^2,\delta}(y)\,\d y\\
    &\hspace{1.5cm}+\int_{(\log x)/L}^{L\log x\log\log x} \Big((\log x+y)(\log x)^{1/3}y^{-1/8}+y^{7/8}(\log x+y)e^{-\frac{\delta y^2}{(\log x)^2}}\Big)\,\d y\bigg)\\
    &\ll (\log x)^3+(\log x)^{8/3}+(\log\log x)^{d+3}(\log x)^{47/16}\\
    &\ll (\log x)^3.
\end{align*}
Combining the estimates above, we arrive at
$$I_3\ll o(1)+\ee+C(\ee)e^{-c_1c_2K/4}.$$
Altogether, we have
\begin{align}
    \begin{split}
         \p(\exists\, w\in V_{t_{x,K}},\,\bet_{w,t_{x,K}}(t_{x,K})\in B_x,\,\sG_{\wtx,K_2}^{\mathrm{c}})\ll I_1+I_2+I_3\ll o(1)+\ee+C(\ee)e^{-c_1c_2K/4}.
    \end{split}\label{eq:2nd prob to show}
\end{align}Recall that the $o(1)$ term may depend on $\ee,K$ but the implicit constant in $\ll$ does not depend on $\ee,K,x$. Therefore, \eqref{eq:2nd prob to show} can be made arbitrarily small by picking in order $\ee$ small enough, $K$ large enough, and then $x\to\infty$. 
On the other hand, $\p(\sG_{\wtx,K_2})$ can be bounded using Lemma \ref{lemma:barrier}. Combining with \eqref{eq:3probs} shows the desired lower bound \eqref{eq:toshow lb} of $\tau_x$.
\end{proof}

\subsection{Proof of Theorem \texorpdfstring{\ref{thm:particle probability}}{}}

The goal of this section is to prove Theorem \ref{thm:particle probability}. We start with Section \ref{sec:setting stages} that introduces the setup. In the subsections below, we discuss separately the three cases:
\begin{itemize}
    \item $\ell<-K_3{\log\log x}$ (Section \ref{sec:422});
    \item $-K_3{\log\log x}\leq\ell\leq K_6\log\log x$ (Sections \ref{sec:Local barrier events} and \ref{sec:prelim bound});
    \item $\ell> K_6\log\log x$ (Sections \ref{sec:Excluding unlikely events} and \ref{sec:426}).
\end{itemize}
   Although the three cases are disjoint, the different subsections are not --- the latter cases may occasionally apply the estimates from the former ones that may deal with more general cases.

\subsubsection{Setting up stages}\label{sec:setting stages}

Recall that the law $\q^{\ell,v}$ denotes the probability on the BRW restricted to the descendants of $v$, conditioning on $\eta_{v,\widetilde{t}_x}(\widetilde{t}_x)\in[\wx-\ell-1,\wx-\ell)$ and $\sG_{\wtx,K_2}^{\mathrm{c}}$.
Our goal is to give an upper bound for the local hitting probability
\begin{align}
    \q^{\ell,v}(\exists\, w\in V_{t_{x,K}},w\succ v,\,\bet_{w,t_{x,K}}(t_{x,K})\in B_x)\label{eq:q}
\end{align} for different ranges of $\ell$, as a function of $\ell,x$.

If we ignore the $d-1$ dimensions, this is exactly the probability that BRW reaches a distance $m_{(\log x)^2}+\ell$ in time $(\log x)^2$. When we add the extra dimensions, we need to condition on the extra event of the displacements along the first coordinate of $v$ for the previous $\wtx$ steps. We first introduce the necessary settings required for the proof of Theorem \ref{thm:particle probability}.

For a particle $v\in V_{\wtx}$, we define the \textit{heterogeneity index} $h_v$ of $v$ as the largest number of uncommon generations for two descendants of $v$ that reach $\H_x$ at time $t_{x,K}$. In other words, 
$$h_v=\max\{h\geq K:\,\exists\, v_1\in V_{t_x-h},v_2,v_3\in V_{t_{x,K}}, v_2,v_3\succ v_1\succ v, \eta_{v_i,t_{x,K}}(t_{x,K})\geq x, i=2,3 \}.$$
In the case where there is at most one descendant, we let $h_v=K$. 
It follows that $h_v\geq K$ and there exists a unique \textit{latest common ancestor} (lca) at level $t_x-h_v$ of those who reach $\H_x$ at time $t_{x,K}$. We denote that latest common ancestor by $v_{\mathrm{lca}}$. We also define the \textit{heterogeneity location} $g_v$ of $v$ such that $$\eta_{v_{\mathrm{lca}},t_x-h_v}(t_x-h_v)=x+g_v-m_{h_v-K},$$ and the \textit{multi-dimensional heterogeneity location} $\bu_{v_{\mathrm{lca}}}\in\R^{d-1}$ of $v$ such that $$\bet_{v_{\mathrm{lca}},t_x-h_v}(t_x-h_v)=(x+g_v-m_{h_v-K},\bu_{v_{\mathrm{lca}}}).$$ Let $\bu_v\in\R^{d-1}$ denote the multi-dimensional location of $v$, i.e., $\widehat{\bet}_{v,\wtx}(\wtx)=\bu_v$. The quantities $h_v,g_v,\bu_v$ can all be viewed as random variables under $\q^{\ell,v}$.    
For a vector $\bu=(u_1,\dots,u_{d-1})\in\R^{d-1}$, we let $R_\bu$ be the $(d-1)$-dimensional rectangle $[u_1,u_1+1)\times\dots\times[u_{d-1},u_{d-1}+1)$.
For $h=K,\dots,(\log x)^2$, $g\in\bZ$, and $\bu\in\bZ^{d-1}$, we define the events
\begin{align*}
    \sB_h&=\{h_v=h\},\\
    \sC_g&=\{g_v\in [g,g+1)\},\\
    \sD_\bu&=\{\bu_{v_{\mathrm{lca}}}\in R_\bu\},\\
    \sH_\bu&=\{\bu_{v_{\mathrm{lca}}}-\bu_{v}\in R_\bu\}.
\end{align*}
 All these definitions will be relevant throughout Section \ref{sec:LB}. See Figure \ref{fig:lhg} for an illustration of these parameters.

Since we will partition the probability space into a union of events of the form $\sB_h\cap\sC_g$, it is essential to bound the probabilities of those events under $\q^{\ell,v}$, which is given by the next result.

\begin{lemma}[size of the event $\sB_h\cap\sC_g$]\label{lemma:hg|l}
    It holds that 
    \begin{align*}
&\hspace{0.5cm}\q^{\ell,v}(\sB_h\cap\sC_g)\\
&\ll \min\Big\{1,(|g+\ell|+1)e^{-c_2(g+\ell+c_1K/2)}\varphi_{(\log x)^2,\delta}(g+\ell)\Big\}\min\{1,((|g|+1)e^{c_2g})^2\}.
\end{align*}
\end{lemma}
\begin{proof}
    Let $\sJ_{h,g}$ denote the event that there exists $w\in V_{t_x-h},\,w\succ v$, such that $\eta_{w,t_x-h}(t_x-h)\in[x-m_{h-K}+g-1,x-m_{h-K}+g)$. In particular, the displacement of the BRW initiated by $v$ is at least (by concavity of the logarithm and assuming $K$ is large enough)
    $$m_{(\log x)^2}-m_{h-K}+\ell+g\geq m_{(\log x)^2-h}+\ell+g+\frac{c_1K}{2}.$$
    Let $\sI_{n,g}$ denote the event that there exist two descendants in time $n$ running above $m_n-g$ with a common ancestor only at time $0$, i.e.,
    \begin{align}
    \sI_{n,g}:=\Big\{\exists\, v,w\in V_n,\,\mathrm{lca}(v,w)=\emptyset,\,\eta_{v,n}(n)\geq m_n-g,\,\eta_{w,n}(n)\geq m_n-g\Big\},~n\in\N_0,\label{eq:Ing}
\end{align}
where $\emptyset$ denotes the unique particle at time zero of the BRW.
    For the event $\sB_h\cap\sC_g$ to hold, the event $\sJ_{h,g}$ must hold and if $v_{\mathrm{lca}}$ denotes the latest common ancestor at time $t_x-h$, the sub-BRW with root $v_{\mathrm{lca}}$ satisfies the event $\sI_{h-K,g}$. 
    By independence of the process before and after time $t_x-h$, Lemma \ref{lemma:brw small deviation prob}, and Lemma \ref{lemma:double prob}, we conclude that
    \begin{align*}
        &\hspace{0.5cm}\q^{\ell,v}(\sB_h\cap\sC_g)\\
        &\leq \q^{\ell,v}(\sJ_{h,g})\,\p(\sI_{h-K,g})\\
        &\ll \min\Big\{1,(|g+\ell|+1)e^{-c_2(g+\ell+c_1K/2)}\varphi_{(\log x)^2,\delta}(g+\ell)\Big\}\min\{1,((|g|+1)e^{c_2g})^2\},
    \end{align*}
    where we have also used the fact that a probability is trivially bounded by one.
\end{proof}

\subsubsection{A uniform conditional probability bound and local hitting probabilities for \texorpdfstring{$\ell<-K_3{\log\log x}$}{}}\label{sec:422}
   

This section aims to prove the following result that serves as a general conditional bound for \eqref{eq:q}. As a consequence, it yields an upper bound of the local hitting probabilities for $\ell<-K_3{\log\log x}$ that suffices for our purpose.

\begin{lemma}[uniform conditional probability bound]\label{lemma:uniform bd}
It holds for $|\ell|\ll x^{1/3}$, $h=K,\dots,(\log x)^2$, and $g\in\R$ that
    $$\q^{\ell,v}(\exists\, w\in V_{t_{x,K}},w\succ v,\,\bet_{w,t_{x,K}}(t_{x,K})\in B_x\mid\sB_h\cap\sC_g)\ll  {h^{d-1}}{x^{-\frac{d-1}{2}}}.$$
\end{lemma}

An immediate consequence of Lemma \ref{lemma:uniform bd} is that 
$$\q^{\ell,v}(\exists\, w\in V_{t_{x,K}},\,w\succ v,\,\bet_{w,t_{x,K}}(t_{x,K})\in B_x)\ll (\log x)^{2(d-1)}x^{-\frac{d-1}{2}}.$$
In particular, the estimate in Theorem \ref{thm:particle probability} for $\ell<-K_3\log\log x$ holds. 
We need a few preparations to prove Lemma \ref{lemma:uniform bd}. 
\begin{lemma}[uniform local upper limit theorem]\label{lemma:uniform local CLT}
Uniformly in $|\ell|\ll x^{1/3}$ and $\bu\in\R^{d-1}$,
    $$\q^{\ell,v}(\bu_v\in   R_\bu  )\ll x^{-\frac{d-1}{2}}.$$
\end{lemma}

\begin{proof}

Recall that under the law $\q^{\ell,v}$, we condition on $\eta_{v,\widetilde{t}_x}(\widetilde{t}_x)\in[\wx-\ell-1,\wx-\ell)$ and $\sG_{\wtx,K_2}^{\mathrm{c}}$. In particular, the latter event means that the trajectory $\{\eta_{v,\wtx}(k)\}_{1\leq k\leq \wtx}$ is bounded from above by the curve
$$\psi(k):=\frac{km_{\wtx}}{\wtx}+K_2+\frac{6}{c_2}(\log\min\{k,\wtx-k\})_+,~1\leq k\leq \wtx.$$
Denote by $E_{\psi,x}$ the corresponding event that $S_k\leq \psi(k)$ for $1\leq k\leq \wtx$. 
An equivalent formulation of the statement is that uniformly in $|\ell|\ll x^{1/3}$ and $\bu\in\R^{d-1}$,
$$\p(\bS_{\wtx}\in [\wx-\ell-1,\wx-\ell)\times R_\bu\mid S_{\wtx}\in [\wx-\ell-1,\wx-\ell),\,E_{\psi,x})\ll x^{-\frac{d-1}{2}}.$$
The left-hand side probability can be written as 
\begin{align*}
    \p(\bS_{\wtx}\in [\wx-\ell-1,\wx-\ell)\times &R_\bu\mid S_{\wtx}\in [\wx-\ell-1,\wx-\ell),\,E_{\psi,x})\\
    &=\frac{\p(\bS_{\wtx}\in [\wx-\ell-1,\wx-\ell)\times R_\bu,\,E_{\psi,x})}{\p(S_{\wtx}\in [\wx-\ell-1,\wx-\ell),\,E_{\psi,x})}.
\end{align*}

We next apply a change of measure argument similarly as in the proof of Proposition \ref{prop:particle density}, whose notation we follow. Consider the measure $\q$ defined by
\eqref{eq:dpdqn}. 
It follows that under $\q$, the jump $\bxi$ is i.i.d.~with mean $(m_{\wtx}/\wtx,\z)\in\R^d$, and consequently, the random walk $\{\widetilde{\bS}_n\}:=\{{\bS}_n-(m_{\wtx}n/\wtx,\z)\}$ is centered. It follows that
\begin{align*}
    &\hspace{0.5cm}\frac{\p(\bS_{\wtx}\in [\wx-\ell-1,\wx-\ell)\times R_\bu,\,E_{\psi,x})}{\p(S_{\wtx}\in [\wx-\ell-1,\wx-\ell),\,E_{\psi,x})}\\
    &\asymp \frac{(\wtx)^{3/2}e^{-\widehat{\lambda}(\wx-\ell-m_{\wtx})}\q(\bS_{\wtx}\in [\wx-\ell-1,\wx-\ell)\times R_\bu,\,E_{\psi,x})}{(\wtx)^{3/2}e^{-\widehat{\lambda}(\wx-\ell-m_{\wtx})}\q(S_{\wtx}\in [\wx-\ell-1,\wx-\ell),\,E_{\psi,x})}\\
    &=\frac{\q(\widetilde{\bS}_{\wtx}\in [\wx-\ell-1-m_{\wtx},\wx-\ell-m_{\wtx})\times R_\bu,\,E_{\psi,x})}{\q(\widetilde{S}_{\wtx}\in [\wx-\ell-1-m_{\wtx},\wx-\ell-m_{\wtx}),\,E_{\psi,x})},
\end{align*}
where we note that $$E_{\psi,x}=\{S_k\leq \psi(k)\text{ for all }1\leq k\leq \wtx\}=\Big\{\widetilde{S}_k\leq \psi(k)-\frac{km_{\wtx}}{\wtx}\text{ for all }1\leq k\leq \wtx\Big\}.$$

Using (9) and (10) of Lemma 2.3 of \citep{bramson2016convergence}, the denominator has the lower bound 
$$\q(\widetilde{S}_{\wtx}\in [\wx-\ell-1-m_{\wtx},\wx-\ell-m_{\wtx}),\,E_{\psi,x})\gg (1+K_2-(\wx-\ell-m_{\wtx}))x^{-3/2}.$$
Similarly, by Lemma \ref{thm:d-dim ballot with log},
$$\q(\widetilde{\bS}_{\wtx}\in [\wx-\ell-1-m_{\wtx},\wx-\ell-m_{\wtx})\times R_\bu,\,E_{\psi,x})\ll (1+K_2-(\wx-\ell-m_{\wtx}))x^{-(d+2)/2}.$$
  Note that here we allow the asymptotic constants to depend on $K_2$.  Therefore,
    $$\frac{\p(\bS_{\wtx}\in [\wx-\ell-1,\wx-\ell)\times R_\bu,\,E_{\psi,x})}{\p(S_{\wtx}\in [\wx-\ell-1,\wx-\ell),\,E_{\psi,x})}\ll x^{-\frac{d-1}{2}},$$
    as desired.
\end{proof}

\begin{lemma}[uniform local upper limit theorem for an independent sum]\label{lemma:indsum}
    For any $x>0$, suppose that $\bzeta_1(x),\bzeta_2(x)$ are two independent random variables in $\R^{d-1}$ such that the law of $\bzeta_1(x)$ satisfies that for any $\bv\in\R^{d-1}$,
    $$\p(\bzeta_1(x)\in R_\bv)\ll x^{-\frac{d-1}{2}}.$$
    Then for any $\bu\in\R^{d-1}$,
    $$\p(\bzeta_1(x)+\bzeta_2(x)\in R_\bu)\ll x^{-\frac{d-1}{2}}.$$
\end{lemma}
\begin{proof}
    It follows from independence and elementary geometry that
    \begin{align*}
        \p(\bzeta_1(x)+\bzeta_2(x)\in R_\bu)&\leq \sum_{\bv\in \bZ^{d-1}}\p(\bzeta_1(x)\in R_\bv)\,\p(\bzeta_2(x)\in R_{\bu}-R_{\bv})\\
        &\ll x^{-\frac{d-1}{2}}\sum_{\bv\in \bZ^{d-1}}\p(\bzeta_2(x)\in R_{\bu}-R_{\bv})\ll x^{-\frac{d-1}{2}}.
    \end{align*}
    This completes the proof.
\end{proof}

\begin{lemma}[Markov property]\label{lemma:independence}
It holds for all $\ell,h,g$ that under law $\q^{\ell,v}$,
$$\bu_v  \dd \bu_v\mid(\sB_h,\sC_g)\dd\bu_v\mid(\sB_h,\sC_g,\sigma(\bu_{v_{\mathrm{lca}}}-\bu_v)).$$
\end{lemma}

\begin{proof}
    By definition of the events, we have the following diagram of dependence:
    \begin{figure}[H]
    \centering
    \begin{tikzcd}[row sep=2em, column sep = 4em]
\sigma(\eta_{v,\wtx}(\wtx)) \arrow{r} \arrow{d}& \cF^{(1)}_{v,\wtx,t_{x,K}}  \arrow{d}\\
\sigma(\bet_{v,\wtx}(\wtx)) & \cF^{(d-1)}_{v,\wtx,t_{x,K}}
\end{tikzcd}
    \caption{Dependence relation of the $\sigma$-algebras. Arrows indicate that under the probability measure induced by the BRW, any random variable that is measurable with respect to the object being pointed to (i.e., the head of the arrow) can be simulated as a function of random variables measurable with respect to the object pointing from (i.e., the tail of the arrow) along with some independent randomness. Here, $\cF^{(1)}_{v,\wtx,t_{x,K}}$ denotes the $\sigma$-algebra generated by the first coordinate of the tree rooted at $v$ (with time scale $[\wtx,t_{x,K}]$), and $\cF^{(d-1)}_{v,\wtx,t_{x,K}}$ denotes the $\sigma$-algebra generated by the last $d-1$ coordinates of the tree rooted at $v$ (again with time scale $[\wtx,t_{x,K}]$).}
    \label{fig:commute}
\end{figure}
Note that $\sigma(\eta_{v,\wtx}(\wtx)) $ is trivial under $\q^{\ell,v}$, $\bu_v$ is $\sigma(\bet_{v,\wtx}(\wtx))$-measurable, $\sB_h\cap\sC_g$ is $\cF^{(1)}_{v,\wtx,t_{x,K}}$-measurable, and $\bu_{v_{\mathrm{lca}}}-\bu_v$ is $\cF^{(d-1)}_{v,\wtx,t_{x,K}}$-measurable. Therefore, on the law $\q^{\ell,v}$, the random variable $\bu_v$ is independent from  $\sB_h\cap\sC_g$ and $\bu_{v_{\mathrm{lca}}}-\bu_v$. This finishes the proof.
\end{proof}

\begin{proof}[Proof of Lemma \ref{lemma:uniform bd}]
   By the law of total probability,
    \begin{align}
        \begin{split}
            &\hspace{0.5cm}\q^{\ell,v}(\exists\, w\in V_{t_{x,K}},w\succ v,\,\bet_{w,t_{x,K}}(t_{x,K})\in B_x\mid\sB_h\cap\sC_g)\\
        &=\sum_{\bu}\q^{\ell,v}(\exists\, w\in V_{t_{x,K}},w\succ v,\,\bet_{w,t_{x,K}}(t_{x,K})\in B_x\mid\sB_h\cap\sC_g\cap \sD_\bu)\,\q^{\ell,v}(\sD_\bu\mid \sB_h\cap\sC_g).
        \end{split}\label{eq:fix0}
    \end{align}
   
Our first observation is that the conditioned law of the displacements of the tree initiated at $v_{\mathrm{lca}}$ in the other $d-1$ dimensions is equivalent to its law conditioned only on the event that the same tree initiated at $v_{\mathrm{lca}}$ has at least two completely disjoint paths starting from $v_{\mathrm{lca}}$ that reach a distance of $m_{h-K}-g$ in time $h-K$ (which we denoted by $\sI_{h-K,g}$ in \eqref{eq:Ing}), since the remaining of the conditioned events belong to other independent $\sigma$-algebras.\footnote{This fact will be frequently used below, such as in the proof of Lemma \ref{lemma:ell<loglog}. \label{foot}}

Suppose first that $g\geq 0$, so the event $M^{(1)}_{h-K}>m_{h-K}-g$ is common, meaning that $\p(\sI_{h-K,g})\gg 1$, where the implicit constant may depend on $K$.
Therefore, we have
    \begin{align}
      \begin{split}
            &\hspace{0.5cm}\q^{\ell,v}(\exists\, w\in V_{t_{x,K}},w\succ v,\,\bet_{w,t_{x,K}}(t_{x,K})\in B_x\mid\sB_h\cap\sC_g\cap \sD_\bu)\\
        &=\q^{\ell,v}(\exists\, w\in V_{t_{x,K}},w\succ v,\,\bet_{w,t_{x,K}}(t_{x,K})\in B_x\mid \sI_{h-K,g}\cap \sD_\bu)\\
        &\ll \p(\exists u\in V_{h-K},\,\widehat{\bet}_{u,h-K}(h-K)\in \widehat{B}_\z-R_{\bu}),
      \end{split}\label{eq:ggeq 0}
    \end{align}
    where we denote by $\widehat{B}_\z$ the unit ball in $\R^{d-1}$. 
Here and later, we also slightly abuse notation --- the conditioned event $\sI_{h-K,g}$ refers to the event that the sub-BRW process with root $v_{\mathrm{lca}}$ satisfies the event $\sI_{h-K,g}$. Denote by $\{\widehat{\bS}_k\}_{k\geq 1}$ a random walk with i.i.d.~increments distributed as the last $d-1$ coordinates of $\bxi$. By a first moment computation and the triangle inequality,
$$\p(\exists u\in V_{h-K},\,\widehat{\bet}_{u,h-K}(h-K)\in \widehat{B}_\z-R_{\bu})\ll \rho^{h-K}\p(\|\widehat{\bS}_{h-K}\|\geq \|\bu\|-2).$$
By Cram\'{e}r's large deviation upper bound, there exist $L,\delta>0$ such that if $\|\bu\|\geq \max\{10,Lh\}$, $\p(\|\widehat{\bS}_{h-K}\|\geq \|\bu\|-2)\ll \rho^{-(h-K)}e^{-\delta \|\bu\|}$ holds uniformly. Combining the above yields
\begin{align}
   \begin{split}
        &\hspace{0.5cm}\q^{\ell,v}(\exists\, w\in V_{t_{x,K}},w\succ v,\,\bet_{w,t_{x,K}}(t_{x,K})\in B_x\mid\sB_h\cap\sC_g\cap \sD_\bu)\\
    &\ll \begin{cases}
        e^{-\delta \|\bu\|}&\text{ if }\|\bu\|\geq \max\{10,Lh\};\\
        1&\text{ otherwise}.
    \end{cases}
   \end{split}\label{eq:fix1}
\end{align}

    On the other hand, Lemma \ref{lemma:independence} implies that conditioning on $\sB_h\cap\sC_g$, the random variables $\bu_v$ and $\bu_{v_{\mathrm{lca}}}-\bu_v$ are independent under $\q^{\ell,v}$. By Lemmas \ref{lemma:uniform local CLT} and \ref{lemma:indsum},  and recalling the definition of $\sD_\bu$
 above Lemma \ref{lemma:hg|l}, 
    \begin{align}
        \q^{\ell,v}(\sD_\bu\mid \sB_h\cap\sC_g)&\ll x^{-\frac{d-1}{2}}.\label{eq:u|lhg}
    \end{align}
    Combining \eqref{eq:fix0}, \eqref{eq:fix1}, and \eqref{eq:u|lhg} yields
    \begin{align*}
        &\hspace{0.5cm}\q^{\ell,v}(\exists\, w\in V_{t_{x,K}},w\succ v,\,\bet_{w,t_{x,K}}(t_{x,K})\in B_x\mid\sB_h\cap\sC_g)    \\
        &\ll \bigg(\max\{10,Lh\}^{d-1}+\sum_{\substack{\bu\in\bZ^{d-1}\\ \|\bu\|\geq \max\{10,Lh\}}} e^{-\delta \|\bu\|}\bigg)x^{-\frac{d-1}{2}}\\
        &\ll h^{d-1}x^{-\frac{d-1}{2}},
    \end{align*}
    as desired. This proves the lemma for the case $g\geq 0$.

    Suppose now that $g<0$. We replace \eqref{eq:ggeq 0} by the bound
     \begin{align*}
      \begin{split}
            &\hspace{0.5cm}\q^{\ell,v}(\exists\, w\in V_{t_{x,K}},w\succ v,\,\bet_{w,t_{x,K}}(t_{x,K})\in B_x\mid\sB_h\cap\sC_g\cap \sD_\bu)\\
        &=\q^{\ell,v}(\exists\, w\in V_{t_{x,K}},w\succ v,\,\bet_{w,t_{x,K}}(t_{x,K})\in B_x\mid \sI_{h-K,g}\cap \sD_\bu)\\
        &\leq\sup_{\bu'\in R_{\bu}} \p(\exists u\in V_{h-K},\,{\bet}_{u,h-K}(h-K)\in B_{(m_{h-K}-g,-\bu')}\mid \sI_{h-K,g}).
      \end{split}
    \end{align*}
By Lemma \ref{lemma:double prob}, $\p(\sI_{h-K,g})\gg (|g|+1)^2e^{2c_2g}$. By Lemma \ref{lemma:double prob 2}, there exist $L,\delta>0$ such that
\begin{align}
   \begin{split}
        &\hspace{0.5cm}\p(\exists u\in V_{h-K},\,{\bet}_{u,h-K}(h-K)\in B_{(m_{h-K}-g,-\bu')};\, \sI_{h-K,g})\\
    &\ll \begin{cases}
        (|g|+1)^2e^{2c_2g}e^{-\delta \|\bu'\|}&\text{ if }\|\bu'\|\geq \max\{10,Lh\};\\
        (|g|+1)^2e^{2c_2g}&\text{ otherwise}.
    \end{cases}
   \end{split}\label{eq:fix2}
\end{align}
Replacing \eqref{eq:fix1} with \eqref{eq:fix2} and applying the rest of the arguments for $g\geq 0$ complete the proof for $g<0$.
\end{proof}

\subsubsection{Excluding local barrier events for \texorpdfstring{$-K_3{\log\log x}\leq\ell\leq K_6\log\log x$}{}}\label{sec:Local barrier events}
We define a collection of  \textit{local} barrier events and show that they have small total probabilities. Consider for $\ell\geq -K_3\log\log x$ the collection $W_{\ell,h,g}$ of particles $w\in V_{t_x-h},~w\succ v$ such that the event $\sB_h\cap\sC_g$ holds and $\n{\widehat{\bet}_{w,t_x-h}(t_x-h)}\ll h$. Intuitively, these are the possible particles that can serve as the latest common ancestor $v_{\mathrm{lca}}$. By Lemma \ref{lemma:hg|l} and \eqref{eq:u|lhg}, the number of such particles (under the global barrier event $\sG_{\wtx,K_2}^{\mathrm{c}}$) has an expectation
\begin{align}
\begin{split}
    \E_{\q^{\ell,v}}\big[\#W_{\ell,h,g}\bone_{\sG_{\wtx,K_2}^{\mathrm{c}}}\big] &\ll \min\Big\{1,(|g+\ell|+1)e^{-c_2(g+\ell)}\varphi_{(\log x)^2,\delta}(g+\ell)\Big\}\\
&\hspace{4cm}\times\min\{1,((|g|+1)e^{c_2g})^2\}h^{d-1} x^{-\frac{d-1}{2}}.
\end{split}\label{eq:Wexpect}
\end{align}
 For $w\in W_{\ell,h,g}$, we constrain the BRW initiated from $w$, in time $k\in[t_x-h,t_{x,K}]$, by the barrier 
\begin{align}
   \begin{split}
       \widehat{\psi}_{g,h}(k):= x-&m_{h-K}+g+L\log h+\frac{k-(t_x-h)}{h-K}m_{h-K}\\
   &+\frac{4}{c_2}(\log\min\{k-(t_x-h),t_{x,K}-k\})_+,t_x-h\leq k\leq t_{x,K}.
   \end{split}\label{eq:hat psi}
\end{align}
With $L$ picked large enough, it follows from Lemma \ref{lemma:barrier} that each local ballot probability is $\ll \ee h^{-3d}$. This in particular means that $L$ may depend on $\ee$, and hence in the estimates below involving the barrier \eqref{eq:hat psi}, the asymptotic constants may depend on $\ee$. We state this dependence implicitly in Lemmas \ref{lemma:ell<loglog} and \ref{lemma:ell>loglog} but omit it in their proofs for simplicity. 
Using \eqref{eq:Wexpect}, the total local ballot probability then has an expectation
\begin{align*}
    &\hspace{0.5cm}\sum_{h=K}^{(\log x)^2}\sum_{g\in\bZ}(\ee h^{-3d})\min\Big\{1,(|g+\ell|+1)e^{-c_2(g+\ell)}\varphi_{(\log x)^2,\delta}(g+\ell)\Big\}\\
    &\hspace{3cm}\times\min\{1,((|g|+1)e^{c_2g})^2\}h^{d-1} x^{-\frac{d-1}{2}}\\
    &\ll \ee x^{-\frac{d-1}{2}}e^{-c_2\ell}\sum_{h=K}^{(\log x)^2}\sum_{g\in\bZ} h^{-2d}(|g+\ell|+1)e^{-c_2g}\varphi_{(\log x)^2,\delta}(g+\ell)\min\{1,((|g|+1)e^{c_2g})^2\}\\
    &\ll \ee x^{-\frac{d-1}{2}}e^{-c_2\ell}\sum_{g\in\bZ} (|g+\ell|+1)e^{-c_2g}\varphi_{(\log x)^2,\delta}(g+\ell)\min\{1,((|g|+1)e^{c_2g})^2\}\\
    &\ll \ee x^{-\frac{d-1}{2}}e^{-c_2\ell}\bigg(\sum_{g\geq 0}(|g+\ell|+1)e^{-c_2g}\varphi_{(\log x)^2,\delta}(g+\ell)\\
    &\hspace{3cm}+\sum_{g<0} (|g+\ell|+1)e^{c_2g}\varphi_{(\log x)^2,\delta}(g+\ell)(|g|+1)^2\bigg)\\
    &\ll \ee x^{-\frac{d-1}{2}}e^{-c_2\ell}(|\ell|+1)\varphi_{(\log x)^2,\delta}(\ell).
\end{align*}

As a summary, for the barrier event 
\begin{align}\begin{split}
    \sE_{1,\ell}:=\bigcup_{K\leq h\leq (\log x)^2}\bigcup_{g\in\bZ}\bigcup_{\substack{w\in V_{t_x-h}\\ w\in W_{\ell,h,g}}}\bigcup_{\substack{w'\in V_{t_{x,K}}\\ w'\succ w}}\bigcup_{t_x-h\leq k\leq t_{x,K}}\{\eta_{w',t_{x,K}}(k)>\widehat{\psi}_{g,h}(k)\},
\end{split}\label{eq:ballot1}
\end{align}
it holds that 
\begin{align}
    \q^{\ell,v}(\sE_{1,\ell})\ll  \ee x^{-\frac{d-1}{2}}e^{-c_2\ell}(|\ell|+1)\varphi_{(\log x)^2,\delta}(\ell).\label{eq:E3l prob}
\end{align}

Another local barrier event to be removed from our consideration is, roughly speaking, the random walk $\{\eta_{w,t_{x,K}}(k)\}_{\wtx\leq k\leq t_{x,K}}$ crosses a certain barrier \textit{before} time $t_x-h$. Define for some large constant $K_{11}$ (to be determined) the barrier function
\begin{align}
    \begin{split}
        \psi_{x,K}^*(k):=\wx+K_{11}\log\log x+\frac{k-\wtx}{(\log x)^2}m_{(\log x)^2}+\frac{6}{\bl}(\log\min\{k-\wtx,&t_{x,K}-k\})_+,\\
    &\wtx\leq k\leq t_{x,K}
    \end{split}\label{eq:barrier 3}
\end{align}
and the local barrier event
\begin{align}
    \begin{split}
        \sE_{2,\ell}&:=\bigcup_{K\leq h\leq (\log x)^2}\bigg(\sB_h\cap\Big( \bigcup_{\substack{u\in V_{t_{x}-h}}}\Big(\bigcup_{\wtx\leq k\leq t_x-h}\{\eta_{u,t_x-h}(k)>{\psi}_{x,K}^*(k)\}\Big)\cap\\
    &\hspace{3cm}\Big(\bigcup_{\substack{w\in V_{t_{x,K}}\\ w\succ u}}\{\bet_{w,t_{x,K}}(t_{x,K})\in B_x\}\Big)\Big)\bigg).
    \end{split}\label{eq:E2}
\end{align}
To bound the size of $\sE_{2,\ell}$, define
$$T^*_{\wtx}:=\inf\Big\{k\in[\wtx,t_{x,K}]:\exists\, u\in V_k,u\succ v, \eta_{u,k}(k)\geq \psi_{x,K}^{*}(k)\Big\}.$$
By Lemma \ref{lemma:barrier} (with $\beta$ therein given by $\ell+K_{11}\log\log x$ as the random walk starts from $[\wx-\ell-1,\wx-\ell)$),
\begin{align*}
    &\hspace{0.5cm}\q^{\ell,v}(T^*_{\wtx}=\wtx+j)\\
    &\ll \min\{j,(\log x)^2+1-j\}^{-3}(\ell+K_{11}\log\log x) e^{-c_2(\ell+K_{11}\log\log x)}\varphi_{(\log x)^2,\delta}(\ell)\\
    &\ll \min\{j,(\log x)^2+1-j\}^{-3}(\ell+K_{11}\log\log x) e^{-c_2\ell}(\log x)^{-c_2K_{11}}.
\end{align*}
We may then compute the total contribution in the case where the barrier is crossed for a fixed $\ell$.
By Lemma \ref{lemma:uniform bd} and since the barrier event is measurable at time $t_x-h$, 
\begin{align}\begin{split}
    \q^{\ell,v}(\sE_{2,\ell})    &\leq \sum_{h=K}^{(\log x)^2}\q^{\ell,v}\bigg(\sB_h\cap \{T_{\wtx}^*\leq t_x-h\}\cap\Big(\bigcup_{\substack{w\in V_{t_{x,K}}\\ w\succ u}}\{\bet_{w,t_{x,K}}(t_{x,K})\in B_x\}\Big)\Big)\bigg)\\
    &\ll \sum_{h=K}^{(\log x)^2}h^{d-1}x^{-\frac{d-1}{2}}(\ell+K_{11}\log\log x) e^{-c_2\ell}(\log x)^{-c_2K_{11}}\\
    &\hspace{3cm}\times\sum_{j=1}^{(\log x)^2-h}\min\{j,(\log x)^2+1-j\}^{-3}\\
    &\ll x^{-\frac{d-1}{2}}(\ell+K_{11}\log\log x) e^{-c_2\ell}(\log x)^{2d-c_2K_{11}}\\
    &\ll x^{-\frac{d-1}{2}} e^{-c_2\ell},
\end{split}\label{eq:E2l prob}
\end{align}
where in the last step we picked $K_{11}$ large enough and used that $|\ell|\ll\log\log x$. As a summary,
\begin{align}
    \q^{\ell,v}(\sE_{2,\ell})\ll x^{-\frac{d-1}{2}} e^{-c_2\ell}.\label{eq:E2 bound}
\end{align}

\subsubsection{Local hitting probabilities for \texorpdfstring{$-K_3{\log\log x}\leq\ell\leq K_6\log\log x$}{}}
\label{sec:prelim bound}

In this section, we prove an upper bound of \eqref{eq:q} in the case {$-K_3\log\log x\leq \ell\leq K_6\log\log x$}. Define \begin{align}
    \sE_{\ell}^*:=\sE_{1,\ell}\cup \sE_{2,\ell}\label{eq:el*}
\end{align} and
\begin{align}
    \begin{split}
        \sK^*_{h}:=\sB_h\cap\Big( \bigcup_{\substack{u\in V_{t_{x}-h}}}&\Big(\bigcup_{\wtx\leq k\leq t_x-h}\{\eta_{u,t_x-h}(k)>{\psi}_{x,K}^*(k)\}\Big)\\
    &\cap\Big(\bigcup_{\substack{w\in V_{t_{x,K}}\\ w\succ u}}\{\bet_{w,t_{x,K}}(t_{x,K})\in B_x\}\Big)\Big).
    \end{split}\label{eq:Kh*}
\end{align}
Note that $\sE_{2,\ell}=\bigcup_{K\leq h\leq (\log x)^2}\sK_{h}^*$, so that $(\sE_\ell^*)^{\mathrm{c}}\subseteq (\sK_h^*)^{\mathrm{c}}$ for all $K\leq h\leq (\log x)^2$. Also, define $k_h:=(\log x)^2-h$.

\begin{lemma}[size of the event $\sB_h\cap\sC_g\cap\sD_\bu\cap(\sK_h^*)^{\mathrm{c}}$]\label{lemma:hgu|l}Suppose that $-K_3\log\log x\leq \ell\leq K_6\log\log x$. 
It holds that uniformly for all $\bu\in\R^{d-1}$, $K\leq h\leq (\log x)^2$, and $g\in\bZ$,
\begin{align}
    \begin{split}
        \q^{\ell,v}(\sB_h\cap\sC_g\cap\sD_\bu)&\ll x^{-\frac{d-1}{2}}\min\Big\{1,(|g+\ell|+1)e^{-c_2(g+\ell)}\varphi_{(\log x)^2,\delta}(g+\ell)\Big\}\\
        &\hspace{6cm}\times\min\{1,((|g|+1)e^{c_2g})^2\}.
    \end{split}\label{eq:hgu|l1}
\end{align}
Moreover, assume that $k_h\geq \log x$, and fix $K_5>0$. If $0\leq g\leq K_5\log h$,
\begin{align}
   \q^{\ell,v}(\sB_h\cap\sC_g\cap\sD_\bu\cap(\sK_h^*)^{\mathrm{c}})\ll x^{-\frac{d-1}{2}}e^{-c_2(\ell+g)}\min\{k_h,h\}^{-5/4}(\log\log x)^2e^{\frac{K_{10}(\log k_h)g}{k_h}},\label{eq:hgu|l2}
\end{align}
and if $g<0$,
\begin{align}
   \q^{\ell,v}(\sB_h\cap\sC_g\cap\sD_\bu\cap(\sK_h^*)^{\mathrm{c}})\ll x^{-\frac{d-1}{2}}e^{-c_2\ell}\min\{k_h,h\}^{-5/4}(\log\log x)^2(|g|+1)^2e^{c_2g}.\label{eq:hgu|l3}
\end{align}

\end{lemma}
\begin{proof}
First, we write 
$$\q^{\ell,v}(\sB_h\cap\sC_g\cap\sD_\bu)=\q^{\ell,v}(\sD_\bu\mid \sB_h\cap\sC_g)\,\q^{\ell,v}(\sB_h\cap\sC_g).$$
The first probability can be controlled by \eqref{eq:u|lhg}, and the second probability by Lemma \ref{lemma:hg|l}. Inserting these estimates proves \eqref{eq:hgu|l1}.

To prove \eqref{eq:hgu|l2} and \eqref{eq:hgu|l3}, we apply the ballot theorem under a change of measure similarly as in the proof of Proposition \ref{prop:particle density}, with the barrier given by \eqref{eq:barrier 3}. The starting location of the BRW at time $\wtx$ is $\wx-\ell$, which is of distance $\ell+K_{11}\log\log x$ below the barrier $\psi_{x,K}^*(\wtx)$. The end location of the BRW at time $\wtx+k_h$ is $x-m_{h-K}+g$, which is of distance $O(\log\min\{k_h,h\})-g-\frac{c_1K}{2}+K_{11}\log\log x$ below the barrier $\psi_{x,K}^*(\wtx+k_h)$. 
Applying Lemma 2.3 of \citep{bramson2016convergence}, if $0\leq g\leq K_5\log h$,
\begin{align*}
      &\hspace{0.5cm}\q^{\ell,v}(\sB_h\cap\sC_g\cap (\sK_{h}^*)^{\mathrm{c}})\\
      &\ll \rho^{k_h}(\rho^{-k_h}k_h^{3/2}e^{-\hat{\lambda}(x-m_{h-K}+g-(\wx-\ell)-m_{k_h})})\\
      &\hspace{3cm}\times(k_h^{-3/2}(\ell+K_{11}\log\log x)(\log\min\{k_h,h\}+K_{11}\log\log x))\\
      &\ll  e^{-c_2(\ell+g)}\min\{k_h,h\}^{-3/2}(\log\log x)^2e^{(c_2-\hat{\lambda})(m_{(\log x)^2}-m_{h-K}-m_{k_h}+\ell+g)}.
    \end{align*}

    To further bound the term $e^{(c_2-\hat{\lambda})(m_{(\log x)^2}-m_{h-K}-m_{k_h}+\ell+g)}$, we recall from the proof of Proposition \ref{prop:particle density} that $c_2-\hat{\lambda}\leq K_{10}(\log k_h)/k_h$ for some $K_{10}>0$. Since $k_h\geq \log x\to\infty$ as $x\to\infty$, we may assume that $c_2-\hat{\lambda}\leq c_2/6$ by letting $x$ be large enough. In this case,
$$e^{(c_2-\hat{\lambda})(m_{(\log x)^2}-m_{h-K}-m_{k_h})}\leq e^{c_2(m_{(\log x)^2}-m_{h-K}-m_{k_h})/6}\ll \min\{k_h,h\}^{1/4}.$$
In addition, since $k_h\geq \log x$ and $|\ell|\ll\log\log x$, we have $e^{\frac{K_{10}(\log k_h)\ell}{k_h}}\ll 1$. 
These considerations altogether lead to
$$\q^{\ell,v}(\sB_h\cap\sC_g\cap(\sK_h^*)^{\mathrm{c}})\ll e^{-c_2(\ell+g)}\min\{k_h,h\}^{-5/4}(\log\log x)^2e^{\frac{K_{10}(\log k_h)(\ell+g)}{k_h}}.$$
On the other hand, it follows from the same independent sum argument leading to \eqref{eq:u|lhg} that $\q^{\ell,v}(\sD_\bu\mid\sB_h\cap\sC_g\cap (\sK_{h}^*)^{\mathrm{c}})\ll x^{-\frac{d-1}{2}}$. This proves \eqref{eq:hgu|l2}.

If $g<0$, we take advantage of the rare event that two independent descendants run distances $m_{h-K}-g$ for time $h-K$ (given by Lemma \ref{lemma:double prob}).  Applying the same arguments as in Lemma \ref{lemma:hg|l} and using Lemma 2.3 of \citep{bramson2016convergence}, we have
    \begin{align*}
         &\hspace{0.5cm}\q^{\ell,v}(\sB_h\cap\sC_g\cap (\sK_{h}^*)^{\mathrm{c}})\\
         &\ll \rho^{k_h}\big(\rho^{-k_h}k_h^{3/2}e^{-\hat{\lambda}(x-m_{h-K}+g-(\wx-\ell)-m_{k_h})}\big)\\
         &\hspace{2cm}\times(k_h^{-3/2}(\ell+K_{11}\log\log x)(\log\min\{k_h,h\}+K_{11}\log\log x))(|g|+1)^2e^{c_2g}\\
         &\ll  e^{-c_2\ell}\min\{k_h,h\}^{-3/2}(|g|+1)^2e^{c_2g}e^{(c_2-\hat{\lambda})(m_{(\log x)^2}-m_{h-K}-m_{k_h}+\ell+g)}.
    \end{align*}
    The remaining follows similarly as the case $g\geq 0$, and we note that $e^{\frac{K_{10}(\log k_h)g}{k_h}}\ll1$ for $g<0$. 
\end{proof}

\begin{lemma}[first passage contribution]\label{lemma:ell<loglog}
For $-K_3\log\log x\leq \ell\leq K_6\log\log x$, it holds for some $K_7\geq 0$ that
    \begin{align}
        \begin{split}
            &\q^{\ell,v}(\exists\, w\in V_{t_{x,K}},w\succ v,\,\bet_{w,t_{x,K}}(t_{x,K})\in B_x,\,(\sE_\ell^*)^{\mathrm{c}})\\
            &\hspace{4cm}\ll C(\ee,K)(\log\log x)^{K_7}(\log x)x^{-\frac{d-1}{2}} e^{-c_2\ell}.
        \end{split}\label{eq:fpc1}
    \end{align}
\end{lemma}

\begin{proof}In the following, the asymptotic constants in $\ll$ may depend on $\ee,K$. 
We condition on $\sB_h,\sC_g,\sD_\bu$ and apply the law of total probability to write
\begin{align}\begin{split}
    &\hspace{0.5cm}\q^{\ell,v}(\exists\, w\in V_{t_{x,K}},w\succ v,\,\bet_{w,t_{x,K}}(t_{x,K})\in B_x,\,(\sE_\ell^*)^{\mathrm{c}})\\
    &= \sum_{h=K}^{(\log x)^2}\sum_{g\in\bZ}\sum_{\bu\in\bZ^{d-1}} \q^{\ell,v}(\exists\, w\in V_{t_{x,K}},\,w\succ v,\,\bet_{w,t_{x,K}}(t_{x,K})\in B_x,\,(\sE_\ell^*)^{\mathrm{c}}\mid \sD_\bu\cap \sC_g\cap \sB_h)\\
    &\hspace{3cm}\times\q^{\ell,v}(\sB_h\cap\sC_g\cap\sD_\bu).
\end{split}\label{eq:prob conditioned}
\end{align}
Let us consider a large constant $K_5>0$ to be determined, and separate into three cases depending on the values of $g$. In the following, the asymptotic constants may depend on $\ee$.

Case (a): $g\geq K_5\log h$. In this case, we do \textit{not} condition on $\sD_\bu$, but directly apply Lemmas \ref{lemma:hg|l} and \ref{lemma:uniform bd} to get
\begin{align*}
    &\hspace{0.5cm}\sum_{h=K}^{(\log x)^2}\sum_{g\geq K_5\log h}\sum_{\bu\in\bZ^{d-1}} \q^{\ell,v}(\sB_h\cap\sC_g\cap\sD_\bu)\\
    &\hspace{3cm}\times\q^{\ell,v}(\exists\, w\in V_{t_{x,K}},\,w\succ v,\,\bet_{w,t_{x,K}}(t_{x,K})\in B_x,\,(\sE_\ell^*)^{\mathrm{c}}\mid \sD_\bu\cap \sC_g\cap \sB_h)\\
    &=\sum_{h=K}^{(\log x)^2}\sum_{g\geq K_5\log h}\q^{\ell,v}(\sB_h\cap\sC_g)\\
    &\hspace{3cm}\times\q^{\ell,v}(\exists\, w\in V_{t_{x,K}},\,w\succ v,\,\bet_{w,t_{x,K}}(t_{x,K})\in B_x,\,(\sE_\ell^*)^{\mathrm{c}}\mid  \sC_g\cap \sB_h)\\
    &\ll \sum_{h=K}^{(\log x)^2}\sum_{g\geq K_5\log h} (h^{d-1}x^{-\frac{d-1}{2}})(\min\Big\{1,(|g+\ell|+1)e^{-c_2(g+\ell+c_1K/2)}\varphi_{(\log x)^2,\delta}(g+\ell)\Big\}\\
    &\hspace{3cm}\times 
 \min\{1,((|g|+1)e^{c_2g})^2\})\\
    &= \sum_{h=K}^{(\log x)^2}\sum_{g>K_5\log h}\min\Big\{1,(|g+\ell|+1)e^{-c_2(g+\ell)}\varphi_{(\log x)^2,\delta}(g+\ell)\Big\}h^{d-1}x^{-\frac{d-1}{2}}\\
    &\leq \sum_{h=K}^{(\log x)^2}(|K_5\log h+\ell|+1)\min\Big\{1,e^{-c_2(K_5\log h+\ell)}\varphi_{(\log x)^2,\delta}(K_5\log h+\ell)\Big\}h^{d-1}x^{-\frac{d-1}{2}}\\
    &\ll (\log\log x)e^{-c_2\ell}x^{-\frac{d-1}{2}}\varphi_{(\log x)^2,\delta}(\ell),
\end{align*}where in the last step we pick $K_5$ large enough.

Case (b): $0\leq g<K_5\log h$. We further split into two sub-cases. 
First, consider $\bu$ such that $\n{\bu}\leq \sqrt{h}\log h$.
We first compute an upper bound for
\begin{align*}
    \q^{\ell,v}(\exists\, w\in V_{t_{x,K}},\,w\succ v,\,\bet_{w,t_{x,K}}(t_{x,K})\in B_x,\,(\sE_\ell^*)^{\mathrm{c}}\mid \sD_\bu\cap \sC_g\cap \sB_h).
\end{align*}
Conditioned on $\sB_h\cap \sC_g$, the event that $\bet_{w,t_{x,K}}(t_{x,K})\in B_x$ holds implies $w\succ v_{\mathrm{lca}}$. Denote by $B_\bz$ the unit ball centered at $\bz\in\R^d$.  
By independence (see footnote \ref{foot}), 
\begin{align}
\begin{split}
    &\hspace{0.5cm}\q^{\ell,v}(\exists\, w\in V_{t_{x,K}},\,w\succ v,\,\bet_{w,t_{x,K}}(t_{x,K})\in B_x,\,(\sE_\ell^*)^{\mathrm{c}}\mid \sD_\bu\cap \sC_g\cap \sB_h)\\
    &\leq\sup_{\bu\in\R^{d-1}}\p(\exists v\in V_{h-K},\,\bet_{v,h-K}(h-K)\in B_{(m_{h-K}-g,\bu)},\,(\sE_\ell^*)^{\mathrm{c}}\mid \sI_{h-K,g})\\
    &\ll \sup_{\bu\in\R^{d-1}}\p(\exists v\in V_{h-K},\,\bet_{v,h-K}(h-K)\in B_{(m_{h-K}-g,\bu)},\,(\sE_\ell^*)^{\mathrm{c}}),
\end{split}\label{eq:3steps}
\end{align}
where the last step is because for $g\geq 0$, the event $\sI_{h-K,g}$ that two descendants of $v_{\mathrm{lca}}$ separated at first step both reach $\H_x$ at time $t_{x,K}$ have a probability $\gg 1$, and hence we may remove the conditioning on $\sI_{h-K,g}$ in \eqref{eq:3steps} without changing the asymptotic upper bound. We have also abused notation by using $\sE_{1,\ell}^{\mathrm{c}}$ to denote the event that the BRW is constrained by the barrier
$$k\mapsto L\log h+\frac{k}{h-K}m_{h-K}+\frac{4}{c_2}(\log\min\{k,h-K-k\})_+,~1\leq k\leq h-K;$$
see \eqref{eq:hat psi}. By Lemma \ref{thm:d-dim ballot with log} and a standard change of measure computation, we have uniformly in $\bu\in\R^{d-1}$,
$$\p(\exists v\in V_{h-K},\,\bet_{v,h-K}(h-K)\in B_{(m_{h-K}-g,\bu)},\,\sE_{1,\ell}^{\mathrm{c}})\ll (g+1)e^{c_2g}(\log h)^2(h-K+1)^{-\frac{d-1}{2}}.$$
Therefore, we arrive at
\begin{align}
    \begin{split}
        &\q^{\ell,v}(\exists\, w\in V_{t_{x,K}},\,w\succ v,\,\bet_{w,t_{x,K}}(t_{x,K})\in B_x,\,(\sE_\ell^*)^{\mathrm{c}}\mid \sD_\bu\cap \sC_g\cap \sB_h)\\
        &\hspace{4cm}\ll (g+1)e^{c_2g}(\log h)^2(h-K+1)^{-\frac{d-1}{2}}.
    \end{split}\label{eq:x|ghlu}
\end{align}
Since the event $(\sK_h^*)^{\mathrm{c}}$ depends only on times $[\wtx,\wtx+k_h]$ once we know a descendant of the latest common ancestor reaches $B_x$ at time $t_{x,K}$, we obtain also that 
\begin{align}
\begin{split}
    &\hspace{0.5cm}\q^{\ell,v}(\exists\, w\in V_{t_{x,K}},\,w\succ v,\,\bet_{w,t_{x,K}}(t_{x,K})\in B_x,\,(\sE_\ell^*)^{\mathrm{c}}\mid \sD_\bu\cap \sC_g\cap \sB_h\cap (\sK_h^*)^{\mathrm{c}})\\
&\ll (g+1)e^{c_2g}(\log h)^2(h-K+1)^{-\frac{d-1}{2}}.
\end{split}\label{eq:x|ghlu2}
\end{align}

Fix $h\in[(\log x)^2-\log x,(\log x)^2]$, i.e., $k_h\leq \log x$. Applying \eqref{eq:x|ghlu} and \eqref{eq:hgu|l1} of Lemma \ref{lemma:hgu|l}, we have
\begin{align*}
    &\hspace{0.5cm}\sum_{g=0}^{K_5\log h}\sum_{\substack{\bu\in\bZ^{d-1}\\ \n{\bu}\leq \sqrt{h}\log h}}\q^{\ell,v}(\exists\, w\in V_{t_{x,K}},\,w\succ v,\,\bet_{w,t_{x,K}}(t_{x,K})\in B_x,\,(\sE_\ell^*)^{\mathrm{c}}\mid \sD_\bu\cap \sC_g\cap \sB_h)\\
    &\hspace{3cm}\times\q^{\ell,v}(\sB_h\cap\sC_g\cap\sD_\bu)\\
    &\ll \sum_{g=0}^{K_5\log h}(g+1)e^{c_2g}(\log h)^2(h-K+1)^{-\frac{d-1}{2}}\sum_{\substack{\bu\in\bZ^{d-1}\\ \n{\bu}\leq \sqrt{h}\log h}}\q^{\ell,v}(\sB_h\cap\sC_g\cap\sD_\bu)\\
    &\ll \sum_{g=0}^{K_5\log h}(g+1)e^{c_2g}(\log h)^2(h-K+1)^{-\frac{d-1}{2}}(\sqrt{h}\log h)^{d-1}x^{-\frac{d-1}{2}}\\
    &\hspace{3cm}\times\min\Big\{1,(|g+\ell|+1)e^{-c_2(g+\ell)}\varphi_{(\log x)^2,\delta}(g+\ell)\Big\}\\
    &\ll \sum_{g=0}^{K_5\log h}\min\Big\{1,(|g+\ell|+1)e^{-c_2(g+\ell)}\varphi_{(\log x)^2,\delta}(g+\ell)\Big\} \\
    &\hspace{3cm}\times (g+1)e^{c_2g}(\log h)^{d+1}x^{-\frac{d-1}{2}}\Big(\frac{h}{h-K+1}\Big)^{\frac{d-1}{2}}.
\end{align*}
The above quantity after summation over $h$ is thus bounded by
    \begin{align}
  \begin{split}
       &\hspace{0.5cm} \sum_{h=(\log x)^2-\log x}^{(\log x)^2}\sum_{g=0}^{K_5\log h}\min\Big\{1,(|g+\ell|+1)e^{-c_2(g+\ell)}\varphi_{(\log x)^2,\delta}(g+\ell)\Big\}\\
    &\hspace{3cm}\times
    (g+1)e^{c_2g}(\log h)^{d+1}x^{-\frac{d-1}{2}}\Big(\frac{h}{h-K+1}\Big)^{\frac{d-1}{2}}\\
    &\ll (\log\log x)^{K_7}(\log x)x^{-\frac{d-1}{2}}\min\{e^{-c_2\ell},1\}\ll(\log\log x)^{K_7}(\log x)x^{-\frac{d-1}{2}} e^{-c_2\ell},
  \end{split}\label{eq:w1}
\end{align}
where we have used an integral approximation of a sum. Next, we consider $h\in[K,(\log x)^2-\log x]$, i.e., $k_h\geq \log x$. Applying \eqref{eq:x|ghlu2} and \eqref{eq:hgu|l2} of Lemma \ref{lemma:hgu|l} and using that $h\leq (\log x)^2$, we have
\begin{align*}
    &\hspace{0.5cm}\sum_{g=0}^{K_5\log h}\sum_{\substack{\bu\in\bZ^{d-1}\\ \n{\bu}\leq \sqrt{h}\log h}}\q^{\ell,v}(\sB_h\cap\sC_g\cap\sD_\bu\cap(\sK_h^*)^{\mathrm{c}})\\
    &\hspace{1.5cm}\times\q^{\ell,v}(\exists\, w\in V_{t_{x,K}},\,w\succ v,\,\bet_{w,t_{x,K}}(t_{x,K})\in B_x,\,(\sE_\ell^*)^{\mathrm{c}}\mid \sD_\bu\cap \sC_g\cap \sB_h\cap(\sK_h^*)^{\mathrm{c}})\\
    &\ll \sum_{g=0}^{K_5\log h}(g+1)e^{c_2g}(\log h)^2(h-K+1)^{-\frac{d-1}{2}}\\
    &\hspace{3cm}\times\sum_{\substack{\bu\in\bZ^{d-1}\\ \n{\bu}\leq \sqrt{h}\log h}}x^{-\frac{d-1}{2}}e^{-c_2(\ell+g)}\min\{k_h,h\}^{-5/4}(\log\log x)^2e^{\frac{K_{10}(\log k_h)g}{k_h}}\\
    &\ll \min\{k_h,h\}^{-5/4}\Big(\frac{h}{h-K+1}\Big)^{\frac{d-1}{2}}x^{-\frac{d-1}{2}}e^{-c_2\ell}(\log h)^{d+1}(\log\log x)^2\sum_{g=0}^{K_5\log h}ge^{\frac{K_{10}(\log k_h)g}{k_h}}\\
    &\ll \min\{k_h,h\}^{-5/4}x^{-\frac{d-1}{2}}e^{-c_2\ell}(\log\log x)^{K_{12}},
\end{align*}where in the second step, we used that $\#\{\bu\in\bZ^{d-1}:\,\n{\bu}\leq \sqrt{h}\log h\}\ll h^{\frac{d-1}{2}}(\log h)^{d-1}$.
Summing over $h$, we obtain
\begin{align}
    \sum_{h=K}^{(\log x)^2-\log x}\min\{k_h,h\}^{-5/4}x^{-\frac{d-1}{2}}e^{-c_2\ell}(\log\log x)^{K_{12}}\ll K^{\frac{d-1}{2}}x^{-\frac{d-1}{2}}e^{-c_2\ell}(\log\log x)^{K_{12}}.\label{eq:w2}
\end{align}
Combining \eqref{eq:w1} and \eqref{eq:w2} yields a total contribution of at most $$C(K)(\log\log x)^{K_7}(\log x)x^{-\frac{d-1}{2}} e^{-c_2\ell}.$$

Next, we consider $\bu$ with $\n{\bu}>\sqrt{h}\log h$. In this case, using a change of measure computation (without using ballot theorem) and a moderate deviation estimate (e.g., Theorem 3.7.1 of \citep{dembo2009large}), for some $\delta>0$,
\begin{align*}
    &\hspace{0.5cm}\q^{\ell,v}(\exists\, w\in V_{t_{x,K}},\,w\succ v,\,\bet_{w,t_{x,K}}(t_{x,K})\in B_x,\,\sE_{1,\ell}^{\mathrm{c}}\mid \sD_\bu\cap \sC_g\cap \sB_h)\\
    &\leq\sup_{\bu'\in R_\bu} \p(\exists v\in V_{h-K},\,\bet_{v,h-K}(h-K)\in B_{(m_{h-K}-g,-\bu')},\,\sE_{1,\ell}^{\mathrm{c}})\\
    &\ll (h-K)^{3/2}e^{c_2g}\p(\n{\bS_{h-K}}\geq \n{\bu})\\
    &\ll h^{3/2}e^{c_2g}\varphi_{h,\delta}(\n{\bu}).
\end{align*}
Inserting into \eqref{eq:prob conditioned}, we have 
\begin{align*}
    &\hspace{0.5cm}\sum_{h=K}^{(\log x)^2}\sum_{g=0}^{K_5\log h}\sum_{\substack{\bu\in\bZ^{d-1}\\ \n{\bu}>\sqrt{h}\log h}} \q^{\ell,v}(\sB_h\cap\sC_g\cap\sD_\bu)\\
    &\hspace{3cm}\times\q^{\ell,v}(\exists\, w\in V_{t_{x,K}},\,w\succ v,\,\bet_{w,t_{x,K}}(t_{x,K})\in B_x,\,\sE_{1,\ell}^{\mathrm{c}}\mid \sD_\bu\cap \sC_g\cap \sB_h)\\    
    &\ll \sum_{h=K}^{(\log x)^2}\sum_{g=0}^{K_5\log h}\min\Big\{1,(|g+\ell|+1)e^{-c_2(g+\ell)}\varphi_{(\log x)^2,\delta}(g+\ell)\Big\}x^{-\frac{d-1}{2}}\\
    &\hspace{3cm}\sum_{k=\sqrt{h}\log h}^\infty\sum_{\substack{\bu\in\bZ^{d-1}\\ \n{\bu}\in[k,k+1]}}h^{3/2}e^{c_2g}\varphi_{h,\delta}(\n{\bu})\\
    &\ll \sum_{h=K}^{(\log x)^2}\sum_{g=0}^{K_5\log h}(|g+\ell|+1)e^{-c_2(g+\ell)}x^{-\frac{d-1}{2}} h^{3/2}e^{c_2g}\sum_{k=\sqrt{h}\log h}^\infty k^{d-2}\varphi_{h,\delta}(k)\\
    &\ll  \sum_{h=K}^{(\log x)^2}\sum_{g=0}^{K_5\log h}(|g+\ell|+1)e^{-c_2\ell}x^{-\frac{d-1}{2}} h^{-100}\\
    &\ll x^{-\frac{d-1}{2}}e^{-c_2\ell}(\log\log x)^2\varphi_{(\log x)^2,\delta}(\ell),
\end{align*}
where we have used an integral approximation in the third step and that $|\ell|\ll \log\log x$ implies $\varphi_{(\log x)^2,\delta}(\ell)\gg 1$ in the last step. 

Case (c): $g<0$. We exclude barrier events and compute the expected number of particles beyond $x$ at time $t_{x,K}$ under the barrier event and conditioned on  $\sC_g$.
Define the barrier event 
\begin{align}
\begin{split}
    \sF_{h,g}:=\bigcup_{\substack{w\in V_{t_{x,K}}\\w\succ v_{\mathrm{lca}}}}\bigcup_{0\leq k\leq h-K}&\bigg\{\eta_{w,t_{x,K}}(t_x-h+k)-\eta_{w,t_{x,K}}(t_x-h)\\
    &\geq L\log h-g+\frac{k}{h-K}m_{h-K}+\frac{4}{c_2}(\log\min\{k,h-K-k\})_+\bigg\}.
\end{split}\label{eq:Fhg}
\end{align}
 for the sub-tree with root $v_{\mathrm{lca}}$. 
Recall from \eqref{eq:ballot1} that on the event $\sE_{1,\ell}^{\mathrm{c}}$, the event $\sF_{h,g}$ cannot hold for each latest common ancestor $v_{\mathrm{lca}}$.

In \eqref{eq:prob conditioned}, the sum over $\bu$ with $\n{\bu}>\sqrt{h}\log h$ can be handled similarly as the case $g\geq 0$. 
By Lemma \ref{lemma:2prob2case}, we have the upper bound 
$$\q^{\ell,v}(\exists\, w\in V_{t_{x,K}},\,w\succ v,\,\bet_{w,t_{x,K}}(t_{x,K})\in B_x,\,\sE_{1,\ell}^{\mathrm{c}}\mid \sD_\bu\cap \sC_g\cap \sB_h)\ll h^{3/2}\varphi_{h,\delta}(\|\bu\|).$$
Inserting into \eqref{eq:prob conditioned}, we have by Lemma \ref{lemma:hgu|l} and arguing similarly in the case $g\geq 0$,
\begin{align*}
    &\hspace{0.5cm}\sum_{h=K}^{(\log x)^2}\sum_{g<0}\sum_{\substack{\bu\in\bZ^{d-1}\\ \n{\bu}>\sqrt{h}\log h}} \q^{\ell,v}(\exists\, w\in V_{t_{x,K}},\,w\succ v,\,\bet_{w,t_{x,K}}(t_{x,K})\in B_x,\,\sE_{1,\ell}^{\mathrm{c}}\mid \sD_\bu\cap \sC_g\cap \sB_h)\\
    &\hspace{3cm}\times\q^{\ell,v}(\sB_h\cap\sC_g\cap\sD_\bu)\\  
    &\ll \sum_{h=K}^{(\log x)^2}\sum_{g<0}\sum_{\substack{\bu\in\bZ^{d-1}\\ \n{\bu}>\sqrt{h}\log h}}\min\Big\{1,(|g+\ell|+1)e^{-c_2(g+\ell)}\varphi_{(\log x)^2,\delta}(g+\ell)\Big\}x^{-\frac{d-1}{2}} \\
    &\hspace{3cm}\times(|g|+1)^2e^{2c_2g} h^{3/2}\max\big\{e^{-\frac{\delta \n{\bu}^2}{h}},e^{-\delta\n{\bu}}\big\}\\
    &\ll \sum_{h=K}^{(\log x)^2}\sum_{g<0}\min\Big\{1,(|g+\ell|+1)e^{-c_2(g+\ell)}\varphi_{(\log x)^2,\delta}(g+\ell)\Big\}x^{-\frac{d-1}{2}} (|g|+1)^2e^{2c_2g}h^{-100}\\
    &\ll x^{-\frac{d-1}{2}}e^{-c_2\ell}\sum_{g<0}(|g+\ell|+1)(|g|+1)^2\varphi_{(\log x)^2,\delta}(g+\ell)e^{c_2g}\\
    &\ll x^{-\frac{d-1}{2}}e^{-c_2\ell}\varphi_{(\log x)^2,\delta}(\ell).
\end{align*}
Note that, contrary to the case $g\geq 0$, here we do not have any constraint on the value of $\ell$. The same computation will be re-used later in the proof of Lemma \ref{lemma:ell>loglog} when considering $\ell>K_6\log\log x$.

Let us now consider $\bu$ with $\n{\bu}\leq \sqrt{h}\log h$. Using independence (see footnote \ref{foot}), the first probability on the right-hand side of \eqref{eq:prob conditioned} can be controlled by (similar consideration as the case $g\geq 0$)
\begin{align*}
    &\hspace{0.5cm}\q^{\ell,v}(\exists\, w\in V_{t_{x,K}},\,w\succ v,\,\bet_{w,t_{x,K}}(t_{x,K})\in B_x,\,\sF^{\mathrm{c}}_{h,g}\mid\sD_\bu\cap \sC_g\cap \sB_h)\\
    &\leq \sup_{\bu'\in R_{\bu}}\q^{\ell,v}(\exists\, w\in V_{t_{x,K}},\,w\succ v_{\mathrm{lca}},\,\bet_{w,t_{x,K}}(t_{x,K})-\bet_{w,t_{x,K}}(t_x-h)\\
    &\hspace{3cm}\in B_x-(x-m_{h-K}+g,\bu'),\,\sF^{\mathrm{c}}_{h,g}\mid \sI_{h-K,g}).
\end{align*}
 Therefore, by Lemma \ref{lemma:2prob2case},
 \begin{align}
     \begin{split}
         &\hspace{0.5cm}\q^{\ell,v}(\exists\, w\in V_{t_{x,K}},\,w\succ v,\,\bet_{w,t_{x,K}}(t_{x,K})\in B_x,\,\sF^{\mathrm{c}}_{h,g}\mid\sD_\bu\cap \sC_g\cap \sB_h)\\
         &\ll (\log h)^2 (h-K+1)^{-\frac{d-1}{2}}.
     \end{split}\label{eq:x|ughl for g<0}
 \end{align}
 Similarly as in \eqref{eq:x|ghlu2}, we also have 
  \begin{align}
    \begin{split}
         &\hspace{0.5cm}\q^{\ell,v}(\exists\, w\in V_{t_{x,K}},\,w\succ v,\,\bet_{w,t_{x,K}}(t_{x,K})\in B_x,\,\sF^{\mathrm{c}}_{h,g}\mid\sD_\bu\cap \sC_g\cap \sB_h\cap(\sK_h^*)^{\mathrm{c}})\\
         &\ll (\log h)^2 (h-K+1)^{-\frac{d-1}{2}}.
    \end{split}\label{eq:x|ughl for g<02}
 \end{align}
We then apply \eqref{eq:x|ughl for g<0} and \eqref{eq:hgu|l1} of Lemma \ref{lemma:hgu|l} to get for $(\log x)^2-\log x\leq h\leq (\log x)^2$,
\begin{align*}
    &\hspace{0.5cm}\sum_{h=(\log x)^2-\log x}^{(\log x)^2}\sum_{g<0}\sum_{\substack{\bu\in\bZ^{d-1}\\ \n{\bu}\leq \sqrt{h}\log h}} \q^{\ell,v}(\sB_h\cap\sC_g\cap\sD_\bu) \\
    &\hspace{2cm}\times\q^{\ell,v}(\exists\, w\in V_{t_{x,K}},\,w\succ v,\,\bet_{w,t_{x,K}}(t_{x,K})\in B_x,\,\sE_{1,\ell}^{\mathrm{c}}\mid \sD_\bu\cap \sC_g\cap \sB_h)\\
    &\ll \sum_{h=(\log x)^2-\log x}^{(\log x)^2}\sum_{g<0}\sum_{\substack{\bu\in\bZ^{d-1}\\ \n{\bu}\leq \sqrt{h}\log h}} (\log h)^2 (h-K+1)^{-\frac{d-1}{2}}x^{-\frac{d-1}{2}}|g+\ell|\\
    &\hspace{3cm}\times(|g|+1)^2e^{c_2g-c_2\ell}\varphi_{(\log x)^2,\delta}(g+\ell)\\
    &\ll x^{-\frac{d-1}{2}}e^{-c_2\ell}\sum_{h=(\log x)^2-\log x}^{(\log x)^2}\sum_{g<0} (\log h)^{d+1} |g+\ell|\\
    &\hspace{4cm}(|g|+1)^2e^{c_2g}\varphi_{(\log x)^2,\delta}(g+\ell)\Big(\frac{h}{h-K+1}\Big)^{\frac{d-1}{2}}\\
    &\ll x^{-\frac{d-1}{2}}(|\ell|+1)e^{-c_2\ell}(\log x)(\log\log x)^{d+4},
\end{align*}
where in the second step, we used that $\#\{\bu\in\bZ^{d-1}:\,\n{\bu}\leq \sqrt{h}\log h\}\ll h^{\frac{d-1}{2}}(\log h)^{d-1}$. 
For $K\leq h\leq (\log x)^2-\log x$, we apply \eqref{eq:x|ughl for g<02} and \eqref{eq:hgu|l3} of Lemma \ref{lemma:hgu|l} to get
\begin{align*}
    &\hspace{0.5cm}\sum_{h=K}^{(\log x)^2-\log x}\sum_{g<0}\sum_{\substack{\bu\in\bZ^{d-1}\\ \n{\bu}\leq \sqrt{h}\log h}}\q^{\ell,v}(\sB_h\cap\sC_g\cap\sD_\bu\cap(\sK_h^*)^{\mathrm{c}}) \\
    &\hspace{1cm}\times\q^{\ell,v}(\exists\, w\in V_{t_{x,K}},\,w\succ v,\,\bet_{w,t_{x,K}}(t_{x,K})\in B_x,\,(\sE_{\ell}^*)^{\mathrm{c}}\mid \sD_\bu\cap \sC_g\cap \sB_h\cap(\sK_h^*)^{\mathrm{c}})\\
    &\ll\sum_{h=K}^{(\log x)^2-\log x}\sum_{g<0}\sum_{\substack{\bu\in\bZ^{d-1}\\ \n{\bu}\leq \sqrt{h}\log h}} (\log h)^2 (h-K+1)^{-\frac{d-1}{2}}\\
    &\hspace{2cm}\times x^{-\frac{d-1}{2}}e^{-c_2\ell}\min\{k_h,h\}^{-5/4}(\log\log x)^2(|g|+1)^2e^{c_2g}e^{\frac{K_{10}(\log k_h)g}{k_h}}\\
    &\ll \sum_{h=K}^{(\log x)^2-\log x}\min\{k_h,h\}^{-5/4}K^{\frac{d-1}{2}}x^{-\frac{d-1}{2}}e^{-c_2\ell}(\log\log x)^{K_{12}}\\
    &\ll x^{-\frac{d-1}{2}}e^{-c_2\ell}(\log\log x)^{K_{12}}.
\end{align*}
In total, we have a contribution of 
$$C(K)x^{-\frac{d-1}{2}}e^{-c_2\ell}(\log x)(\log\log x)^{d+4}.$$

In summary, using that $-K_3\log\log x\leq \ell\leq K_6\log\log x$ implies $|\ell|\ll \log\log x$ and that $\varphi_{(\log x)^2,\delta}(\ell)\ll 1$, we conclude the following upper bound of \eqref{eq:prob conditioned}:
$$C(K)(\log\log x)^{K_7}(\log x)x^{-\frac{d-1}{2}} e^{-c_2\ell}.$$
The proof is then complete.
\end{proof}

\begin{remark}
    The reason we restrict to $\ell\leq K_6\log\log x$ is that after multiplying the first passage contribution by the particle density from Proposition \ref{prop:particle density} and summing over $\ell$, the $\log\log x$ power term in \eqref{eq:fpc1} explodes. In the following two subsections, we deal with the other case $\ell> K_6\log\log x$.
\end{remark}

\subsubsection{Excluding local barrier events for \texorpdfstring{$\ell> K_6\log\log x$}{}}\label{sec:Excluding unlikely events}
For the case $\ell> K_6\log\log x$, we need to adjust the local barrier events in \eqref{eq:el*}. 
Before this, we remove one more event that the heterogeneity index $h$ is close to $(\log x)^2$ and $g$ is large simultaneously. 
Let $K_9>0$ be a large constant to be determined. Define the event
\begin{align}
    \begin{split}
        \sE_{3,\ell}&:=\bigg(\bigcup_{g\geq -\ell/K_4}\sC_g\bigg)\cap\bigg(\bigcup_{(\log x)^2-K_9\ell\leq h\leq (\log x)^2}\sB_h\bigg)\\
    &=\Big\{g_v\geq -\frac{\ell}{K_4}\Big\}\cap \big\{h_v\geq (\log x)^2-K_9\ell\big\},
    \end{split}\label{eq:E1 def}
\end{align}
where $K_4>0$ is a large constant to be determined.

\begin{lemma}[Removing the event $\sE_{3,\ell}$]\label{lemma:E_1}
    It holds that for some $\delta>0$ and all $\ell> K_6\log\log x$,
    $$\q^{\ell,v}(\exists\, w\in V_{t_{x,K}},\,w\succ v,\,\bet_{w,t_{x,K}}(t_{x,K})\in B_x, \sE_{3,\ell})\ll  e^{-(c_2+\delta/4)\ell}(\log x)^{2(d-1)}x^{-\frac{d-1}{2}}.$$
\end{lemma}

\begin{proof}
We first need an improvement upon Lemma \ref{lemma:hg|l}. For $(\log x)^2-K_9\ell\leq h\leq (\log x)^2$, we have by independence that 
\begin{align}
    \begin{split}
        \q^{\ell,v}(\sB_h\cap\sC_g)
    &\leq \p(M_{(\log x)^2-h}>g-m_{h-K}+m_{(\log x)^2}+\ell)\\
    &\leq \p(M_{(\log x)^2-h}>m_{(\log x)^2-h}+g+\ell).
    \end{split}\label{r}
\end{align}
Since $(\log x)^2-K_9\ell\leq h$, $g\geq -\ell/K_4$, and $\ell>K_6\log\log x$, we have uniformly,
$$r_{g,h,\ell,x}:=\frac{m_{(\log x)^2-h}+g+\ell}{(\log x)^2-h}\geq \frac{c_1K_9\ell-\frac{3}{c_2}\log\log x+(1-K_4^{-1})\ell}{K_9\ell}\geq c_1+\delta_0$$
for some $\delta_0>0$ and all $K_4,K_6$ picked large enough (say, uniformly for all $K_4>2$ and $K_6>10/c_2$). Note that $I$ is strictly convex in a neighborhood of $c_1$, which follows from Theorem 26.3 of \citep{rockafellar1970convex} since (A4) implies that $\xi$ has exponential moments in a neighborhood of $c_2$ and hence $\log\phi_\xi$ is smooth in a neighborhood of $c_2$. Consequently, 
$I(r_{g,h,\ell,x})-I(c_1)\geq (c_2+\delta_1)(r_{g,h,\ell,x}-c_1)$ for some $\delta_1>0$. Let us pick $\ee_0>0$ small enough such that $(c_2+\delta_1)(1-\ee_0)\geq c_2+\delta$ for some $\delta>0$. Then, with $K_4,K_6$ picked large enough depending on $\ee_0$, 
$$r_{g,h,\ell,x}=\frac{m_{(\log x)^2-h}+g+\ell}{(\log x)^2-h}\geq c_1+\frac{g+\ell-\frac{3}{c_2}\log\log x}{(\log x)^2-h}\geq c_1+\frac{(1-\ee_0)(g+\ell)}{(\log x)^2-h}.$$
It follows that
$$I(r_{g,h,\ell,x})-I(c_1)\geq (c_2+\delta_1)(r_{g,h,\ell,x}-c_1)\geq \frac{(c_2+\delta)(g+\ell)}{(\log x)^2-h}.$$
Consequently, by Cram\'{e}r's large deviation upper bound and the union bound, for some $\delta>0$,
\begin{align*}
    \p(M_{(\log x)^2-h}>m_{(\log x)^2-h}+g+\ell)&\ll \rho^{(\log x)^2-h}\p(S_{(\log x)^2-h}>m_{(\log x)^2-h}+g+\ell)\\
    &\ll e^{-((\log x)^2-h)(I(r_{g,h,\ell,x})-I(c_1))}\leq e^{-(c_2+\delta)(g+\ell)}.
\end{align*}
By \eqref{r}, we then arrive at
\begin{align}
   \q^{\ell,v}(\sB_h\cap\sC_g)\leq  \p(M_{(\log x)^2-h}>m_{(\log x)^2-h}+g+\ell)&\leq e^{-(c_2+\delta)(g+\ell)}.\label{eq:hlarge}
\end{align}
By Lemma \ref{lemma:uniform bd} and \eqref{eq:hlarge}, with $K_4$ picked large enough,\footnote{Here we omit the case $|\ell|\gg x^{1/3}$, in which case the first passage probabilities decay exponentially in $\ell$ if $\ell \gg x^{1/3}$, and can be trivially bounded by $1$ if $\ell\ll -x^{1/3}$.} 
\begin{align*}
    &\hspace{0.5cm}\q^{\ell,v}(\exists\, w\in V_{t_{x,K}},\,w\succ v,\,\bet_{w,t_{x,K}}(t_{x,K})\in B_x,\, \sE_{3,\ell})\\
    &\leq \sum_{h=(\log x)^2-K_9\ell}^{(\log x)^2}\sum_{g\geq -\ell/K_4}\q^{\ell,v}(\exists\, w\in V_{t_{x,K}},\,w\succ v,\,\bet_{w,t_{x,K}}(t_{x,K})\in B_x,\, \sE_{3,\ell}\mid\sC_g\cap \sB_h)\\
    &\hspace{3cm}\times\q^{\ell,v}(\sB_h\cap\sC_g)\\
    &\ll \sum_{h=(\log x)^2-K_9\ell}^{(\log x)^2}\sum_{g\geq -\ell/K_4} (\log x)^{2(d-1)} e^{-(c_2+\delta)(g+\ell)}x^{-\frac{d-1}{2}}\\
    &\ll  e^{-(c_2+\delta/4)\ell}(\log x)^{2(d-1)}x^{-\frac{d-1}{2}}.
\end{align*}
    This finishes the proof.
\end{proof}

On the event $\sE_{3,\ell}^{\mathrm{c}}$, we may adjust the barrier event $\sE_{2,\ell}$ as follows. Define for some large constant $K_8$ the following barrier function
\begin{align}
    \psi_{x,K}(k):=\wx+\frac{k-\wtx}{(\log x)^2}m_{(\log x)^2}+\frac{2K_8}{\bl}(\log\min\{k-\wtx,t_{x,K}-k\})_+,~\wtx\leq k\leq t_{x,K}\label{eq:barrier 2}
\end{align}
and the local barrier event
\begin{align}
    \begin{split}
        \sE_{4,\ell}&:=\bigcup_{K\leq h\leq (\log x)^2-K_9\ell}\bigg(\sB_h\cap\Big( \bigcup_{\substack{u\in V_{t_{x}-h}}}\Big(\bigcup_{\wtx\leq k\leq t_x-h}\{\eta_{u,t_x-h}(k)>{\psi}_{x,K}(k)\}\Big)\\
    &\hspace{4cm}\cap\Big(\bigcup_{\substack{w\in V_{t_{x,K}}\\ w\succ u}}\{\bet_{w,t_{x,K}}(t_{x,K})\in B_x\}\Big)\Big)\bigg).
    \end{split}\label{eq:E4}
\end{align}
The following considerations are similar to Section \ref{sec:Local barrier events}, with a different range of $\ell$. To bound the size of $\sE_{4,\ell}$, define
$$T_{\wtx}:=\inf\Big\{k\in[\wtx,t_{x,K}]:\exists\, u\in V_k,u\succ v, \eta_{u,k}(k)\geq \psi_{x,K}(k)\Big\}.$$
By Lemma \ref{lemma:barrier} (with $\beta$ therein given by $\ell$ as the random walk starts from $[\wx-\ell-1,\wx-\ell)$, and coefficient of log replaced by $2K_8/c_2$),
$$\q^{\ell,v}(T_{\wtx}=\wtx+j)\ll \min\{j,(\log x)^2+1-j\}^{-K_8}\ell e^{-c_2\ell}\varphi_{(\log x)^2,\delta}(\ell).$$
We may then calculate the total contribution in the case where the barrier is crossed for a fixed $\ell$, on the event $g\geq -\ell/K_4$ (and hence $h\leq (\log x)^2-K_9\ell$ since we excluded the event $\sE_{3,\ell}$).
By Lemma \ref{lemma:uniform bd} and since the barrier event is measurable at time $t_x-h$, 
\begin{align}\begin{split}
&\hspace{0.5cm}\q^{\ell,v}(\sE_{4,\ell}\cap \sE_{3,\ell}^{\mathrm{c}})\\
    &\leq \sum_{h=K}^{(\log x)^2-K_9\ell}\q^{\ell,v}\bigg(\sB_h\cap \{T_{\wtx}\leq t_x-h\}\cap\Big(\bigcup_{\substack{w\in V_{t_{x,K}}\\ w\succ u}}\{\bet_{w,t_{x,K}}(t_{x,K})\in B_x\}\Big)\Big)\bigg)\\
    &\ll \sum_{h=K}^{(\log x)^2-K_9\ell}h^{d-1}x^{-\frac{d-1}{2}}\ell e^{-c_2\ell}\varphi_{(\log x)^2,\delta}(\ell)\sum_{j=K_9\ell}^{(\log x)^2-h}\min\{j,(\log x)^2+1-j\}^{-K_8}\\
    &\ll x^{-\frac{d-1}{2}}\ell e^{-c_2\ell}\varphi_{(\log x)^2,\delta}(\ell)\sum_{h=K}^{(\log x)^2-K_9\ell}h^{d-1}\max\{\ell^{-K_8/2},h^{-K_8}\}\\
    &\ll x^{-\frac{d-1}{2}}\ell e^{-c_2\ell}\varphi_{(\log x)^2,\delta}(\ell)\big(K^{d-1-K_8}+\ell^{-K_8/2}(\log x)^{2d}\big).
\end{split}\label{eq:E4l prob}
\end{align}

\subsubsection{Local hitting probabilities for \texorpdfstring{$\ell> K_6\log\log x$}{}}
\label{sec:426}

In this section, we prove an upper bound of \eqref{eq:q} in the case $\ell> K_6\log\log x$.
The main improvement compared to Lemma \ref{lemma:ell<loglog} stems from an improvement of Lemma \ref{lemma:hgu|l}, after removing the (unlikely) barrier event $\sE_{4,\ell}$ up to time $t_x-h$ defined in \eqref{eq:E4}. Here and later, we denote by \begin{align}
    \sE_{\ell}:=\sE_{1,\ell}\cup \sE_{3,\ell}\cup \sE_{4,\ell}\label{eq:el}
\end{align} and
\begin{align}
   \begin{split}
        \sK_{h}:=\sB_h\cap\Big( \bigcup_{\substack{u\in V_{t_{x}-h}}}&\Big(\bigcup_{\wtx\leq k\leq t_x-h}\{\eta_{u,t_x-h}(k)>{\psi}_{x,K}(k)\}\Big)\\
    &\cap\Big(\bigcup_{\substack{w\in V_{t_{x,K}}\\ w\succ u}}\{\bet_{w,t_{x,K}}(t_{x,K})\in B_x\}\Big)\Big).
   \end{split}\label{eq:Kh}
\end{align}
Note that $\sE_{4,\ell}=\bigcup_{K\leq h\leq (\log x)^2-K_9\ell}\sK_{h}$, so that $\sE_\ell^{\mathrm{c}}\subseteq \sK_h^{\mathrm{c}}$ for all $k\leq h\leq (\log x)^2-K_9\ell$.

By the same independent sum argument leading to \eqref{eq:u|lhg}, it is not hard to see that
\begin{align}
    \q^{\ell,v}(\sD_\bu\mid\sB_h\cap\sC_g\cap \sK_{h}^{\mathrm{c}})\ll x^{-\frac{d-1}{2}}.\label{eq:u|hgh}
\end{align}
In view of the upper bounds in Lemmas \ref{lemma:ballot+deviation 1} and \ref{lemma:ballot+deviation 2} below (which we will apply with $\ee=1/3$), we fix a large constant $L>0$ and define the following auxiliary function
\begin{align}
    \Phi_{n,\delta}(x,y):=\begin{cases}
xyn^{-3/2}&\text{ if }0\leq y<\frac{\sqrt{n}}{L};\\
   \min\{ xn^{-4/3}+xyn^{-3/2}e^{-\frac{\delta y^2}{n}}, \varphi_{n,\delta}(y)\}&\text{ if } y\geq \frac{\sqrt{n}}{L},
\end{cases}\label{eq:Phi def}
\end{align}
where $x\in[0,O(n^{1/6})]$. The function $\Phi_{n,\delta}(x,y)$ serves as asymptotic upper bounds of ballot probabilities. Recall also the short-hand notation $k_h=(\log x)^2-h$.

\begin{lemma}[size of the event $\sB_h\cap\sC_g\cap \sK_{h}^{\mathrm{c}}$]\label{lemma:hg|l too}
    Assume that $\ell> K_6\log\log x$, $k_h=(\log x)^2-h\geq K_9\ell$, and $g\geq  \max\{-\ell/K_4,-k_h^{1/6}\}$. If $g\geq 0$,
\begin{align}
    \begin{split}
        \q^{\ell,v}(\sB_h\cap\sC_g\cap \sK_{h}^{\mathrm{c}})&\ll e^{-c_2(\ell+g)}\min\{k_h,h\}^{-5/4}k_h^{3/2}e^{-c_1c_2K/2}\\
        &\hspace{3cm}\times\Phi_{k_h,\delta}((\log\min\{k_h,h\}-g)_+,\ell)e^{\frac{K_{10}(\log k_h)(\ell+g)}{k_h}}.
    \end{split}\label{eq:hgl|l1}
\end{align}
If $g<0$, 
\begin{align}
    \begin{split}
\q^{\ell,v}(\sB_h\cap\sC_g\cap \sK_{h}^{\mathrm{c}})\ll e^{-c_2\ell}&\min\{k_h,h\}^{-5/4}k_h^{3/2}(|g|+1)^2e^{c_2g}e^{-c_1c_2K/2}\\
&\times\Phi_{k_h,\delta}((\log\min\{k_h,h\}-g)_+,\ell)e^{\frac{K_{10}(\log k_h)(\ell+g)}{k_h}}.\end{split}\label{eq:hgl|l2}
\end{align}
On the other hand, in the case where $-\ell/K_4<-k_h^{1/6}$, we have for $g\in[-\ell/K_4,-k_h^{1/6})$ that
\begin{align}
    \begin{split}
        &\hspace{0.5cm}\q^{\ell,v}(\sB_h\cap\sC_g\cap \sK_{h}^{\mathrm{c}})\\
        &\ll e^{-c_2\ell}\min\{k_h,h\}^{-5/4}k_h^{3/2}(|g|+1)^2e^{c_2g}e^{-c_1c_2K/2}\varphi_{(\log x)^2,\delta}(\ell)e^{\frac{K_{10}(\log k_h)(\ell+g)}{k_h}}.
    \end{split}\label{eq:hgl|l3}
\end{align}
\end{lemma}

\begin{remark}
   It is instructive to compare Lemma \ref{lemma:hg|l too} with Lemma \ref{lemma:hgu|l}. First, the two results are based on different local ballot events: Lemma \ref{lemma:hgu|l} excludes the event $\sE_{1,\ell}\cup \sE_{2,\ell}$, and Lemma \ref{lemma:hg|l too} excludes the event $\sE_{1,\ell}\cup \sE_{3,\ell}\cup \sE_{4,\ell}$. Second, both results require a lower bound for $k_h$, which greatly helps dealing with the extra term $e^{(c_2-\hat{\lambda})(m_{(\log x)^2}-m_{h-K}-m_{k_h}+\ell+g)}$ in the ballot probabilities. Third, for the case $\ell> K_6\log\log x$, we need extra preciseness in controlling the ballot probabilities, since the bound used in Lemma \ref{lemma:hgu|l} is not tight for $\ell$ large. This stems from Lemmas \ref{lemma:ballot+deviation 1} and \ref{lemma:ballot+deviation 2} in Appendix \ref{sec:ballot ub}, and results in the terms involving the function $\Phi_{k_h,\delta}$ in Lemma \ref{lemma:hg|l too}. The proofs are quite similar, both applying ballot upper bounds under a proper change of measure.
\end{remark}

\begin{proof}
The barrier given by \eqref{eq:barrier 2} starts at location $\wx=(\wx-\ell)+\ell$ at time $\wtx$ (where we recall $\wx=x-m_{(\log x)^2}$) and ends at location 
\begin{align*}
    &\wx+\frac{k_h}{(\log x)^2}m_{(\log x)^2}+\frac{2K_8}{\bl}(\log\min\{k_h,h\})_+\\
    &\hspace{3cm}\leq x-m_{h-K}+g+(O(\log\min\{k_h,h\})-g-\frac{c_1K}{2})
\end{align*}
at time $\wtx+k_h=t_x-h$. Note that since $g\geq -k_h^{1/6}$, we have
$O(\log\min\{k_h,h\})-g-{c_1K}/{2}\ll k_h^{1/6}.$
Define $\hat{\lambda}:=I'(m_{k_h}/k_h)$.  For $g\geq 0$, the ballot upper bounds (Lemmas \ref{lemma:ballot+deviation 1} and \ref{lemma:ballot+deviation 2}, together with Remark \ref{rem}) under a change of measure (identically as in the proof of Proposition \ref{prop:particle density}) then gives 
    \begin{align*}
      &\hspace{0.5cm}\q^{\ell,v}(\sB_h\cap\sC_g\cap \sK_{h}^{\mathrm{c}})\\
      &\ll \rho^{k_h}(\rho^{-k_h}k_h^{3/2}e^{-\hat{\lambda}(x-m_{h-K}+g-(\wx-\ell)-m_{k_h})})\Phi_{k_h,\delta}((\log\min\{k_h,h\}-g)_+,\ell)\\
      &\ll  e^{-c_2(\ell+g)}\min\{k_h,h\}^{-3/2}k_h^{3/2}e^{-c_1c_2K/2}\Phi_{k_h,\delta}((\log\min\{k_h,h\}-g)_+,\ell)\\
      &\hspace{5cm}\times e^{(c_2-\hat{\lambda})(m_{(\log x)^2}-m_{h-K}-m_{k_h}+\ell+g)}.
    \end{align*}
To proceed, we need control of the final term $e^{(c_2-\hat{\lambda})(m_{(\log x)^2}-m_{h-K}-m_{k_h}+\ell+g)}$. Recall from the proof of Proposition \ref{prop:particle density} that $c_2-\hat{\lambda}\leq K_{10}(\log k_h)/k_h$ for some $K_{10}>0$. Since $k_h\geq K_9\ell$ and $\ell>K_6\log\log x\to\infty$ as $x\to\infty$, we may assume that $c_2-\hat{\lambda}\leq c_2/6$ by letting $x$ be large enough. In this case,
$$e^{(c_2-\hat{\lambda})(m_{(\log x)^2}-m_{h-K}-m_{k_h})}\leq e^{c_2(m_{(\log x)^2}-m_{h-K}-m_{k_h})/6}\ll \min\{k_h,h\}^{1/4}.$$
Combining the above leads to \eqref{eq:hgl|l1}.

    For the case $g<0$, we need to exploit the rare event that two independent descendants run distances $m_{h-K}-g$ for time $h-K$ (given by Lemma \ref{lemma:double prob}).  Applying the same arguments as in Lemma \ref{lemma:hg|l} and using the ballot upper bounds (Lemmas \ref{lemma:ballot+deviation 1} and \ref{lemma:ballot+deviation 2}),
    \begin{align*}
         &\hspace{0.5cm}\q^{\ell,v}(\sB_h\cap\sC_g\cap \sK_{h}^{\mathrm{c}})\\
         &\ll \rho^{k_h}\big(\rho^{-k_h}k_h^{3/2}e^{-\hat{\lambda}(x-m_{h-K}+g-(\wx-\ell)-m_{k_h})}\big)\Phi_{k_h,\delta}((\log\min\{k_h,h\}-g)_+,\ell)(|g|+1)^2e^{c_2g}\\
         &\ll  e^{-c_2\ell}\min\{k_h,h\}^{-3/2}k_h^{3/2}(|g|+1)^2e^{c_2g}e^{-c_1c_2K/2}\\
         &\hspace{3cm}\times\Phi_{k_h,\delta}((\log\min\{k_h,h\}-g)_+,\ell)e^{(c_2-\hat{\lambda})(m_{(\log x)^2}-m_{h-K}-m_{k_h}+\ell+g)}.
    \end{align*}
    The remaining follows similarly as the case $g\geq 0$. 

    Finally, for $g\in[-\ell/K_4,-k_h^{1/6})$ we apply the same proof as above while bounding the ballot probability by $\varphi_{(\log x)^2,\delta}(\ell)$ for some $\delta>0$, instead of $\Phi_{k_h,\delta}((\log\min\{k_h,h\}-g)_+,\ell)$. We omit the details here.
\end{proof}

Before proceeding, it is helpful to simplify the quantity $\Phi_{k_h,\delta}((\log\min\{k_h,h\}-g)_+,\ell)$ appearing in Lemma \ref{lemma:hg|l too} a bit. By adjusting the constants $\delta$ and $L$ and since $k_h\leq(\log x)^2$, it holds that
\begin{align}
    \begin{split}
        &\hspace{0.5cm}k_h^{3/2}\Phi_{k_h,\delta}((\log\min\{k_h,h\}-g)_+,\ell)\\
    &\ll \begin{cases}
        \ell (\log\min\{k_h,h\}-g)_+&\text{ if }\ell\in( K_6\log\log x,\frac{\log x}{L});\\
        (\log\min\{k_h,h\}-g)_+((\log x)^{1/3}+\ell e^{-\frac{\delta\ell^2}{(\log x)^2}})&\text{ if }\frac{\log x}{L}\leq \ell\leq L\log x\log\log x;\\
        \varphi_{(\log x)^2,\delta}(\ell)&\text{ if }\ell> L\log x\log\log x.
    \end{cases}
    \end{split}\label{eq:Phi}
\end{align}
 Recall also \eqref{eq:Psi} and \eqref{eq:el}.

\begin{lemma}[first passage contribution]\label{lemma:ell>loglog}
For $\ell> K_6\log\log x$, there exists $C(K)>0$ such that
    \begin{align*}
        &\hspace{0.5cm}\q^{\ell,v}(\exists\, w\in V_{t_{x,K}},w\succ v,\,\bet_{w,t_{x,K}}(t_{x,K})\in B_x,\, \sE_{\ell}^{\mathrm{c}})\\
        &\ll e^{-c_1c_2K/4} x^{-\frac{d-1}{2}} e^{-c_2\ell}\Psi_{(\log x)^2,\delta}(\ell)\big(e^{2K_{10}\frac{\ell\log\log x}{(\log x)^2}}+(\log\log x)^{d+3}\ell^{-1/8}e^{\frac{\log\ell\log\log x}{8\ell}}\big)\\
        &\hspace{5cm}+C(K)x^{-\frac{d-1}{2}}(\log x)^3e^{-c_2\ell}e^{-c_2\ell/(2K_4)},
    \end{align*}
    where the asymptotic constant in $\ll$ does not depend on $K,x,\ell$ but may depend on $\ee$.
\end{lemma}
\begin{proof}
In the following, the asymptotic constants in $\ll$ may depend on $\ee$ but not on $K$. We condition on $\sB_h\cap\sC_g$ and divide into four cases: $g\geq K_5\log h$, $0\leq g<K_5\log h$, $-\ell/K_4\leq g<0$, and $g<-\ell/K_4$.\footnote{The three cases $g\geq K_5\log h$, $0\leq g<K_5\log h$, and $g<0$ as discussed in the proof of Lemma \ref{lemma:ell<loglog} do not suffice due to \eqref{eq:E1 def}.}

Case (a): $g\geq K_5\log h$. We apply Lemma \ref{lemma:uniform bd} and \eqref{eq:hgl|l1} of Lemma \ref{lemma:hg|l too} to obtain
\begin{align*}
    &\hspace{0.5cm}\sum_{h=K}^{(\log x)^2-K_9\ell}\sum_{g>K_5\log h}h^{d-1}   x^{-\frac{d-1}{2}}\q^{\ell,v}(\sB_h\cap\sC_g\cap \sK_{h}^{\mathrm{c}})\\
    &\ll x^{-\frac{d-1}{2}}e^{-c_2\ell}e^{-c_1c_2K/2}\sum_{h=K}^{(\log x)^2-K_9\ell}h^{d-1}\min\{k_h,h\}^{-5/4}k_h^{3/2}\\
    &\hspace{3cm}\times\sum_{g>K_5\log h}   e^{-c_2g}\Phi_{k_h,\delta}((\log\min\{k_h,h\}-g)_+,\ell)e^{\frac{K_{10}(\log k_h)(\ell+g)}{k_h}}\\
    &\ll x^{-\frac{d-1}{2}}e^{-c_2\ell}e^{-c_1c_2K/2}\sum_{h=K}^{(\log x)^2-K_9\ell}h^{d-1-c_2K_5}\min\{k_h,h\}^{-5/4}k_h^{3/2}\\
    &\hspace{3cm}\Phi_{k_h,\delta}((\log\min\{k_h,h\})_+,\ell)e^{K_{10}\frac{(\log k_h)(\ell+K_5\log h)}{k_h}}. 
\end{align*}
We split the sum over $h$ depending on whether $k_h\leq h$ or $k_h>h$ (recalling the definition that $k_h=(\log x)^2-h$). We have for $K_5$ large enough, by \eqref{eq:k large} below,
\begin{align*}
    &\hspace{0.5cm}\sum_{h=K}^{(\log x)^2/2}h^{d-1-c_2K_5}\min\{k_h,h\}^{-5/4}k_h^{3/2}\Phi_{k_h,\delta}((\log\min\{k_h,h\})_+,\ell)e^{K_{10}\frac{(\log k_h)(\ell+K_5\log h)}{k_h}}\\
    &\ll \sum_{h=K}^{(\log x)^2/2}h^{-100}k_h^{3/2}\Phi_{k_h,\delta}((\log h)_+,\ell)e^{2K_{10}\frac{\ell\log\log x}{(\log x)^2}}\\
    &\ll \Psi_{(\log x)^2,\delta}(\ell)e^{2K_{10}\frac{\ell\log\log x}{(\log x)^2}},
\end{align*}
and
\begin{align*}
    &\hspace{0.5cm}\sum_{h=(\log x)^2/2}^{(\log x)^2-K_9\ell}h^{d-1-c_2K_5}\min\{k_h,h\}^{-5/4}k_h^{3/2}\Phi_{k_h,\delta}((\log\min\{k_h,h\})_+,\ell)e^{K_{10}\frac{(\log k_h)(\ell+K_5\log h)}{k_h}}\\
    &\ll \sum_{h=(\log x)^2/2}^{(\log x)^2-K_9\ell}h^{-100}k_h^{-5/4}k_h^{3/2}\Phi_{k_h,\delta}((\log k_h)_+,\ell)e^{K_{10}\frac{(\log k_h)(\ell+2K_5\log \log x)}{k_h}}\\
    &\ll (\log x)^{-100}\sum_{h=(\log x)^2/2}^{(\log x)^2-K_9\ell}k_h^{-5/4}k_h^{3/2}\Phi_{k_h,\delta}((\log k_h)_+,\ell)\ell^{K_{10}/K_9}\\
    &\ll (\log x)^{-100}\Psi_{(\log x)^2,\delta}(\ell)\ell^{K_{10}/K_9-1/4}\ll \Psi_{(\log x)^2,\delta}(\ell),
\end{align*}
where we have used $k_h\geq K_9\ell$ and $\ell> K_6\log\log x\to\infty$ in the second step, and $K_9$ picked large enough (depending only on $K_{10}$) in the last step. Altogether, we conclude that
\begin{align*}
&\sum_{h=K}^{(\log x)^2-K_9\ell}\sum_{g>K_5\log h}h^{d-1}   x^{-\frac{d-1}{2}}\q^{\ell,v}(\sB_h\cap\sC_g\cap \sK_{h}^{\mathrm{c}})\\
&\hspace{3cm}\ll     x^{-\frac{d-1}{2}}e^{-c_2\ell}e^{-c_1c_2K/2}\Psi_{(\log x)^2,\delta}(\ell)e^{2K_{10}\frac{\ell\log\log x}{(\log x)^2}}.
\end{align*}

In the other three cases, we will use, within the corresponding regions for $g$,
\begin{align}
    &\hspace{0.5cm}\q^{\ell,v}(\exists\, w\in V_{t_{x,K}},w\succ v,\,\bet_{w,t_{x,K}}(t_{x,K})\in B_x,\, \sE_{\ell}^{\mathrm{c}})\nonumber\\
    \begin{split}
        &\leq  \sum_{h=K}^{(\log x)^2}\sum_{g\in\bZ}\sum_{\bu\in\bZ^{d-1}} \q^{\ell,v}(\exists\, w\in V_{t_{x,K}},\,w\succ v,\,\bet_{w,t_{x,K}}(t_{x,K})\in B_x,\, \sE_{\ell}^{\mathrm{c}}\mid\sD_\bu\cap \sC_g\cap \sB_h\cap  \sK_{h}^{\mathrm{c}})\\
    &\hspace{4cm}\times\q^{\ell,v}(\sB_h\cap\sC_g\cap\sD_\bu\cap \sK_{h}^{\mathrm{c}}).
    \end{split}
    \label{eq:hguh}
\end{align}
For the middle two cases $-\ell/K_4\leq g<K_5\log h$, the asymptotic constants do not depend on the lag time $K>0$ (recall that $t_{x,K}=t_x-K$). 
The first conditional probability in each summand in \eqref{eq:hguh} can be controlled in the same way as \eqref{eq:x|ghlu2}. We may also apply the same argument in Lemma \ref{lemma:hgu|l}, using now \eqref{eq:u|hgh}, to obtain bounds on $\q^{\ell,v}(\sB_h\cap\sC_g\cap\sD_\bu\cap \sK_{h}^{\mathrm{c}})$. This amounts to multiplying the right-hand sides of \eqref{eq:hgl|l1} and \eqref{eq:hgl|l2} by $x^{-\frac{d-1}{2}}$. 

Case (b): $0\leq g<K_5\log h$. Recall that on the event $\sE_{3,\ell}^c$ we removed $h\in[(\log x)^2-K_9\ell,(\log x)^2]$. We further scrutinize the term $e^{K_{10}(\log k_h)(\ell+g)/k_h}$ that appears in \eqref{eq:hgl|l1}. Note that $K_{10}$ here does not depend on the other constants $K_1,\dots,K_9$. Since $k_h\geq K_9\ell$, $0\leq g<K_5\log h$, and $h\leq (\log x)^2$, we have for $k_h\leq h$,
\begin{align}
    e^{K_{10}\frac{(\log k_h)(\ell+g)}{k_h}}\leq e^{\frac{2K_{10}\log k_h\log \log x}{k_h}+\frac{K_{10}\ell\log k_h}{K_9\ell}}\leq e^{\frac{\log\ell\log\log x}{8\ell}} k_h^{1/8}\label{eq:k small}
\end{align}
for $K_9$ picked large enough depending only on $K_{10}$. For $k_h\geq h$ (and hence $k_h\geq (\log x)^2/2$ and $g/k_h=O(1)$), 
\begin{align}
    e^{K_{10}\frac{(\log k_h)(\ell+g)}{k_h}}\ll e^{K_{10}\frac{(\log k_h)\ell}{k_h}}\leq e^{2K_{10}\frac{\ell\log\log x}{(\log x)^2}}.\label{eq:k large}
\end{align}

For $\n{\bu}\leq \sqrt{h}\log h$ and $0\leq g<K_5\log h$ (note that we may start the sum from $h=K$), we compute using \eqref{eq:hgl|l1} of Lemma \ref{lemma:hg|l too} in the first step and \eqref{eq:k small} and \eqref{eq:k large} in the third step that
\begin{align*}
    &\hspace{0.5cm}\sum_{h=K}^{(\log x)^2-K_9\ell}\sum_{g=0}^{K_5\log h}\sum_{\substack{\bu\in\bZ^{d-1}\\ \n{\bu}\leq \sqrt{h}\log h}} \q^{\ell,v}(\sB_h\cap\sC_g\cap\sD_\bu\cap \sK_{h}^{\mathrm{c}})\\
    &\hspace{2cm}\times\q^{\ell,v}(\exists\, w\in V_{t_{x,K}},\,w\succ v,\,\bet_{w,t_{x,K}}(t_{x,K})\in B_x,\, \sE_{\ell}^{\mathrm{c}}\mid\sD_\bu\cap \sC_g\cap \sB_h\cap \sK_{h}^{\mathrm{c}})\\
    &\ll \sum_{h=K}^{(\log x)^2-K_9\ell}\sum_{g=0}^{K_5\log h}\sum_{\substack{\bu\in\bZ^{d-1}\\ \n{\bu}\leq \sqrt{h}\log h}} ((g+1)e^{c_2g}(\log h)^2(h-K+1)^{-\frac{d-1}{2}})\\
    &\hspace{2cm}\times\big(x^{-\frac{d-1}{2}} e^{-c_2\ell}\min\{k_h,h\}^{-5/4}e^{-c_2g}e^{-c_1c_2K/2}\\
    &\hspace{4cm}\times k_h^{3/2}\Phi_{k_h,\delta}((\log\min\{k_h,h\}-g)_+,\ell)e^{\frac{K_{10}(\log k_h)(\ell+g)}{k_h}}\big)\\
    &\ll e^{-c_1c_2K/2}x^{-\frac{d-1}{2}} e^{-c_2\ell}\sum_{h=K}^{(\log x)^2-K_9\ell}(\log h)^{d+3}k_h^{3/2}\Phi_{k_h,\delta}((\log\min\{k_h,h\})_+,\ell)\\
    &\hspace{1cm}\times\min\{k_h,h\}^{-5/4}\Big(\frac{h}{h-K+1}\Big)^{\frac{d-1}{2}}e^{K_{10}\frac{(\log k_h)\ell}{k_h}}\\
    &\ll e^{-c_1c_2K/2}x^{-\frac{d-1}{2}} e^{-c_2\ell}\bigg(\sum_{h=(\log x)^2/2}^{(\log x)^2-K_9\ell}(\log\log x)^{d+3}k_h^{3/2}\Phi_{k_h,\delta}((\log k_h)_+,\ell)k_h^{-9/8}e^{\frac{\log\ell\log\log x}{8\ell}}\\
    &\hspace{1.5cm}+\sum_{h=K}^{(\log x)^2/2}(\log h)^{d+3}k_h^{3/2}\Phi_{k_h,\delta}((\log h)_+,\ell)h^{-5/4}\Big(\frac{h}{h-K+1}\Big)^{\frac{d-1}{2}}e^{2K_{10}\frac{\ell\log\log x}{(\log x)^2}}\bigg)\\
    &\ll e^{-c_1c_2K/4} x^{-\frac{d-1}{2}} e^{-c_2\ell}\Psi_{(\log x)^2,\delta}(\ell)\big(e^{2K_{10}\frac{\ell\log\log x}{(\log x)^2}}+(\log\log x)^{d+3}\ell^{-1/8}e^{\frac{\log\ell\log\log x}{8\ell}}\big),
\end{align*}
where in the last step we used \eqref{eq:Phi} and \eqref{eq:Psi}.
The case of $\n{\bu}>\sqrt{h}\log h$ can be dealt with in the same way as in Section \ref{sec:prelim bound}, leading to a contribution of $x^{-\frac{d-1}{2}}\min\{e^{-c_2\ell},1\}\varphi_{(\log x)^2,\delta}(\ell)$.

Case (c): $-\ell/K_4\leq g<0$. Similarly as the way we derived \eqref{eq:k small} and \eqref{eq:k large}, we have
\begin{align}
    e^{K_{10}\frac{(\log k_h)(\ell+g)}{k_h}}\leq e^{K_{10}\frac{(\log k_h)\ell}{k_h}}\ll\begin{cases}
        e^{2K_{10}\frac{\ell\log\log x}{(\log x)^2}}&\text{ if }k_h\geq h;\\
        e^{\frac{\log\ell\log\log x}{8\ell}}&\text{ if }k_h< h.
    \end{cases}\label{eq:k2}
\end{align}
We first consider the subcase where $\max\{-\ell/K_4,-k_h^{1/6}\}\leq g<0$. Again the sum over $\n{\bu}>\sqrt{h}\log h$ can be controlled similarly as in case (c) in the proof of Lemma \ref{lemma:ell<loglog} (which only used $g<0$ but has no constraint on the range of $\ell$). Recall that restricting to the event $\sE_{3,\ell}^{\mathrm{c}}$ allows us to remove the sum over $(h,g)$ such that $g>-\ell/K_4$ and $h>(\log x)^2-K_9\ell$. Applying \eqref{eq:hgl|l2} of Lemma \ref{lemma:hg|l too} and \eqref{eq:x|ughl for g<0} in the first step, that $\#\{\bu\in\bZ^{d-1}:\,\n{\bu}\leq \sqrt{h}\log h\}\ll h^{\frac{d-1}{2}}(\log h)^{d-1}$ in the second step, and \eqref{eq:k2} in the fourth step, we have 
\begin{align*}
    &\hspace{0.5cm}\sum_{h=K}^{(\log x)^2-K_9\ell}\sum_{\max\{-\ell/K_4,-k_h^{1/6}\}\leq g<0}\sum_{\substack{\bu\in\bZ^{d-1}\\ \n{\bu}\leq \sqrt{h}\log h}} \q^{\ell,v}(\sB_h\cap\sC_g\cap\sD_\bu\cap \sK_{h}^{\mathrm{c}})\\
    &\hspace{2cm}\times\q^{\ell,v}(\exists\, w\in V_{t_{x,K}},\,w\succ v,\,\bet_{w,t_{x,K}}(t_{x,K})\in B_x,\, \sE_{\ell}^{\mathrm{c}}\mid\sD_\bu\cap \sC_g\cap \sB_h\cap \sK_{h}^{\mathrm{c}})\\
    &\ll \sum_{h=K}^{(\log x)^2-K_9\ell}\sum_{\max\{-\ell/K_4,-k_h^{1/6}\}\leq g<0}\sum_{\substack{\bu\in\bZ^{d-1}\\ \n{\bu}\leq \sqrt{h}\log h}}((\log h)^2 (h-K+1)^{-\frac{d-1}{2}})\\
    &\hspace{0.5cm}\times\big(x^{-\frac{d-1}{2}} e^{-c_2\ell}\min\{k_h,h\}^{-5/4}(|g|+1)^2e^{c_2g}e^{-c_1c_2K/2}\\
    &\hspace{2cm}\times k_h^{3/2}\Phi_{k_h,\delta}((\log\min\{k_h,h\}-g)_+,\ell)e^{\frac{K_{10}(\log k_h)(\ell+g)}{k_h}}\big)\\
    &\ll e^{-c_1c_2K/4} x^{-\frac{d-1}{2}} e^{-c_2\ell}\sum_{h=K}^{(\log x)^2-K_9\ell}(\log h)^{d+1}\min\{k_h,h\}^{-5/4}e^{K_{10}\frac{(\log k_h)\ell}{k_h}}\\
    &\hspace{2cm}\times\sum_{\max\{-\ell/K_4,-k_h^{1/6}\}\leq g<0}(|g|+1)^2e^{c_2g}k_h^{3/2}\Phi_{k_h,\delta}((\log\min\{k_h,h\}-g)_+,\ell)\\
    &\ll e^{-c_1c_2K/4} x^{-\frac{d-1}{2}} e^{-c_2\ell}\\
    &\hspace{1cm}\times\sum_{h=K}^{(\log x)^2-K_9\ell}(\log h)^{d+1}\min\{k_h,h\}^{-5/4}e^{K_{10}\frac{(\log k_h)\ell}{k_h}}k_h^{3/2}\Phi_{k_h,\delta}((\log\min\{k_h,h\})_+,\ell)\\
    &\ll e^{-c_1c_2K/4} x^{-\frac{d-1}{2}} e^{-c_2\ell}\bigg(\sum_{h=K}^{(\log x)^2/2}(\log h)^{d+1}h^{-5/4}e^{2K_{10}\frac{\ell\log\log x}{(\log x)^2}}k_h^{3/2}\Phi_{k_h,\delta}((\log h)_+,\ell)\\
    &\hspace{3cm}+\sum_{h=(\log x)^2/2}^{(\log x)^2-K_9\ell}(\log h)^{d+1}k_h^{-5/4}e^{\frac{\log\ell\log\log x}{8\ell}}k_h^{3/2}\Phi_{k_h,\delta}((\log k_h)_+,\ell)\bigg)\\
    &\ll e^{-c_1c_2K/4}x^{-\frac{d-1}{2}} e^{-c_2\ell}\Psi_{(\log x)^2,\delta}(\ell)\big(e^{2K_{10}\frac{\ell\log\log x}{(\log x)^2}}+(\log\log x)^{d+3} \ell^{-1/4}e^{\frac{\log\ell\log\log x}{8\ell}}\big),
\end{align*}
where in the last step we applied \eqref{eq:Phi} and \eqref{eq:Psi}. In the other subcase where $-\ell/K_4\leq g<-k_h^{1/6}$, applying  \eqref{eq:hgl|l3} of Lemma \ref{lemma:hg|l too} yields 
\begin{align*}
    &\hspace{0.5cm}\sum_{h=K}^{(\log x)^2-K_9\ell}\sum_{-\ell/K_4\leq g<-k_h^{1/6}}\sum_{\substack{\bu\in\bZ^{d-1}\\ \n{\bu}\leq \sqrt{h}\log h}} \q^{\ell,v}(\sB_h\cap\sC_g\cap\sD_\bu\cap \sK_{h}^{\mathrm{c}})\\
    &\hspace{2cm}\times\q^{\ell,v}(\exists\, w\in V_{t_{x,K}},\,w\succ v,\,\bet_{w,t_{x,K}}(t_{x,K})\in B_x,\, \sE_{\ell}^{\mathrm{c}}\mid\sD_\bu\cap \sC_g\cap \sB_h\cap \sK_{h}^{\mathrm{c}})\\
    &\ll e^{-c_1c_2K/4} x^{-\frac{d-1}{2}} e^{-c_2\ell}\sum_{h=K}^{(\log x)^2-K_9\ell}(\log h)^{d+1}\min\{k_h,h\}^{-5/4}\\
        &\hspace{3cm}\sum_{-\ell/K_4\leq g<-k_h^{1/6}}(|g|+1)^2e^{c_2g}k_h^{3/2}e^{\frac{K_{10}(\log k_h)(\ell+g)}{k_h}}\\
    &\ll e^{-c_1c_2K/4} x^{-\frac{d-1}{2}} e^{-c_2\ell}\sum_{h=K}^{(\log x)^2-K_9\ell}(\log h)^{d+1}\min\{k_h,h\}^{-5/4}\\
    &\hspace{6cm}\times k_h^{-100}\varphi_{(\log x)^2,\delta}(\ell)e^{\frac{K_{10}(\log k_h)(\ell+g)}{k_h}}\\
    &\ll e^{-c_1c_2K/4}x^{-\frac{d-1}{2}} e^{-c_2\ell}\Psi_{(\log x)^2,\delta}(\ell)\big(e^{2K_{10}\frac{\ell\log\log x}{(\log x)^2}}+(\log\log x)^{d+3} \ell^{-1/4}e^{\frac{\log\ell\log\log x}{8\ell}}\big),
\end{align*}
where in the last step we also used $\varphi_{(\log x)^2,\delta}(\ell)\ll \Psi_{(\log x)^2,\delta}(\ell)$ for $\ell> K_6\log\log x$.

Case (d): $g<-\ell/K_4$. In this case, we cannot use Lemma \ref{lemma:hg|l too}, but we use Lemma \ref{lemma:hgu|l}. The sum over $\n{\bu}>\sqrt{h}\log h$ can be controlled similarly as in case (c) in the proof of Lemma \ref{lemma:ell<loglog}. We obtain using \eqref{eq:x|ughl for g<0} and Lemma \ref{lemma:hgu|l} that
\begin{align*}
     &\hspace{0.5cm}\sum_{h=K}^{(\log x)^2}\sum_{g<-\ell/K_4}\sum_{\substack{\bu\in\bZ^{d-1}\\ \n{\bu}\leq \sqrt{h}\log h}} \q^{\ell,v}(\sB_h\cap\sC_g\cap\sD_\bu\cap \sK_{h}^{\mathrm{c}})\\
    &\hspace{2cm}\times\q^{\ell,v}(\exists\, w\in V_{t_{x,K}},\,w\succ v,\,\bet_{w,t_{x,K}}(t_{x,K})\in B_x\mid\sD_\bu\cap \sC_g\cap \sB_h\cap \sK_{h}^{\mathrm{c}})\\
    &\leq \sum_{h=K}^{(\log x)^2}\sum_{g<-\ell/K_4}\sum_{\substack{\bu\in\bZ^{d-1}\\ \n{\bu}\leq \sqrt{h}\log h}} \q^{\ell,v}(\exists\, w\in V_{t_{x,K}},\,w\succ v,\,\bet_{w,t_{x,K}}(t_{x,K})\in B_x\mid\sD_\bu\cap \sC_g\cap \sB_h)\\
        &\hspace{3cm}\times\q^{\ell,v}(\sB_h\cap\sC_g\cap\sD_\bu)\\
    &\ll x^{-\frac{d-1}{2}}\sum_{h=K}^{(\log x)^2}\sum_{g<-\ell/K_4}\sum_{\substack{\bu\in\bZ^{d-1}\\ \n{\bu}\leq \sqrt{h}\log h}} ((\log h)^2 (h-K+1)^{-\frac{d-1}{2}})\\
    &\hspace{2cm}\times\Big(\min\Big\{1,(|g+\ell|+1)e^{-c_2(g+\ell)}\varphi_{(\log x)^2,\delta}(g+\ell)\Big\}\min\{1,((|g|+1)e^{c_2g})^2\}\Big)\\
    &\ll x^{-\frac{d-1}{2}}\sum_{h=K}^{(\log x)^2}(\log h)^{d+1} \sum_{g<-\ell/K_4}\min\Big\{1,(|g+\ell|+1)e^{-c_2(g+\ell)}\varphi_{(\log x)^2,\delta}(g+\ell)\Big\}\\
        &\hspace{3cm}\times(|g|+1)^2e^{2c_2g}\Big(\frac{h}{h-K+1}\Big)^{\frac{d-1}{2}}\\
    &\ll x^{-\frac{d-1}{2}}\sum_{h=K}^{(\log x)^2}(\log h)^{d+1} \Big(\ell^2e^{-2c_2\ell}+\ell^3e^{-c_2\ell}e^{-c_2\ell/K_4}\Big)\Big(\frac{h}{h-K+1}\Big)^{\frac{d-1}{2}}\\
    &\ll C(K) x^{-\frac{d-1}{2}}(\log x)^3e^{-c_2\ell}e^{-c_2\ell/(2K_4)},
\end{align*}
where in the last step we bounded $(\log h)^{d+1}$ by $\log x$.

Combining the above four cases finishes the proof.
\end{proof}

\subsubsection{Combining everything above---proof of Theorem \texorpdfstring{\ref{thm:particle probability}}{}}\label{sec:Combining everything above}
Our goal is to bound from above the quantity
$$\q^{\ell,v}(\exists\, w\in V_{t_{x,K}},w\succ v,\,\bet_{w,t_{x,K}}(t_{x,K})\in B_x).$$
We divide into three cases according to the range of $\ell$.

\begin{itemize}
    \item The case $\ell<-K_3\log\log x$. We apply Lemma \ref{lemma:uniform bd} to bound directly
$$\q^{\ell,v}(\exists\, w\in V_{t_{x,K}},w\succ v,\,\bet_{w,t_{x,K}}(t_{x,K})\in B_x)\ll (\log x)^{2(d-1)}x^{-\frac{d-1}{2}}.$$

\item The case $-K_3\log\log x\leq \ell\leq K_6\log\log x$. We get from Lemma \ref{lemma:ell<loglog} that while excluding the local ballot event $\sE_{\ell}^*$, the remaining satisfies
\begin{align*}
    &\q^{\ell,v}(\exists\, w\in V_{t_{x,K}},w\succ v,\,\bet_{w,t_{x,K}}(t_{x,K})\in B_x,\,(\sE_\ell^*)^{\mathrm{c}})\\
    &\ll C(\ee,K)(\log\log x)^{K_7}(\log x)x^{-\frac{d-1}{2}} e^{-c_2\ell}.
\end{align*}
Combining this with \eqref{eq:E3l prob} and \eqref{eq:E2 bound} yields Theorem \ref{thm:particle probability} for $-K_3\log\log x\leq \ell\leq K_6\log\log x$.

\item The case 
 $\ell> K_6\log\log x$. The event $\sE_{3,\ell}$ contributes 
\begin{align}
    e^{-(c_2+\delta/4)\ell}(\log x)^{2(d-1)}x^{-\frac{d-1}{2}}\label{eq:E_1 contribution}
\end{align}
by Lemma \ref{lemma:E_1}. We get from Lemma \ref{lemma:ell>loglog} that while excluding the ballot event $\sE_\ell$, the remaining satisfies
\begin{align*}
        &\hspace{0.5cm}\q^{\ell,v}(\exists\, w\in V_{t_{x,K}},w\succ v,\,\bet_{w,t_{x,K}}(t_{x,K})\in B_x,\, \sE_{\ell}^{\mathrm{c}})\\
        &\ll C(\ee)e^{-c_1c_2K/4} x^{-\frac{d-1}{2}} e^{-c_2\ell}\Psi_{(\log x)^2,\delta}(\ell)\big(e^{2K_{10}\frac{\ell\log\log x}{(\log x)^2}}+(\log\log x)^{d+3}\ell^{-1/8}e^{\frac{\log\ell\log\log x}{8\ell}}\big)\\
        &\hspace{7cm}+C(K)x^{-\frac{d-1}{2}}(\log x)^3e^{-c_2\ell}e^{-c_2\ell/(2K_4)},
    \end{align*}
    The second term $C(K)x^{-\frac{d-1}{2}}(\log x)^3e^{-c_2\ell}e^{-c_2\ell/(2K_4)}$ on the right-hand side above together with \eqref{eq:E_1 contribution} contribute
at most    $C(\ee,K)e^{-(c_2+\delta/4)\ell}(\log x)^{3d}x^{-\frac{d-1}{2}}$
    for some $\delta>0$. Taking into account \eqref{eq:E3l prob}, Lemma \ref{lemma:E_1}, and \eqref{eq:E4l prob} leads to the desired bound in Theorem \ref{thm:particle probability} for $\ell> K_6\log\log x$.

\end{itemize}

These considerations conclude the proof, given the definition of $I_{\ell,x}$ above Theorem \ref{thm:particle probability}.

\section{Further discussions}
\label{sec:discussions}
\subsection{The general non-spherically-symmetric case}\label{sec:non-sphere}

Our cluster approach in the spherically-symmetric case (Theorem \ref{thm:main}) also extends to the general case (Theorem \ref{thm:main2}). Unfortunately, the notation becomes much heavier, although it is intuitively clear how the proof can be modified, as we will explain below.\footnote{There are a few exceptions where the proof is direct, if the law of $\bxi$ is an invertible linear transformation of a spherically-symmetric one, such as centered Gaussian with an invertible covariance matrix. We refer to Section 3.2 of \citep{blanchet2024first} for a detailed account of this approach.} 
We have chosen to focus on the spherically-symmetric case for clarity of our presentation and to avoid repetitions. In this section, we comment on the necessary changes to prove Theorem \ref{thm:main2}, and the details are left to the interested reader. \\

\textbf{Overview.}
Essentially, our approach reduces the study of the $d$-dimensional BRW to that of its projection onto a certain one-dimensional direction (pivot) $\bc_2\in\R^d$ such that the first passage event to $B_x$ is reasonably close to a certain first passage event of the \textit{projected} BRW (to the projection of the target $B_x$). 
The underlying mechanism of this approximation is that the range of a non-symmetric BRW grows roughly as a convex shape that linearly expands in time, and $\bc_2$ is the normal vector to the tangent hyperplane between the convex shape and the ball $B_x$.  
 For example, in the spherically-symmetric case, $\bc_2$ is along the direction of the first coordinate (i.e., a constant multiple of $\be_1$). The non-symmetric case requires the same techniques, up to finding the correct pivot $\bc_2$ based on the rate function $\widehat{I}(\bxi)$, which is given below \eqref{eq:I long}. 

 Let us formulate the new projection. In the spherically-symmetric case, we follow the projection
 $$\bet_{v,n}(k)=\eta_{v,n}(k)\,\be_1+(0,\widehat{\bet}_{v,n}(k)).$$
Instead, we now decompose
 $$\bet_{v,n}(k)=\eta_{v,n}^{\bc_2}(k)\,\bc_2+\widehat{\bet}_{v,n}^{\bc_2}(k),$$
 where $\widehat{\bet}_{v,n}^{\bc_2}(k)\in\R^d$ is perpendicular to $\bc_2$, i.e., $\widehat{\bet}_{v,n}^{\bc_2}(k)\cdot\bc_2=0$. The proof of Theorem \ref{thm:main2} is mostly verbatim, while in the two paragraphs below we spell out a few details that differ from the proof of Theorem \ref{thm:main}. \\
 
\textbf{Identifying the constants in the preliminary results.}
Let us re-discover the formula \eqref{eq:fpt2} based on a calculation using the BRW projected onto $\bc_2$. The one-step jump distribution is $\xi^{\bc_2}:=\bxi\cdot\bc_2$. Using the definition $\widehat{I}(\widehat{c}_1\be_1)=\log\rho$ and $\bc_2=\nabla \widehat{I}(\widehat{c}_1\be_1)$, we have
\begin{align*}
    \sup_{\bla\in\R^d}\Big(\widehat{c}_1\bla\cdot \be_1-\log\E[e^{\bla\cdot\bxi}]\Big)=\log\rho\quad\text{and}\quad \widehat{c}_1\be_1=\frac{\E[\bxi e^{\bc_2\cdot\bxi}]}{\E[e^{\bc_2\cdot\bxi}]}.
\end{align*}
It is then straightforward to check that the supremum in
$$I^{\xi^{\bc_2}}(\widehat{c}_1\be_1\cdot \bc_2)=\sup_{\lambda\in\R}\Big(\lambda \widehat{c}_1\be_1\cdot \bc_2-\log\E[e^{\lambda\xi^{\bc_2}}]\Big)$$
is attained at $\widetilde{\lambda}=1$ and the value of the supremum is $\log\rho$. This has two consequences. First, the linear speed of the BRW with jump $\xi^{\bc_2}$ is $\widehat{c}_1\be_1\cdot \bc_2$, which gives the linear coefficient in \eqref{eq:fpt2}. Second, when dealing with the projected BRW, the analogue of the constant $c_2$ in the results presented in Section \ref{sec:one-dim results} becomes $\widetilde{\lambda}=1$. Consequently, the logarithm correction term is $$\frac{d+2}{2\widetilde{\lambda} \widehat{c}_1\be_1\cdot \bc_2}\log x=\frac{d+2}{2 \widehat{c}_1\be_1\cdot \bc_2}\log x,$$  
 giving the logarithmic correction term in \eqref{eq:fpt2}.\\

\textbf{Conditional local CLT in the direction $\bc_2$.}
 While the cluster structure remains unchanged for the one-dimensional BRW projected onto $\bc_2$, certain modification is required to turn the size of the clusters to the local hitting probabilities (that is, given a trajectory that advances in the direction $\bc_2$, we compute the chance that it reaches the ball $B_x$). In the spherically-symmetric case, this is driven by the conditional local CLT (Lemmas \ref{lemma:conditioned local CLT} and \ref{lemma:uniform local CLT}). The selection of the vector $\bc_2$ is exactly such that the analogous local CLT holds in the new direction $\bc_2$. We showcase this by providing the proof to a more general version of Lemma \ref{lemma:conditioned local CLT}, given by Lemma \ref{lemma:conditioned local CLT2} below. The same extension to Lemma \ref{lemma:uniform local CLT} can be done similarly. 

 \subsection{Proof of Corollary \texorpdfstring{\ref{coro}}{}}\label{sec:proof of coro}

We provide the proof for the spherically symmetric case, as the more general case follows from the same considerations in Section \ref{sec:non-sphere}. Recall that our goal is to study the random quantity $\widetilde{M}_n$, which is the maximum location among particles within unit distance of the first axis. Denote by
$$\widetilde{m}_n:=\widehat{c}_1n-\frac{d+2}{2\,\be_1\cdot\bc_2}\,\log n,$$
which is the deterministic part of the conjectural asymptotics in \eqref{eq:wMn}. Recall \eqref{eq:txdef}. Observe that $\widetilde{m}_{t_x}=x+O(1)$ as $x\to\infty$ and $t_{\widetilde{m}_n}=n+O(1)$ as $n\to\infty$. 

To see the lower bound of $\widetilde{M}_n$, note from Section \ref{sec:UB} that we have shown that for some $C>0$,
$$\p\Big(\exists v\in V_{t_x},\,\bet_{v,t_x}(t_x)\in [x-\frac{1}{2},x+\frac{1}{2}]\times B_\z(\frac{1}{2})\Big)\geq \frac{1}{C}.$$
Note that the event on the left-hand side implies the event $\{\widetilde{M}_{t_x}\geq x-1/2\}$. 
The lower bound $1/C$ can be improved to a lower bound $1-\ee$ with the same bootstrapping argument as in Remark \ref{rem:bootstrap}: first evolve the BRW for a large time $T$ to obtain $N$ particles near the origin, then evolve the $N$ particles independently for time $t_x$. As a consequence, for any $\ee>0$, there is $T>0$ such that 
$$\p\Big(\exists v\in V_{t_x+T},\,\bet_{v,t_x+T}(t_x+T)\in [x-\frac{1}{2},x+\frac{1}{2}]\times B_\z(\frac{1}{2})\Big)\geq 1-\ee,$$
and hence there exists 
$K>0$ such that $\p(\widetilde{M}_n\geq \widetilde{m}_n-K)\geq 1-\ee$.

To see the upper bound of $\widetilde{M}_n$, 
note that the event $\widetilde{M}_n-\widetilde{m}_n\in[K-1/2,K+1/2]$ implies that the event $\tau^{(2)}_{\widetilde{m}_n+K}\leq n$ holds, where $\tau^{(2)}_x$ refers to the FPT to the ball centered at $\bx$ with radius two (to which the asymptotics \eqref{eq:taux asymp} still holds). Summing over $K$ large enough and applying Theorem \ref{thm:main} and Lemma \ref{lemma:concentration} (which also holds with $\tau_x$ replaced by $\tau^{(2)}_x$ up to changing the constants) lead to the upper bound.

\subsection{Targets with varying shapes}\label{sec:targets}

We expect that the same techniques in proving Theorems \ref{thm:main} and \ref{thm:main2} also apply to asymptotics of FPT to targets of distance $x$ from the origin that also vary in shape as $x$ grows. There are numerous such instances, while in the following we provide three canonical examples that might inspire future research. For simplicity, we focus on the spherically-symmetric case. Recall that $\bx=(x,0,\dots,0)\in\R^d$.\\

\textbf{Multiple target locations.} In this setting, we consider a constant $L>0$ and an integer-valued function $c(x)\to\infty$ and assume that the target is given by the union of $c(x)$ disjoint balls of radius one, whose centers lie on $\{x\}\times \{\bz\in\R^{d-1}:\n{\bz}\leq L\sqrt{x}\}$.

\begin{conjecture}\label{c1}
   Under suitable conditions on the growth rate of $c(x)\to\infty$, the FPT to the union of $c(x)$ disjoint balls of radius one centered on $\{x\}\times \{\bz\in\R^{d-1}:\n{\bz}\leq L\sqrt{x}\}$ is given by
   $$\tau^{\mathrm{M}}_{x,c(x)}=\frac{x}{{c}_1}+\frac{d+2}{2c_1\bl}\,\log x-\frac{1}{c_1\bl}\,\log c(x)+O_\p(1).$$
\end{conjecture}

For a sanity check, suppose that $c(x)\asymp x^{d-1}$, then $\tau_{x,c(x)}^{\mathrm{M}}$ is close to the FPT in dimension one. This aligns with Conjecture \ref{c1} in view of \eqref{eq:taux asymp}.\\

\textbf{Expanding balls of radius $R(x)\to\infty$.}
Consider balls centered at $\bx$ with radius $R(x)\to\infty$ as our targets. Slice up the ball using the collection of sets $\{A_\ell\}_{-R(x)\leq\ell<R(x)}$, where each $A_\ell$ lies in the slice $[x+\ell,x+\ell+1)\times\R^{d-1}$. To reach the ball centered at $\bx$ with radius $R(x)$, particles must reach some of the sets $A_\ell$, and hence it is reasonable to compare the first passage contributions to the sets $A_\ell$, possibly by applying Conjecture \ref{c1}. We conjecture that the first passage contribution is dominated by sets $\{A_\ell\}_{-R(x)\leq\ell\leq C-R(x)}$, where $C$ is a constant that does not depend on $x$. This leads to the following conjecture.

\begin{conjecture}\label{c2}
    Under suitable conditions on the growth rate of $R(x)\to\infty$, the FPT to the expanding balls of radii $R(x)$ centered at $\bx$ is given by
   $$\tau_{x,R(x)}^{\mathrm{E}}=\frac{x-R(x)}{{c}_1}+\frac{d+2}{2c_1\bl}\,\log x+O_\p(1).$$
\end{conjecture}

\textbf{Shrinking balls of radius $r(x)\to 0$.}
Analogously, one may consider FPT to balls of shrinking radii $r(x)\to 0$ centered at $\bx$. 

\begin{conjecture}\label{c3}
   Under appropriate assumptions on the density of $\bxi$ and the decay rate of $r(x)\to 0$, the FPT to the  balls of shrinking radii $r(x)$ centered at $\bx$ is given by
   $$\tau_{x,r(x)}^{\mathrm{S}}=\frac{x}{{c}_1}+\frac{d+2}{2c_1\bl}\,\log x+\frac{d}{c_1\bl}\,|\log r(x)|+O_\p(1).$$
\end{conjecture}
If one adopts the same approach to Theorem \ref{thm:main}, one needs to upgrade the ballot theorems to take care of the target of shrinking size, 
which seems beyond the reach of existing techniques of conditioned local limit theorems for random walks.

\begin{appendix}
\section{Index of frequently used notation}\label{appendix:notation}

\begin{tabularx}{\textwidth}{@{}lX@{}}

\toprule

 \multicolumn{2}{c}{\underline{Deterministic quantities}}\\
 
 $d$ & Underlying dimension of the BRW, $d\geq 1$\\
 $\{p_j\}_{j\geq 0}$ & Reproduction law of the BRW\\
$\bxi$ & Jump distribution of the BRW\\
$I(x)$ & Large deviation rate function for the first coordinate $\xi$ of $\bxi$\\
$\widehat{I}(\bx)$&Large deviation rate function for $\bxi$\\
$\rho$ & Expected number of descendants at time one, $\rho=\sum_{i}i\,p_i$\\
$c_1$ & Defined through $I(c_1)=\log\rho$\\
$c_2$ & $I'(c_1)$ \\
$\widehat{c}_1$ & Defined through $\widehat{I}(\widehat{c}_1,\z)=\log\rho$\\
$\bc_2$ & $\nabla \widehat{I}(\widehat{c}_1,\z)$ \\
$m_n$  & (One-dimensional) maximum asymptotic $c_1n-\frac{3}{2c_2}\log n$\\
$t_x$  & First passage time asymptotic $\frac{x}{c_1}+\frac{d+2}{2\bl c_1}\,\log x$ in dimension $d$\\
$t_{x,K}$  & $t_x-K$ \\
$\wx$  & $x-m_{(\log x)^2}$ \\
$\wtx$  & $t_x-(\log x)^2$ (with the exception of Appendix \ref{sec:cond local CLT proof})\\
$k_h$ & $(\log x)^2-h$\\
$\varphi_{n,\delta}(i)$&$e^{-\delta|i|\min(\frac{|i|}{n},1)}$\\
$\Phi_{n,\delta}(x,y)$ & Defined by \eqref{eq:Phi def}\\
$\Psi_{(\log x)^2,\delta}(\ell)$&Defined by \eqref{eq:Psi}\\

\multicolumn{2}{c}{}\\

\multicolumn{2}{c}{\underline{Events}}\\
$\sB_h$ & $h=\max\{\widetilde{h}:\exists\, v_1\in V_{t_x-\widetilde{h}},v_2,v_3\in V_{t_{x,K}}, v_2,v_3\succ v_1\succ v, \eta_{v_i,t_{x,K}}(t_{x,K})\geq x, i=2,3 \}$ \\
$v_{\mathrm{lca}}$ & The particle $v_1\in V_{t_x-h}$ realizing the maximum above (latest common ancestor)\\
$\sC_g$ & $\{\eta_{v_{\mathrm{lca}},t_x-h}(t_x-h)\in[x+g-m_{h-K},x+g-m_{h-K}+1]\}$\\
$\sD_\bu$ & The event that the last $d-1$ coordinates of $\bet_{v_{\mathrm{lca}},t_x-h}(t_x-h)$ belongs to $R_\bu$\\
$W_{\ell,h,g}$&$\{w\in V_{t_x-h}:\,w\succ v,\,\n{\widehat{\bet}_{w,t_x-h}(t_x-h)}\ll h,\,\sB_h\cap\sC_g\}$\\
$\sE_{1,\ell}$&See \eqref{eq:ballot1}\\
$\sE_{2,\ell}$& See \eqref{eq:E2}\\
$\sE_{3,\ell}$&$\big(\bigcup_{g\geq -\ell/K_4}\sC_g\big)\cap\big(\bigcup_{(\log x)^2-K_9\ell\leq h\leq (\log x)^2}\sB_h\big),\,\ell> K_6\log\log x$\\
$\sE_{4,\ell}$&See \eqref{eq:E4}\\
$\sE_{\ell}^*$ & $\sE_{1,\ell}\cup\sE_{2,\ell}$\\
$\sE_{\ell}$ & $\sE_{1,\ell}\cup\sE_{3,\ell}\cup\sE_{4,\ell}$\\
$\sF_{h,g}$&See \eqref{eq:Fhg}\\
$\sG_{n,\beta}$ & $\bigcup_{v\in V_n}\bigcup_{0\leq k\leq n}\left\{\eta_{v,n}(k)\geq \frac{km_n}{n}+\beta+\frac{6}{\bl}(\log\min\{k,n-k\})_+\right\}$ \\
$\sH_\bu$ &  $\{\bet_{v_{\mathrm{lca}},t_x-h}(t_x-h)-\bet_{v_{\mathrm{lca}},t_x-h}(\wtx)\in\R\times R_\bu$\}\\
$\sI_{n,g}$&$\big\{\exists\, v,w\in V_n,\,\mathrm{lca}(v,w)=\emptyset,\,\eta_{v,n}(n)\geq m_n-g,\,\eta_{w,n}(n)\geq m_n-g\big\}$\\
$\sJ_{h,g}$& $\{\exists\, w\in V_{t_x-h},\,w\succ v,\,\eta_{w,t_x-h}(t_x-h)\in[x-m_{h-K}+g-1,x-m_{h-K}+g)\}$\\
$\sK_{h}$ & See \eqref{eq:Kh}\\
$\sK_{h}^*$ &See \eqref{eq:Kh*}\\
$S$ & All-time survival event of the BRW\\

\multicolumn{2}{c}{}\\

\multicolumn{2}{c}{\underline{Other definitions}}\\
$\#A$ & Cardinality of a finite set $A$\\
$w\succ v$ & Particle $w$ is a descendant of $v$\\
$\emptyset$ & The unique particle at time zero\\
$B_x$ & Unit ball centered at $\bx=(x,0,\dots,0)\in\R^d$\\
$B_\bz$ & Unit ball centered at $\bz\in\R^d$\\
$\H_x$ & $[x,\infty)\times\R^{d-1}$\\
$\tau_x$  & First passage time of $d$-dimensional BRW to $B_x$ \\
$M_n$  & Maximum of one-dimensional BRW at time $n$\\
$P_n$ & Production number, defined as $\#\{v\in V_n:\exists\, w\in V_{t_x},\,w\succ v,\,\eta_{w,t_x}(t_x)\geq x\}$\\
$R_\bu$ & The rectangle $[u_1,u_1+1)\times\dots\times[u_{d-1},u_{d-1}+1)$ for $\bu=(u_1,\dots,u_{d-1})$\\
$V_n$ & The collection of particles at time step $n$\\
$\bet_{v,n}(k)$ & Location of the $d$-dimensional random walk that leads to $v\in V_n$ evaluated at time $k$\\
$\eta_{v,n}(k)$ & The first coordinate of $\bet_{v,n}(k)$\\
$\widehat{\bet}_{v,n}(k)$ & The last $d-1$ coordinates of $\bet_{v,n}(k)$\\
$\q^{\ell,v}$ & The probability measure on the BRW restricted to descendants of $v$,\\
&\hspace{1cm} conditioned on $\eta_{v,\widetilde{t}_x}(\widetilde{t}_x)\in[\wx-\ell-1,\wx-\ell)$ and $\sG_{\wtx,K_2}^{\mathrm{c}}$\\

\bottomrule
\label{tab:TableOfNotationForMyResearch}
\end{tabularx}

\section{Escape probability of BRW}
\label{sec:escape}
The goal of this appendix is to establish the following result, which was used in Section \ref{sec:412}.
\begin{lemma}\label{lemma:escape}
  Assume (A1)--(A4).  There exists $K_1>0$ such that
    $$\p(\n{\bet_{v,n}(n)}\geq 1\text{ for all }v\in V_n)\leq K_1e^{-\sqrt{n}/K_1}.$$
\end{lemma}

\begin{remark}
    The closest result in this direction is perhaps \citep{zhang2023lower}, which studied convergence rates of
    $$\p(\#\{v\in V_n:\eta_{v,n}(n)>\theta c_1 n\}>e^{an})$$
    for $a\in[0,\log\rho-I(\theta c_1))$. For branching Brownian motion, \citep{Oz2020branching,oz2023} studied large deviation probabilities of the number of particles in a linearly moving ball. In particular, the analogue of Lemma \ref{lemma:escape} for BBM was established as a special case of Theorem 2.1 of \citep{oz2023}. Our result is quite crude (for instance, we believe that the escape probability can be improved to $O(e^{-n/L})$ based on analogues in \citep{Oz2020branching}), but it suffices for our purpose.  Additionally, the arguments required are relatively simple compared to the literature above.
\end{remark}

\begin{proof}[Proof of Lemma \ref{lemma:escape}]
We provide the proof in the one-dimensional setting and the multi-dimensional case follows similarly. The strategy is to evolve particles independently in the periods $[0,\sqrt{n}]$ and $[\sqrt{n},n]$. We show that at time $\sqrt{n}$, with high probability there are $e^{\delta\sqrt{n}}$ particles present and located in $O(\sqrt{n})$, and with high probability, a certain portion of the particles located in $O(\sqrt{n})$ at time $\sqrt{n}$ will have a descendant in $[-1,1]$ at time $n$.

To carry out the above plan, let $\delta_1>0$ be a small constant and we define the events 
$$E_1:=\{\#V_{\sqrt{n}}\geq e^{\delta_1\sqrt{n}}\}\quad \text{and}\quad E_2:=\Big\{\forall v\in V_{\sqrt{n}},\,|\eta_{v,\sqrt{n}}(\sqrt{n})|<\frac{\sqrt{n}}{\delta_1}\Big\}.$$
It follows from the main result of \citep{zhang2023lower} and Theorem 3.2 of \citep{gantert2018large} that $\p(E_1\cap E_2)>1-O(e^{-\delta_2\sqrt{n}})$ for some $\delta_2>0$. On the event $E_1\cap E_2$, we may identify particles $v_j,~1\leq j\leq e^{\delta_1\sqrt{n}}$, where $\eta_{v_j,\sqrt{n}}(\sqrt{n})\in[-\sqrt{n}/\delta_1,\sqrt{n}/\delta_1]$. Let $S_j$ denote the survival event of the particle $v_j$. It follows from local CLT applied to $\xi$ (see e.g.~Lemma 23 of \citep{blanchet2024first}) that for some $\delta_3>0$,
$$\p\bigg(\exists\, w\in V_{n},\,w\succ v_j,\,|\eta_{w,n}(n)|\leq 1\mid S_j\bigg)\geq \frac{\delta_3}{\sqrt{n}},$$
and hence using independence, the event
$$E_3:=\Big\{\#\{w\in V_{n}:|\eta_{w,n}(n)|\leq 1\}>0\Big\}$$
satisfies $$\p(E_3\mid E_1\cap E_2)\geq 1-(1-\frac{\delta_3}{\sqrt{n}})^{e^{\delta_1\sqrt{n}}}>1-O(e^{-n}).$$
We thus conclude that
\begin{align*}
    \p(\n{\bet_{v,n}(n)}\geq 1\text{ for all }v\in V_n)&\leq \p(E_1^{\mathrm{c}}\cup E_2^{\mathrm{c}})+\p(E_3^{\mathrm{c}}\mid E_1\cap E_2)\leq O(e^{-\delta_2\sqrt{n}}).
\end{align*}
This proves Lemma \ref{lemma:escape}.
\end{proof}

\section{Some upper bounds of (conditional) ballot probabilities}
\label{sec:upper ballot}
\subsection{A multi-dimensional ballot upper bound}

The results in this appendix are essential for establishing Lemma \ref{lemma:uniform bd} through Lemma \ref{lemma:uniform local CLT}. 
Following the seminal work of \citep{denisov2015random} on random walks in cones, we prove a multi-dimensional ballot upper bound where the random walk reaches a target in $\R^d$ and the path projected onto the first dimension is constrained by a linear barrier tilted by a logarithmic term. The connection to random walks in cones (i.e.~collections of rays from $\z\in\R^d$ going through a certain open subset of the sphere $\S^{d-1}$) is realized by letting the cone be the half-space $\{\bx=(x_1,\dots,x_d)\in\R^d:\,x_1>0\}$. The following statements are self-contained, but we refer to Section 2.4.1 of \citep{blanchet2024first} for a brief introduction to random walks in cones.

In the following, let $\{\bxi_i\}_{i\geq 1}$ be an i.i.d.~sequence in $\R^d$ satisfying (A2)--(A4) and $\bS_n=\bxi_1+\dots+\bxi_n$ be its partial sum. Denote by $\xi_i$ and $S_n$ their first coordinates. Let
$$\overline{\psi}(k)=-y-L\log\min\{k,n-k\}_+,~0\leq k\leq n,$$
where $L>0$ is a fixed constant. Recall that $R_\bu=[u_1,u_1+1)\times\dots\times[u_{d-1},u_{d-1}+1)$ for $\bu=(u_1,\dots,u_{d-1})\in\R^{d-1}$. The following result improves upon Lemma 16 of \citep{blanchet2024first}. 
\begin{lemma}\label{thm:d-dim ballot with log}
    Consider $n\geq 1$, $1\leq y\ll \log n$, and $y+a>0$. Then 
    $$\p(\bS_n\in [a,a+1]\times R_\bu,\,S_k\geq \overline{\psi}(k)\text{ for all }0<k<n)\ll y(y+a)n^{-\frac{d+2}{2}}.$$
\end{lemma}

\begin{lemma}\label{lemma:ballot linear}
   Consider $n\geq 1$, $1\leq y\ll \log n$, and  $y+a>0$. It holds 
    $$\p(\bS_n\in [a,a+1]\times R_\bu,\,S_k\geq -y\text{ for all }0<k<n)\ll  y(y+a)n^{-\frac{d+2}{2}}.$$
\end{lemma}

\begin{proof}
   This follows from the derivation of Lemmas 27 and 28 in \citep{denisov2015random}, along with their Theorem 1.  Note that $p=1$ therein since we take the cone to be the half-space $\{\bx\in\R^d:\,x_1> 0\}$.  The results in \citep{denisov2015random} were stated in the lattice case but the derivation of Lemmas 27 and 28 depends only on the non-singularity of the jumps (which is in turn a consequence of (A5)) but not the lattice property.
\end{proof}

The proof of Lemma \ref{thm:d-dim ballot with log} follows a similar route as the proof of (11) in \citep{bramson2016convergence}, while in the proof we apply now Lemma \ref{lemma:ballot linear} instead of (6) of Lemma 2.1 therein.

\begin{proof}[Proof of Lemma \ref{thm:d-dim ballot with log}]
Let $\tau$ be the first time in $[0,n]$ at which $S_k$ takes its minimum, and suppose that $a\geq -1/2$.  We split into cases depending on the value of $\tau$ and apply a union bound. By symmetry of the function $\overline{\psi}$, we may assume $\tau<n/2$. Define $I_y:=\bZ\cap(y^7,n/2]$ and
\begin{align*}
    \Omega_{n,y}&=\big\{\bS_n\in [a,a+1]\times R_\bu,\,S_k\geq -y\text{ for all }0<k<n\big\},\\
    \Omega'_{n,y}&=\big\{\bS_n\in [a,a+1]\times R_\bu,\,S_k\geq \overline{\psi}(k)\text{ for all }0<k<n\big\}.
\end{align*}
It follows from Lemma \ref{lemma:ballot linear} that (using $a\geq -1/2$)
\begin{align*}
    \p(\tau\in I_y;{\Omega'_{n,y}})&\ll \sum_{k\in I_y}\sum_{j=0}^{y+L\log k}\frac{(j+1)(j+a)}{k^{3/2}(n-k)^{\frac{d+2}{2}}}\\
    &\ll \sum_{k\in I_y}\frac{(y+L\log k+1)^2(y+L\log k+a)}{k^{3/2}(n-k)^{\frac{d+2}{2}}}\ll (y+a)y^{-3/2}n^{-\frac{d+2}{2}}.
\end{align*}

Next, we consider $k\leq y^{19/10}$. 
On the event $\Omega'_{n,y}\setminus \Omega_{n,y}$, we have $S_k\in[-y-L\log k,-y]$. Therefore, by independence before and after time $k$ and Lemma \ref{lemma:ballot linear},
\begin{align*}
    &\hspace{0.5cm}\p(\tau=k;{\Omega'_{n,y}\setminus \Omega_{n,y}})\\
    &\leq \p(S_k\in[-y-L\log k,-y])\\
    &\hspace{0.3cm}\times\max_{x\in[-y-L\log k,-y]}\sup_{\substack{\bu'\in\R^{d-1}}}\p(S_j\geq 0,\,\forall 1\leq j\leq n-k;\,\bS_{n-k}\in [a+x,a+x+1]\times R_{\bu'})\\
    &\ll y^{-2}\times (a+y+\log k)n^{-\frac{d+2}{2}}\ll  y^{-2}(y+a)n^{-\frac{d+2}{2}},
\end{align*}
where the upper bound for $\p(S_k\in[-y-L\log k,-y])$ follows from the same reasoning as below (91) of \citep{bramson2016convergence}.

For $k\in(y^{19/10},y^7)$, we again apply independence of the random walk before and after time $k$. We have in this case the boundary of $\Omega'_{n,y}$ is at most $\ll y^{10}$ below that for $\Omega_{n,y}$ uniformly in $k\in(y^{19/10},y^7)$. Therefore, applying Lemma \ref{lemma:ballot linear} twice yields
\begin{align*}
    \p(\tau=k;{\Omega'_{n,y}\setminus \Omega_{n,y}})&\ll  y^{1/10}k^{-3/2}\times (y+\log k)(y+\log k+a)n^{-\frac{d+2}{2}}\\
    &\ll y^{11/10}(y+a)k^{-3/2}n^{-\frac{d+2}{2}}.
\end{align*}


Combining the above estimates yields that for all $\bu\in\R^{d-1}$,
\begin{align*}
    \p(\Omega'_{n,y}\setminus \Omega_{n,y})&\ll (y+a)y^{-3/2}n^{-\frac{d+2}{2}}+\sum_{k=1}^{y^{19/10}}y^{-2}(y+a)n^{-\frac{d+2}{2}}\\
    &\hspace{5cm}+\sum_{k=y^{19/10}}^{y^7}y^{11/10}(y+a)k^{-3/2}n^{-\frac{d+2}{2}}\\
    &\ll  y(y+a)n^{-\frac{d+2}{2}}
\end{align*}
    for some $\delta>0$. Also note that Lemma \ref{lemma:ballot linear} yields $\p(\Omega_{n,y})\ll y(y+a)n^{-\frac{d+2}{2}}$. This finishes the proof for $a\geq -1/2$. The case $a<-1/2$ follows by reversing the random walk, using precisely the same argument at the end of the proof of (11) in \citep{bramson2016convergence}.
\end{proof}

\subsection{A ballot upper bound involving moderate deviation}\label{sec:ballot ub}

In this appendix, we revisit the recent work \citep{grama2021conditioned} and extract one-dimensional ballot upper bounds involving moderate deviation. Consider a one-dimensional non-lattice random walk $\{S_n\}_{n\geq 1}$ with centered i.i.d.~jumps and finite moments of any order. 
The quantity of interest is 
\begin{align}
    \p(x+S_n\in[y,y+1],\,x+S_k\geq  0\text{ for all }1\leq k\leq n),\label{eq:b2}
\end{align}
where $|x-y|\gg\sqrt{n}$. Applying the classical ballot theorem leads only to an upper bound of $\ll xyn^{-3/2}$, which does not account for the fact that the (conditioned) random walk is unlikely to travel a distance of $|x-y|$ in time $n$. 
The Brownian motion analogue of \eqref{eq:b2} was analyzed in Lemma 18 of \citep{denisov2015random}, while we are unaware of general tight asymptotics for the random walk case. 
A notable exception is Theorem 1.2 of \citep{grama2021conditioned}, which provided precise asymptotics of the ballot probability \eqref{eq:b2} for $x\in[0,n^{1/2-\ee}]$, $y\in [C_1\sqrt{n},C_2\sqrt{n\log n}]$ for fixed constants $C_1,C_2>0$. On the contrary, we satisfy ourselves with asymptotic upper bounds, which allow for a wider range of the parameter $y$. We first record below the result from \citep{grama2021conditioned} that we will employ.

\begin{lemma}\label{lemma:ballot+deviation 01}
 Fix $\ee\in(0,1/2)$ and $L>0$.   There exists $\delta>0$ such that uniformly for $x\in[0,n^{1/2-\ee}]$ and $y\geq \sqrt{n}/L$,
 $$\p(x+S_n\in[y,y+1],\,x+S_k\geq 0\text{ for all }1\leq k\leq n)\ll xn^{-1-\ee}+xyn^{-3/2}e^{-\frac{\delta y^2}{n}}.$$
\end{lemma}

\begin{proof}
    By (3.1) of Theorem 3.1 in \citep{grama2021conditioned} (and following the beginning of the proof of Theorem 1.2 therein), it holds uniformly for $x\in[0,n^{1/2-\ee}]$ that
    \begin{align}
        \p(x+S_n\in[y,y+1],\,x+S_k\geq 0\text{ for all }1\leq k\leq n)\ll \frac{x}{n}J_n,\label{eq:b1}
    \end{align}
     where 
    $$J_n:=\int_{-\ee}^{1+\ee}\bigg(\Big(1+\frac{t+y}{\sqrt{n}}\Big)e^{-\frac{\delta(t+y)^2}{n}}+n^{-\ee}\bigg)\,\d t$$
    for some large constant $C>0$ depending on the law of the jump and $\ee$. We have uniformly for $y\geq \sqrt{n}/L$, 
    $$J_n\ll n^{-\ee}+\frac{y}{\sqrt{n}}e^{-\frac{\delta y^2}{n}}.$$
Combined with \eqref{eq:b1} finishes the proof.    
\end{proof}

Let us now consider a logarithmically tilted barrier. 
Applying the same arguments that derived Lemma \ref{thm:d-dim ballot with log} from Lemma \ref{lemma:ballot linear}, and using Lemma \ref{lemma:ballot+deviation 01} instead of Lemma \ref{lemma:ballot linear}, we arrive at the following result.

\begin{lemma}
    \label{lemma:ballot+deviation 1}
    Fix $\ee\in(0,1/2)$ and $L,L'>0$.   There exists $\delta>0$ such that uniformly for $x\in[0,n^{1/2-\ee}]$ and $y\geq \sqrt{n}/L$,
 \begin{align*}
     &\p(x+S_n\in[y,y+1],\,x+S_k\geq  -L'(\log\min\{k,n-k\})_+\text{ for all }1\leq k\leq n)\\
     &\hspace{4cm}\ll xn^{-1-\ee}+xyn^{-3/2}e^{-\frac{\delta y^2}{n}}.
 \end{align*}
\end{lemma}

Due to the term $xn^{-1-\ee}$, the bound in Lemma \ref{lemma:ballot+deviation 1} cannot be tight for $y\gg \sqrt{n\log n}$.
While a general tight bound seems reminiscent in the literature, the following weaker estimate suffices for our purpose.

\begin{lemma}\label{lemma:ballot+deviation 2}
 Fix $L'>0$.   There exists $\delta>0$ such that uniformly for $x,y>0$,
$$\p(x+S_n\in[y,y+1],\,x+S_k\geq  -L'(\log\min\{k,n-k\})_+\text{ for all }1\leq k\leq n)\ll \varphi_{n,\delta}(|x-y|).$$
In particular, for $x\in[0,n^{1/2-\ee}]$ with $\ee\in(0,1/2)$ and $y\gg \sqrt{n}$,
    \begin{align*}
        &\p(x+S_n\in[y,y+1],\,x+S_k\geq  -L'(\log\min\{k,n-k\})_+\text{ for all }1\leq k\leq n)\\
        &\hspace{3cm}\leq\p(|S_n|\geq |x-y|)\ll \varphi_{n,\delta}(y).
    \end{align*}
\end{lemma}

\begin{proof}
   By a moderate deviation estimate (e.g., Theorem 3.7.1 of \citep{dembo2009large}), we have
   \begin{align*}
       &\p(x+S_n\in[y,y+1],\,x+S_k\geq  -L'(\log\min\{k,n-k\})_+\text{ for all }1\leq k\leq n)\\
       &\hspace{3cm}\leq\p(|S_n|\geq |x-y|)\ll \varphi_{n,\delta}(|x-y|)
   \end{align*}
   for some $\delta>0$.
\end{proof}

\begin{remark}\label{rem}
    By symmetry, the same statements of Lemmas \ref{lemma:ballot+deviation 1} and \ref{lemma:ballot+deviation 2} hold with the roles of $x,y$ interchanged. Moreover, the same results hold for $0\leq x\ll n^{1/2-\ee}$ by a suitable scaling.
\end{remark}

\subsection{BRW conditioned on two descendants with large displacements separated at the first step}\label{sec:two BRw}

In practice, when conditioning on the event $\sB_h$ and a fixed location at time $t_x-h$, we would like to understand the conditional law of the BRW in the period $[t_x-h,t_{x,K}]$, given the information that two trajectories separated at time $t_x-h$ both reach the level $x$ at time $t_{x,K}$. In this self-contained appendix, we consider a large number $n$ and $g<0$, and recall from \eqref{eq:Ing} that
\begin{align*}
    \sI_{n,g}=\Big\{\exists\, v,w\in V_n,\,\mathrm{lca}(v,w)=\emptyset,\,\eta_{v,n}(n)\geq m_n-g,\,\eta_{w,n}(n)\geq m_n-g\Big\},~n\in\N,\,g<0.
\end{align*}
The general strategy to deal with an event of this type is to condition on the first generation and partition the event into disjoint events where at least two of them have descendants with large maxima. When conditioned on the first generation, the events of having large maxima will then be independent, and hence can be decoupled. 

In the following, for $j\in\N$, let $\p_j$ denote the law of the tuple $(\eta_1,\dots,\eta_j)$ of i.i.d.~random variables with the same law as $\xi$. Throughout, assume (A1)--(A4).

\begin{lemma}\label{lemma:double prob}
  For $g<0$,  it holds that $\p(\sI_{n,g})\asymp (|g|+1)^2e^{2c_2g}$.
\end{lemma}
\begin{proof}
We condition on the first generation and obtain
\begin{align*}
    \p(\sI_{n,g})&= \sum_{j=2}^\infty p_j\int\p(\sI_{n,g}\mid V_1=\{v_k\}_{1\leq k\leq j},\,\eta_{v_k,1}(1)=\eta_k)\,  \d\p_j(\eta_1,\dots,\eta_j). 
\end{align*}
We first prove the upper bound of $\p(\sI_{n,g})$. 
By assumption (A4),
\begin{align}
    \p(\xi>x)\ll e^{-(c_2+\delta)x}.\label{eq:exp decay}
\end{align}
For a fixed $j$, there are $\leq j^2$ possibilities of pairs $(k,k'),~1\leq k<k'\leq j$ so that descendants of $v_k,\,v_{k'}$ realize the event $\sI_{n,g}$. By a union bound, we have
\begin{align*}
    &\hspace{0.5cm}\int\p(\sI_{n,g}\mid V_1=\{v_k\}_{1\leq k\leq j},\,\eta_{v_k,1}(1)=\eta_k)\,  \d\p_j(\eta_1,\dots,\eta_j)\\
    &\leq j^2 \int \p(M_{n-1}>m_n-g-\eta_1)\p(M_{n-1}>m_n-g-\eta_2) \,\d \p_2(\eta_1,\eta_2)\\
    &\ll j^2\E\big[\min\{(|g+\xi|+1)e^{c_2(g+\xi)},1\}\big]^2\\
        &\ll j^2(|g|+1)^2e^{2c_2g}.
\end{align*}
where we have used \eqref{eq:exp decay} and Lemma \ref{lemma:brw small deviation prob}. By assumption (A1), we conclude that
$$\p(\sI_{n,g})\ll \sum_{j=2}^\infty p_jj^2(|g|+1)^2e^{2c_2g}\ll (|g|+1)^2e^{2c_2g},$$
as desired.

To show the lower bound of $\p(\sI_{n,g})$, suppose that $p_j>0$ for some $j\geq 2$ (valid since $\rho>1$). Since the probability that $j$ i.i.d.~samples of $\eta_1$ are all positive is $\gg 1$, we have by \eqref{eq:BRW max tail2} of Lemma \ref{lemma:brw small deviation prob},
$$\int\p(\sI_{n,g}\mid V_1=\{v_k\}_{1\leq k\leq j},\,\eta_{v_k,1}(1)=\eta_k)\,  \d\p_j(\eta_1,\dots,\eta_j)\gg (|g|+1)^2e^{2c_2g},$$
as desired.
\end{proof}

Recall the defining barrier \eqref{eq:hat psi} for the ballot event $\sE_{1,\ell}$. 
Let $V_{n,g}$ denote the set of particles $v\in V_n$ whose past trajectory $\{\eta_{v,n}(k)\}_{1\leq k\leq n}$ does not cross the barrier 
$$\widetilde{\psi}(k):= L\log n-g+\frac{k}{n}m_n+\frac{4}{c_2}(\log\min\{k,n-k\})_+,~1\leq k\leq n.$$
Recall that $\varphi_{n,\delta}(i):=e^{-\delta|i|\min(\frac{|i|}{n},1)}$.

\begin{lemma}\label{lemma:2prob2case}
    It holds that for $g<0$,
    \begin{align}
        \p(\exists v\in V_{n,g},\,\bet_{v,n}(n)\in B_{(m_{n}-g,\bu)}\mid \sI_{n,g})\ll \begin{cases}
            n^{3/2}\varphi_{n,\delta}(\n{\bu})&\text{ if }\n{\bu}>\sqrt{n}\log n;\\
            (\log n)^2 n^{-\frac{d-1}{2}}&\text{ if }\n{\bu}\leq \sqrt{n}\log n.
        \end{cases}\label{eq:2}
    \end{align}
\end{lemma}


\begin{proof}
We write 
$$\p(\exists v\in V_{n,g},\,\bet_{v,n}(n)\in B_{(m_{n}-g,\bu)}\mid \sI_{n,g})=\frac{\p(\exists v\in V_{n,g},\,\bet_{v,n}(n)\in B_{(m_{n}-g,\bu)},\,\sI_{n,g})}{\p(\sI_{n,g})}.$$
The denominator is bounded from above by $\p(\sI_{n,g})\gg (|g|+1)^2e^{2c_2g}$ by Lemma \ref{lemma:double prob}, and hence it suffices to establish an upper bound for the numerator. To this end, we condition on the law of the first generation, by first conditioning on $\#V_1$ and then the locations of particles belonging to the first generation, we have
\begin{align*}
    &\hspace{0.5cm}\p(\exists v\in V_{n,g},\,\bet_{v,n}(n)\in B_{(m_{n}-g,\bu)},\,\sI_{n,g})\\
    &=\sum_{j=2}^\infty p_j\int \p(\exists v\in V_{n,g},\,\bet_{v,n}(n)\in B_{(m_{n}-g,\bu)},\,\sI_{n,g}\\
        &\hspace{3cm}\mid V_1=\{v_k\}_{1\leq k\leq j},\,\bet_{v_k,1}(1)=\bet_k)\,  \d\p_j(\bet_1,\dots,\bet_j),
\end{align*}
where $p_j$ is the law of a tuple of $j$ i.i.d.~random variables with the same law as $\bxi$. 
To deal with the inner probability, we note that on the event $\{\exists v\in V_{n,g},\,\bet_{v,n}(n)\in B_{(m_{n}-g,\bu)},\,\sI_{n,g}\}$, there must exist $k\neq k'$ with $1\leq k,k'\leq j$ such that:
\begin{itemize}
    \item $\exists v\in V_{n,g},\,v\succ v_k,\,\bet_{v,n}(n)\in B_{(m_{n}-g,\bu)}$;
    \item $\exists w\in V_n,\,w\succ v_{k'},\,\eta_{w,n}(n)\geq m_n-g$.
\end{itemize}
Applying a union bound over the set of feasible $k,k'$, we obtain
\begin{align*}
    &\hspace{0.5cm}\p(\exists v\in V_{n,g},\,\bet_{v,n}(n)\in B_{(m_{n}-g,\bu)},\,\sI_{n,g}\mid V_1=\{v_k\}_{1\leq k\leq j},\,\bet_{v_k,1}(1)=\bet_k)\\
    &\leq \sum_{1\leq k\neq k'\leq j} \p(\exists v\in V_{n,g},\,v\succ v_k,\,\bet_{v,n}(n)\in B_{(m_{n}-g,\bu)};\,\exists w\in V_n,\,w\succ v_{k'},\\
        &\hspace{5cm}\eta_{w,n}(n)\geq m_n-g\mid V_1=\{v_k\}_{1\leq k\leq j},\,\bet_{v_k,1}(1)=\bet_k).
\end{align*}
Applying the independence and Lemma \ref{lemma:brw small deviation prob}, we have for each $k\neq k'$,
\begin{align*}
    &\hspace{0.5cm} \p(\exists v\in V_{n,g},\,v\succ v_k,\,\bet_{v,n}(n)\in B_{(m_{n}-g,\bu)};\,\exists w\in V_n,\,w\succ v_{k'},\\
        &\hspace{5cm}\eta_{w,n}(n)\geq m_n-g\mid V_1=\{v_k\}_{1\leq k\leq j},\,\bet_{v_k,1}(1)=\bet_k)\\
    &\leq \p(M_{n-1}>m_n-g-\eta_{k'})\p(\exists v\in V_{n-1,g},\,\bet_{v,n-1}(n-1)\in B_{(m_{n}-g,\bu)-\bet_k})\\
    &\ll \min\{(|g+\eta_{k'}|+1)e^{c_2(g+\eta_{k'})},1\}\p(\exists v\in V_{n-1,g},\,\bet_{v,n-1}(n-1)\in B_{(m_{n}-g,\bu)-\bet_k}).
    \end{align*} 
    Suppose that $\n{\bu}\leq \sqrt{n}\log n$. By a change of measure argument and Lemma \ref{thm:d-dim ballot with log}, 
\begin{align}
    \p(\exists v\in V_{n-1,g},\,\bet_{v,n-1}(n-1)\in B_{(m_{n}-g,\bu)-\bet_k})\ll (\log n)(\log n-g)e^{c_2(g+\eta_k)}n^{-\frac{d-1}{2}}.\label{eq:small bu}
\end{align} 
We therefore have that for $\n{\bu}\leq \sqrt{n}\log n$,
\begin{align*}
    &\hspace{0.5cm} \p(\exists v\in V_{n,g},\,v\succ v_k,\,\bet_{v,n}(n)\in B_{(m_{n}-g,\bu)};\,\exists w\in V_n,\,w\succ v_{k'},\\
        &\hspace{5cm}\eta_{w,n}(n)\geq m_n-g\mid V_1=\{v_k\}_{1\leq k\leq j},\,\bet_{v_k,1}(1)=\bet_k)\\
    &\ll (\log n)^2(|g|+1)e^{c_2(g+\eta_k)} \min\{(|g+\eta_{k'}|+1)e^{c_2(g+\eta_{k'})},1\}n^{-\frac{d-1}{2}}.
\end{align*}
In conclusion, we have by \eqref{eq:exp decay} and assumption (A1) that for the case $\n{\bu}\leq \sqrt{n}\log n$,
\begin{align*}
    &\hspace{0.5cm}\p(\exists v\in V_{n,g},\,\bet_{v,n}(n)\in B_{(m_{n}-g,\bu)},\,\sI_{n,g})\\
    &\ll\sum_{j=2}^\infty p_j n^{-\frac{d-1}{2}}j^2(\log n)^2(|g|+1)e^{c_2g}\E\big[\min\{|g+\xi|e^{c_2(g+\xi)},1\}\big]\E[e^{c_2\xi}]\\
    &\ll n^{-\frac{d-1}{2}}(\log n)^2(|g|+1)^2e^{2c_2g}.
\end{align*}
This proves \eqref{eq:2} in the case $\n{\bu}\leq \sqrt{n}\log n$.

In the case $\n{\bu}>\sqrt{n}\log n$, we need to replace \eqref{eq:small bu} accordingly. Applying a change of measure as in the proof of Lemma \ref{lemma:uniform local CLT},
\begin{align*}
    &\hspace{0.5cm}\p(\exists v\in V_{n-1,g},\,\bet_{v,n-1}(n-1)\in B_{(m_{n}-g,\bu)-\bet_k})\\
    &\ll n^{3/2}e^{c_2(g+\eta_k)}\p(\bS_{n-1}\in B_{(m_{n}-g,\bu)-\bet_k},\,S_k\leq \widetilde{\psi}(k)\text{ for all }1\leq k\leq n-1)\\
    &\ll n^{3/2}e^{c_2(g+\eta_k)}\p(\|{\widehat{\bS}_{n-1}-(\bu-\widehat{\bet}_k)\|}\leq 1)\\
    &\ll n^{3/2}e^{c_2(g+\eta_k)}\varphi_{n,\delta}(\n{\bu-\widehat{\bet}_k}).
\end{align*}
We apply the bound $\varphi_{n,\delta}(\n{\bu-\widehat{\bet}_k})\leq \varphi_{n,\delta}(\n{\bu}/2)$ on the event $\n{\widehat{\bet}_k}\leq \n{\bu}/2$, and apply the bound \eqref{eq:small bu} on the complement of this event which still occurs with probability $\ll \varphi_{n,\delta}(\n{\bu}/2)$ by \eqref{eq:exp decay}. More precisely, we use that for $k\neq k'$,
\begin{align*}
     &\hspace{0.5cm}\p(\exists v\in V_{n,g},\,v\succ v_k,\,\bet_{v,n}(n)\in B_{(m_{n}-g,\bu)};\,\exists w\in V_n,\,w\succ v_{k'},\\
        &\hspace{5cm}\eta_{w,n}(n)\geq m_n-g\mid V_1=\{v_k\}_{1\leq k\leq j},\,\bet_{v_k,1}(1)=\bet_k)\\
     &\ll \min\{(|g+\eta_{k'}|+1)e^{c_2(g+\eta_{k'})},1\}\\
     &\hspace{3cm}\times\begin{cases}
     n^{3/2}e^{c_2(g+\eta_k)}\varphi_{n,\delta}(\n{\bu}/2)&\text{ if }\n{\widehat{\bet}_k}\leq \n{\bu}/2;\\
         (\log n)(\log n-g)e^{c_2(g+\eta_k)}n^{-\frac{d-1}{2}}&\text{ if }\n{\widehat{\bet}_k}> \n{\bu}/2.
     \end{cases}
\end{align*}
This part of the sum is then controlled by
\begin{align*}
&\hspace{0.5cm}\sum_{j=2}^\infty p_j j^2\Big(n^{-\frac{d-1}{2}}(\log n)^2(|g|+1)e^{c_2g}\E\big[\min\{|g+\xi|e^{c_2(g+\xi)},1\}\big]\E[e^{c_2\xi}\bone_{\{\n{\widehat{\bet}_k}> \n{\bu}/2\}}]\\
&\hspace{3cm}+n^{3/2}e^{c_2(g+\eta_k)}\varphi_{n,\delta}(\frac{\n{\bu}}{2})\E\big[\min\{|g+\xi|e^{c_2(g+\xi)},1\}\big]\Big)\\
&\ll n^{3/2}\varphi_{n,\delta}(\n{\bu})(|g|+1)^2e^{2c_2g},
\end{align*}
where the $\delta$ may vary from line to line. 
The remaining of the argument follows analogously as the case $\n{\bu}\leq \sqrt{n}\log n$.    
\end{proof}

\begin{lemma}\label{lemma:double prob 2}
  For $g<0$, there exist $L,\delta>0$ such that
\begin{align*}
   \begin{split}
        &\hspace{0.5cm}\p(\exists u\in V_{n},\,{\bet}_{u,n}(n)\in B_{(m_n-g,-\bu)};\, \sI_{n,g})\\
    &\ll \begin{cases}
        (|g|+1)^2e^{2c_2g}e^{-\delta \|\bu\|}&\text{ if }\|\bu\|\geq \max\{10,Lh\};\\
        (|g|+1)^2e^{2c_2g}&\text{ otherwise}.
    \end{cases}
   \end{split}
\end{align*}
\end{lemma}

\begin{proof}
The proof follows from the same technique of conditioning on the first generation, as in Lemmas \ref{lemma:double prob} and \ref{lemma:2prob2case}. The case $\|\bu\|\leq \max\{10,Lh\}$ follows directly from Lemma \ref{lemma:double prob} using the trivial bound
$$\p(\exists u\in V_{n},\,{\bet}_{u,n}(n)\in B_{(m_n-g,-\bu)};\, \sI_{n,g})\leq \p(\sI_{n,g}).$$
For the case $\|\bu\|\geq \max\{10,Lh\}$, we follow the case of $\|\bu\|>\sqrt{n}\log n$ in Lemma \ref{lemma:2prob2case} (now without the barrier restriction), by similarly dividing into two cases depending on the sizes of $\|\bu\|/2$ and the displacement in the first generation. The factor $n^{3/2}$ can be absorbed into the factor $e^{-\delta \|\bu\|}$ as $\|\bu\|\geq Lh$.
\end{proof}

\section{A conditional local CLT (for general jumps)}\label{sec:cond local CLT proof}
 In this appendix, we prove a more general version of Lemma \ref{lemma:conditioned local CLT}, which deals with a general increment distribution $\bxi$ that may not be spherically-symmetric (see the discussion in Section \ref{sec:non-sphere}). The result is indeed an application of Petrov's theorem \citep{petrov1965probabilities} and a change of measure argument, while we include the details for completeness. 
Recall the setting below \eqref{eq:I long} of the law $\bxi$. If $\bxi$ is spherically-symmetric, we have $\widehat{c}_1=c_1$ and $\bc_2=c_2\be_1$.

Define the set
$$\widetilde{B}_{\bc_2}(x;r):=\Big\{\bx+s\bc_2+\by:\,\by\cdot\bc_2=0,\,\n{\by}\leq r,\,|s|\leq r\Big\}
,~r>0$$
and (for this appendix only)
$$\wtx:=\frac{x}{\widehat{c}_1}+\frac{d+2}{2\widehat{c}_1\partial_{x_1} \widehat{I}(\widehat{c}_1,\z)}\,\log x-(\log x)^2.$$
Lemma \ref{lemma:conditioned local CLT} then follows from the next result.\footnote{When applying the next result to prove Theorem \ref{thm:main2}, one picks $r$ small enough (depending only on $\bc_2$) such that $\widetilde{B}_{\bc_2}(x;r)\subseteq B_x$.}

\begin{lemma}[conditional local CLT]\label{lemma:conditioned local CLT2}Fix $L,r>0$. Uniformly for $\bla(x)=O((\log x)^L)$,
$$\p\Big(\bla(x)+\bS_{\wtx }\in \widetilde{B}_{\bc_2}(x;r)\mid (\bla(x)+\bS_{\wtx })\cdot\bc_2\in [x\bc_2\cdot\be_1-r,x\bc_2\cdot\be_1+r]\Big)\asymp x^{-\frac{d-1}{2}}.$$
\end{lemma}

\begin{proof}
 Let $\Lambda(\bla)=\log\phi_\bxi(\bla)=\log\E[e^{\bla\cdot\bxi}]$ be the log-moment generating function of $\bxi$. The measure $\q$ defined by 
\begin{align}
    \frac{\d\q}{\d\p}(\bx):=e^{\bc_2\cdot \bx-\Lambda(\bc_2)}\label{eq:dqdp}
\end{align}
satisfies that under $\q$, $\{\bxi_i\}_{i\in\N}$ are i.i.d.~with mean $(\widehat{c}_1,\z)\in\R^d$. In other words, under $\q$, the random walk $\{\widetilde{\bS}_n\}:=\{{\bS}_n-n(\widehat{c}_1,\z)\}$ is centered.

    By assumption (A5), the law of $\bxi$ under $\p$ is non-lattice. It follows by definition and triangle inequality that the law of $\bxi$ (and hence of its projection $\bxi\cdot\bc_2$) under $\q$ is also non-lattice. By the local CLT,
    \begin{align*}
        &\hspace{0.5cm}\p\Big((\bla(x)+\bS_{\wtx })\cdot\bc_2\in [x\bc_2\cdot\be_1-r,x\bc_2\cdot\be_1+r]\Big)\\
        &\asymp e^{\wtx \Lambda(\bc_2)-(x\be_1-\bla(x))\cdot\bc_2}\q\Big((\bla(x)+\bS_{\wtx })\cdot\bc_2\in [x\bc_2\cdot\be_1-r,x\bc_2\cdot\be_1+r]\Big)\\
        &\asymp e^{\wtx \Lambda(\bc_2)-(x\be_1-\bla(x))\cdot\bc_2}\\
        &\hspace{2cm}\times\q\Big(\Big|\widetilde{\bS}_{\wtx }\cdot\bc_2- \Big((\widehat{c}_1(\log x)^2-\frac{d+2}{2\partial_{x_1} \widehat{I}(\widehat{c}_1,\z)}\,\log x)\,\be_1\cdot\bc_2-\bla(x)\cdot\bc_2\Big)\Big|\leq r\Big)\\
        &\asymp e^{\wtx \Lambda(\bc_2)-(x\be_1-\bla(x))\cdot\bc_2}x^{-\frac{1}{2}}.
    \end{align*}
    Similarly, using the multi-dimensional local CLT,
     \begin{align*}
        \p\Big(\bla(x)+\bS_{\wtx }\in \widetilde{B}_{\bc_2}(x;r)\Big)     &\asymp e^{\wtx \Lambda(\bc_2)-(x\be_1-\bla(x))\cdot\bc_2}\q\Big(\bla(x)+\bS_{\wtx }\in\widetilde{B}_{\bc_2}(x;r)\Big)\\
        &\asymp e^{\wtx \Lambda(\bc_2)-(x\be_1-\bla(x))\cdot\bc_2}x^{-\frac{d}{2}}.
    \end{align*}
    This completes the proof.
\end{proof}
\end{appendix}

\section*{Acknowledgement}
  We thank Wei Cai, Amir Dembo, Yujin Kim, Oren Louidor, Bastien Mallein, Shaswat Mohanty, and Lenya Ryzhik for helpful discussions. We are also grateful to Haotian Gu and Tianqi Wu for their valuable feedback on a preliminary version of this paper. We also would like to thank two anonymous referees for greatly improving our presentation and suggesting Corollary \ref{coro}.   The material in this paper is based upon work supported by the Air Force Office of Scientific Research under award number FA9550-20-1-0397. Additional support is gratefully acknowledged from NSF 1915967, 2118199, 2229012, 2312204.


\bibliographystyle{plain} 
\bibliography{bibliography}       

\end{document}